\pdfoutput=1

\documentclass[
	bigMargin = false,
	font = palatino,
	final
	]{MyClass}

\usepackage{myMathHeader}
\NewDocumentCommand { \normalForm }{ m }{
	{#1}_{\textnormal{NF}}
}
\NewDocumentCommand { \normalFormGroup }{ m m }{
	{#2}_{{#1} \textnormal{-NF}}
}
\newcommand{\ex}{\textnormal{ex}}

\Title{Symplectic Reduction in Infinite Dimensions}

\Abstract{
	This paper develops a theory of symplectic reduction in the infinite-dimensional setting, covering both the regular and singular case. Extending the classical work of Marsden, Weinstein, Sjamaar and Lerman, we address challenges unique to infinite dimensions, such as the failure of the Darboux theorem and the absence of the Marle--Guillemin--Sternberg normal form.
	Our novel approach centers on a normal form of only the momentum map, for which we utilize new local normal form theorems for smooth equivariant maps in the infinite-dimensional setting.
	This normal form is then used to formulate the theory of singular symplectic reduction in infinite dimensions. We apply our results to important examples like the Yang--Mills equation and the Teichmüller space over a Riemann surface.
}



\Institution{itp}{
    Institut for Theoretical Physics,
    University of Leipzig, 
    04009 Leipzig, Germany}

\Institution{sjtu}{
    School of Mathematical Sciences and
Shanghai Frontier Science Center of Modern Analysis, 
    Shanghai Jiao Tong University,
    Shanghai 200240, China}
    
\Author[
    affiliation = sjtu,
    email = research@tobiasdiez.de,
	funding = {Partially supported by National Natural Science 
	Foundation of China grant 12350410357.}
    ]{diez}{Tobias Diez}

\Author[
    affiliation = itp,
    email = {rudolph@rz.uni-leipzig.de}
    ]{rudolph}{Gerd Rudolph}

\begin{document}

\vspace*{-5em}
\MakeTitle

\tableofcontents

\listoftodos

\section{Introduction}

Symplectic reduction is a fundamental concept in the study of Hamiltonian systems, with significant historical roots in finite-dimensional symplectic geometry. The pioneering work of \textcite{Meyer1973,MarsdenWeinstein1974} established a rigorous framework for reducing symplectic manifolds by the symmetries of a Lie group, a process now central to understanding many physical systems with symmetry. This reduction process simplifies the analysis of Hamiltonian systems by effectively lowering the dimensionality of the phase space while preserving the symplectic structure. 
\Textcite{SjamaarLerman1991} extended the theory of symplectic reduction to include singular cases, where the group action is not necessarily free, leading to a stratified symplectic space. 
A comprehensive picture of regular and singular symplectic reduction in finite dimensions, including symplectic orbit reduction, has been presented in the monograph by \textcite{OrtegaRatiu2003}.
A crucial technical tool in these developments is the normal form by \textcite{Marle1983,GuilleminSternberg1984},
which brings both the momentum map and the symplectic form into a canonical form.

From the outset, there was a strong interest in extending the theory of symplectic reduction to infinite dimensions, motivated by the need to understand the symmetries of field theories and the geometric structure of infinite-dimensional manifolds.
Many partial differential equations can be interpreted as Hamiltonian systems with constraints that take the form of momentum maps. For example, in Yang--Mills theory, the Gauss constraint can be understood as a momentum map condition \parencite{Arms1981,AtiyahBott1983}. Similarly, in general relativity, the constraints arising from the ADM formalism reflect a momentum map structure \parencite{ArnowittDeserMisner1959,ArmsMarsdenEtAl1982,ArmsFischerMarsden1975}. In the realm of Kähler geometry, Donaldson \parencite{Donaldson1999,Donaldson2003} has provided significant examples, where ideas from symplectic reduction play a key role in understanding geometric and topological properties of the spaces involved.

If one wishes to develop a rigorous theory of symplectic reduction in infinite dimensions one faces various serious problems: 
\begin{enumerate}
	\item 
		Even in the simplest setting of weakly symplectic Banach manifolds, by constructing a simple counterexample, \textcite{Marsden1972} showed that the Darboux theorem  fails.
		Hence, the classical approach to symplectic reduction via the Marle--Guillemin--Sternberg normal form (which relies on the Darboux theorem) is not applicable in this setting.
	\item
		In many examples, the symmetry group is a group of diffeomorphisms. The latter is not a Banach Lie group.
		Hence, the framework of symplectic Banach manifolds with actions of Banach Lie groups is not sufficient to cover these examples.
	\item 
		Often a classical momentum map does not exist.
\end{enumerate}

In this paper, we present a theory of symplectic reduction in infinite dimensions, including both the regular and singular case.
In order to overcome the challenges mentioned above, we will proceed as follows.
Instead of using the Darboux theorem to construct a normal form, we are exploiting the observation that it is not essential to bring the symplectic structure into a normal form. Instead, we focus primarily on the momentum map; using the symplectic structure only as a secondary tool. 
Thus, our strategy heavily rests on \parencite{DiezRudolphKuranishi}, where we have proven local normal form theorems for smooth equivariant maps between infinite-dimensional manifolds. The normal forms obtained in the latter paper are new even in finite dimensions. The proofs are inspired by the Lyapunov--Schmidt reduction for dynamical systems and by the Kuranishi method for moduli spaces and they heavily rest on a slice theorem for Fréchet manifolds proven in \parencite{DiezSlice}. The abstract moduli space obtained by factorizing a level set of the equivariant map with respect to the group action carries the structure of a Kuranishi space, that is, such moduli spaces are locally modeled on the quotient by a compact group of the zero set of a smooth map.
In \parencite{DiezRudolphKuranishi}, we have applied the general theory  to the moduli space of anti-self-dual instantons, the Seiberg--Witten moduli space and the moduli space of pseudoholomorphic curves.

In order to include groups of diffeomorphisms as symmetry groups, we will work in the setting of (tame) Fréchet manifolds. 
For basics on infinite-dimensional manifolds, we refer to \parencite{Hamilton1982,Neeb2006,GloecknerNeeb2013}.
Moreover, based on joint work of one of us (T.D.) with T. Ratiu \parencite{DiezRatiuAutomorphisms} on actions of diffeomorphism groups, we work in the general context of group-valued momentum maps thereby unifying several other notions of generalized momentum maps.
The group structure of the target of the momentum map allows to encode discrete topological information; a feature that is especially relevant to the action of geometric automorphism groups.
Such groups are sensitive to the topology of the spaces they live on and often do not admit a classical momentum map.
As a byproduct of our approach, we obtain an unified framework for reduction of a few generalized notions of momentum maps including, for example, circle-valued momentum maps 

The paper is organized as follows.
To set the stage, we start by recalling the foundational aspects of symplectic reduction in finite dimensions and discuss the challenges of the infinite-dimensional setting in more detail.
In the next step, we discuss some algebraic results of linear symplectic geometry needed for the symmetry reduction scheme. 
As a fundamental tool to handle weakly symplectic forms, we introduce and study a class of topologies associated with the symplectic form.
Next, we discuss symplectic manifolds and momentum maps in our infinite-dimensional setting.
Using the results of \parencite{DiezRudolphKuranishi} on normal forms of equivariant maps mentioned above, we establish a novel normal form result for momentum maps. Even in the finite-dimensional case, this normal form is different from the classical Marle--Guillemin--Sternberg normal form, see \cref{rem:momentumMap:MGSnormalForm:notDarboux,rem:momentumMap:MGSnormalForm:comparison} for details. This important technical tool then serves as the basis for our infinite-dimensional version of the Singular Symplectic Reduction Theorem (see \cref{prop:symplecticReduction:mainTheorem}).
As in finite dimensions, it states that the reduced phase space, obtained by factorizing a level set of the
momentum map with respect to the symmetry group, decomposes into smooth manifolds with each of them carrying a natural symplectic structure.
If a suitable approximation property holds, then this decomposition of the reduced phase space is a stratification. 
Moreover, the dynamics of the system reduces to the quotient and restricts to some Hamiltonian dynamics on each symplectic stratum.
As a direct consequence of the singular reduction theorem, we obtain the regular reduction theorem, which states that the reduced phase space is a smooth symplectic manifold if the action of the symmetry group is free, see \cref{prop:symplecticReduction:regular}.
Note that our approach is different from the classical one, where the regular reduction is usually discussed first and then used to derive the symplectic normal form for the singular reduction.
To our knowledge, our results on regular and singular symplectic reduction are new even for (weakly) symplectic Banach manifolds.
Finally, we apply our results to the example of symplectic reduction in the context of the Yang–Mills equation over a Riemannian surface.
This reworks and extends the classical treatment of \textcite{AtiyahBott1983}, and is based on joint work with J.~Huebschmann \parencite{DiezHuebschmann2017}.
Moreover, we study symplectic reduction in the context of the Teichmüller space of a Riemann surface.
The Teichmüller space can be realized as a symplectic orbit reduction; the general theory of symplectic orbit reduction in infinite dimensions will be dealt with in future work.
After restricting to a suitable subgroup, the symplectic point reduction yields a symplectic manifold that symplectically fibers over the Teichmüller space.
Some of the material presented in this work is based on the first author's thesis \parencite{DiezThesis}.

This paper is closely related to our paper \parencite{DiezRudolphReduction}, where we have developed a theory for symplectic cotangent bundle reduction in infinite dimensions in the Fr\'echet context and where we have applied the results obtained there to symplectic reduction of gauge theories with special emphasis on the Higgs sector of the 
Glashow--Weinberg--Salam model. For the latter model we have given a detailed description of the reduced phase space and we have shown that the singular structure is encoded in a finite-dimensional Lie group action. 
The results obtained in \parencite{DiezRudolphReduction} cannot be seen in a straightforward way as an application of  the general theory developed here. This is related to the fact that in the cotangent bundle case one has a secondary stratification structure and, thus, one has to 
construct a normal form which is compatible with that structure, see \cref{sec:cotangentBundleReduction} below.  

\section{Symplectic reduction in finite dimensions}
\label{sec:symplecticReductionFiniteDimensions}

In this section, we recall the foundational aspects of symplectic reduction within the finite-dimensional context. The exploration of principles and techniques in the classical finite-dimensional setting acts as a foundational step, paving the way to comprehend the complexities and extensions encountered in infinite-dimensional symplectic reduction.
Experts in the field may wish to skip this section and proceed directly to the next section, where the challenges arising from the passage to infinite dimensions are discussed.

Consider a finite-dimensional symplectic manifold \( (M, \omega) \) equipped with a symplectic action of a Lie group \( G \).
The action of \( G \) induces a corresponding infinitesimal action \( \xi \mapsto \xi^* \) of the Lie algebra \( \LieA{g} \) on \( M \).
The momentum map associated with this action is denoted by \( J: M \to \LieA{g}^* \) and it satisfies the momentum map equation \( \xi^* \contr \omega = - \dif \dualPair{J}{\xi} \) for all \( \xi \in \LieA{g} \), where the bracket represents the natural pairing between elements of \( \LieA{g}^* \) and \( \LieA{g} \).

The goal of singular symplectic reduction is to construct a reduced space that captures the essential geometric and dynamical features of the original system.
For every \( \mu \in \LieA{g}^* \), consider
\begin{equation}
	\check{M}_\mu \defeq J^{-1}(\mu) / G_\mu,
\end{equation}
where \( G_\mu \) denotes the stabilizer subgroup of \( \mu \) under the coadjoint action of \( G \) on \( \LieA{g}^* \).
In the regular case where \( G \) acts freely and \( \mu \) is a regular value, \textcite{MarsdenWeinstein1974,Meyer1973} showed that \( \check{M}_\mu \) is a smooth manifold carrying again a symplectic structure \( \check{\omega}_\mu \) uniquely determined by the relation \( \pi_\mu^* \check{\omega}_\mu = \restr{\omega}{J^{-1}(\mu)} \), where \( \pi_\mu: J^{-1}(\mu) \to \check{M}_\mu \) denotes the projection map.
The singular case has been studied by \textcite{SjamaarLerman1991,ArmsGotayEtAl1990} and involves a decomposition of \( \check{M}_\mu \) into symplectic manifolds, each associated with a specific orbit type.
Moreover, these symplectic manifolds are glued together in such a way that the Poisson brackets fit together to yield one global Poisson algebra, endowing \( \check{M}_\mu \) with a structure known as a stratified symplectic space.

The proof of singular symplectic reduction involves several key techniques, which we outline below:

\begin{enumerate}
    \item The \emph{Witt-Artin decomposition} provides an \( \omega \)-orthogonal splitting of the tangent space at each point in the symplectic manifold. This splitting can be seen as a linear version of the reduction procedure. The proof of the Witt-Artin decomposition relies mostly on the momentum map equation and on basic properties of symplectic orthogonal complements. 
    
    \item The proof that the \emph{decomposition of \( \check{M}_\mu \) into orbit types yields a stratification} uses a slice theorem.
    Since one needs to take into account the geometric structure of the momentum map level set, it is not enough to use the classical Slice Theorem of Palais for the action of a Lie group on a manifold. Instead, one needs to use the Symplectic Slice Theorem due to \textcite{Marle1983,Marle1985} and \textcite{GuilleminSternberg1984}.
	The proof of this theorem proceeds by using the equivariant Darboux Theorem to bring the symplectic form into a normal form, which then also allows to bring the momentum map into a normal form.
	Moreover, the construction of the symplectic slice uses regular symplectic reduction and thus relies on the Inverse Function Theorem as well as on the ordinary Slice Theorem of Palais.
	
    \item That each stratum of \( \check{M}_\mu \) inherits a \emph{symplectic form} is concluded by mainly using the Witt-Artin decomposition to show that the symplectic form is non-degenerate.
    
    \item For the \emph{gluing of the Poisson structures} on the reduced phase space one defines the algebra of smooth functions on the reduced space essentially as projections of \( G \)-invariant smooth functions on \( M \) (restricted to the level set \( J^{-1}(\mu) \)) and then shows that this algebra is closed under the Poisson bracket.
\end{enumerate}

As we will discuss in the next section, the passage to infinite dimensions comes with several difficulties which require new techniques and thus a different approach.

\section{Problems in infinite dimensions}

In field theory and global analysis, the maps under consideration are usually given by partial differential operators between spaces of sections and thus they give rise to smooth maps between appropriate Sobolev completions.
On the other hand, the symmetry action often involves compositions of maps and thus fails to be differentiable as a map between spaces of sections of a given Sobolev class.
For example, the group of diffeomorphisms of a fixed Sobolev regularity is a Banach manifold as well as a topological group but not a Lie group, because the group operation is not differentiable.
When working with smooth sections these problems disappear and the group of smooth diffeomorphisms is a bona fide Lie group modeled on a Fréchet space.
In order to include these important examples, we assume throughout this work that all infinite-dimensional manifolds are modeled on Fréchet or even more general locally convex spaces.
The approach via Fréchet spaces has also the advantage that certain geometric arguments are simpler, because one does not have to deal with issues originating in the low regularity of the geometric objects under study.\footnotemark{}\footnotetext{
The polyfold framework, as introduced by \textcite{HoferWysockiZehnder2017}, presents an alternative perspective on addressing the challenges related to differentiability.
It concentrates on a sequence of Banach spaces, in contrast to the Fréchet approach, which focuses on the limit space.
For a comprehensive analysis and a detailed comparison between these two approaches, we refer to \parencite{Gerstenberger2016}.}
While the generality and flexibility offered by Fréchet manifolds allow for a broader scope of applications, one must also reckon with the challenges that arise due to the absence of standard tools readily available for Banach manifolds.
In particular, in the Fréchet setting, due to \parencite{Hamilton1982}, the Inverse Function Theorem is available  and, while being a powerful tool, it is more intricate and harder to apply compared to the classical Banach counterpart.

When considering infinite-dimensional symplectic manifolds \( (M, \omega) \), it turns out that the requirement for the musical map \( \omega^\flat: \TBundle M \to \CotBundle M \) to be an isomorphism is often too strong.
This is certainly the case in the Fréchet setting since a proper Fréchet space is never isomorphic to its topological dual; but also in many Banach examples the musical map is not an isomorphism.
It turns out that a more suitable notion of non-degeneracy in the infinite-dimensional setting is that of a \emph{weakly symplectic form} which only requires \( \omega^\flat \) to be injective.
However, in the realm of weakly symplectic structures, significant challenges arise, for example, due to the disparity between the double symplectic orthogonal and the original subspace. This discrepancy poses significant complications when attempting to generalize the Bifurcation Lemma and the Witt-Artin decomposition.
These issues are discussed in more detail in \cref{sec:symplecticFunctionalAnalysis}, where we also provide a solution to these problems by introducing a new topology defined in terms of the symplectic form.

The construction of a symplectic normal form in the infinite-dimensional setting initially appears promising by leveraging Hamilton's Inverse Function Theorem and the infinite-dimensional Slice Theorem developed in \parencite{DiezSlice} to construct a symplectic slice similar to the finite-dimensional setting.
However, a fundamental paper by \textcite{Marsden1972} provided a striking counterexample of the Darboux Theorem for weakly symplectic Banach manifolds.
In addition, for Fréchet manifolds one cannot even hope to use Moser's trick due to the absence of a general solution theory of ODE's in this setting.
Thus, although one can perhaps construct a symplectic slice, it is no longer possible to establish that an open neighborhood of a given orbit is indeed symplectomorphic to this normal form.
This no-go result propels us to adopt a different strategy in establishing an infinite-dimensional version of the symplectic reduction theorem, deviating from the conventional proof employed in finite dimensions. Rather than focusing on achieving a normal form for the symplectic structure, our approach shifts the spotlight directly onto the momentum map, with the symplectic structure playing a secondary role in our derivation. By prioritizing the analysis of the momentum map, we can circumvent the need for a Darboux Theorem and instead employ techniques developed in the study of moduli spaces such as the theory of Kuranishi structures.
In particular, we heavily make use of the equivariant normal form results for arbitrary maps between Fréchet manifolds developed in \parencite{DiezRudolphKuranishi}. 

Finally, in applications, one often encounters situations where the momentum map does not exist.
This is especially the case for systems where the symmetry is encoded in the action of a group of diffeomorphisms.
To overcome this limitation, \parencite{DiezRatiuAutomorphisms} introduced the concept of group-valued momentum maps.
While a traditional momentum map takes values in the dual of the Lie algebra, a group-valued momentum maps has a Lie group as its target space. This shift allows for a broader representation of conserved quantities associated with symmetries, and covers all known examples of symplectic actions of diffeomorphism groups.
Luckily, it turns out that this generalization does not provide serious complications for the reduction theory.

\section{Normal form of equivariant maps}
\label{sec:normalFormEquivariantMap}

In our previous work \parencite{DiezRudolphKuranishi}, we have established local normal form theorems for smooth equivariant maps between infinite-dimensional manifolds.
In this paper, we will refine and adopt these results to the construction of normal forms for momentum maps.
For the convenience of the reader, we will briefly review the main results of \parencite{DiezRudolphKuranishi} and introduce the notation that will be used throughout this work.

In finite dimensions, there is essentially only one way to construct a normal form for a smooth map between manifolds, namely the method provided by the classical Inverse Function Theorem.
In infinite dimensions, however, the situation is more complicated since there are different versions of the Inverse Function Theorem available, each one with its own advantages and disadvantages.
To address this complexity, we adopt a modular approach.
We begin by providing a precise definition of what we mean by a normal form of an (equivariant) map.
Then we introduce various theorems which outline different methods for attaining such a normal form.
Each theorem corresponds to a specific version of the Inverse Function Theorem, accounting for its distinct properties and applicability.
However, it is important to note that all subsequent results rely solely on the abstract notion of a normal form.
This makes our framework flexible and allows to easily accommodate new approaches of bringing a map into normal form.

\begin{defn}
	\label{defn:normalFormMap:normalFormAbstract}
	An \emphDef{abstract normal form} consists of a tuple \( (X, Y, \hat{f}, f_\singularPart) \), where
	\begin{enumerate}
		\item
			\( X \) and \( Y \) are locally convex spaces with topological decompositions\footnotemark{} \( X = \ker \oplus \coimg \) and \( Y = \coker \oplus \img \),
			\footnotetext{In these decompositions \( \ker \), \( \coimg \), \etc denote abstract spaces. Below, we will identify them with the kernel, coimage, \etc of the tangent map \( \tangent_m f \), respectively.}
		\item
			\( \hat{f}: \coimg \to \img \) is a linear topological isomorphism,
		\item
			\( f_\singularPart: X \supseteq U \to \coker \) is a smooth map defined on an open neighborhood \( U \) of \( 0 \) in \( X \) such that \( f_\singularPart(0, x_2) = 0 \) holds for all \( x_2 \in U \intersect \coimg \) and such that the derivative \( \tangent_{(0, 0)} f_\singularPart: X \to \coker \) of \( f_\singularPart \) at \( (0, 0) \) vanishes.
		\end{enumerate}
		Given an abstract normal form \( (X, Y, \hat{f}, f_\singularPart) \), we set
		\begin{equation}
			\normalForm{f} \defeq \hat{f} + f_\singularPart: X \supseteq U \to Y
		\end{equation}
		and call it the associated \emphDef{local normal form map}.
			
		A normal form \( (X, Y, \hat{f}, f_\singularPart) \) is called \emphDef{tame} if \( X, Y \) are tame Fréchet spaces with tame decompositions \( X = \ker \oplus \coimg \), \( Y = \coker \oplus \img \) with \( \hat{f} \) being a tame isomorphism and \( f_\singularPart \) being a tame smooth map (see \parencite{Hamilton1982} for the definition of tameness).
\end{defn}
The \( 0 \)-level set of \( \normalForm{f} \),
\begin{equation}
	\label{eq:normalFormMap:zeroSetNF}
	\normalForm{f}^{-1}(0) = \set*{(x_1, 0) \in U \given f_\singularPart(x_1, 0) = 0},
\end{equation}
is in general not a smooth manifold, and its singular structure is completely determined by \( f_\singularPart \).
For this reason, we refer to \( f_\singularPart \) as the \emphDef{singular part} of \( \normalForm{f} \).

\begin{defn}
	\label{defn:normalFormMap:bringIntoNormalForm}
	We say that a smooth map \( f: M \to N \) between manifolds can \emphDef{be brought into the normal form} \( (X, Y, \hat{f}, f_\singularPart) \) at the point \( m \in M \) if there exist charts \( \kappa: M \supseteq U' \to U \subseteq X \) at \( m \) and \( \rho: N \supseteq V' \to V \subseteq Y \) at \( f(m) \) such that \( f(U') \subseteq V' \), \( f_\singularPart \) is defined on \( U \) and 
	\begin{equation}
		\label{eq:normalFormMap:bringIntoNormalForm:commutative}
		\rho \circ \restr{f}{U'} \circ \kappa^{-1} = \normalForm{f}
	\end{equation}
	holds true.
	In short, we say that \emphDef{\( f \) is locally equivalent to \( \normalForm{f} \)}.
\end{defn}
Assume that the smooth map \( f: M \to N \) can be brought into the normal form \( (X, Y, \hat{f}, f_\singularPart) \) at the point \( m \in M \) using charts \( \kappa: U' \to U \) and \( \rho: V' \to V \).
The isomorphisms \( \tangent_m \kappa: \TBundle_m M \to X \) and \( \tangent_{f(m)} \rho: \TBundle_{f(m)} N \to Y \) identify the abstract spaces \( X \) and \( Y \) with the tangent spaces of \( M \) and \( N \), respectively.
Under these identifications, the spaces \( \ker \) and \( \img \) in the decomposition of \( X \) and \( Y \) coincide with the kernel and the image of \( \tangent_m f \).
Note that if a map \( f \) can be brought into a normal form at \( m \), then it is a submersion (locally a projection) if and only if \( \tangent_m f \) is surjective.
Similarly, \( f \) is an immersion (locally an inclusion) if and only if \( \tangent_m f \) is injective.
Thus the upshot is that one obtains the submersion, regular value, immersion and constant rank theorem \etc practically for free as soon as one knows that the map under study can be brought into a normal form.
In \parencite{DiezRudolphKuranishi}, we have established various theorems that allow us to bring a map into a normal form, see Theorem 3.5 (Banach version), Theorem 3.11 (Nash--Moser version), Theorem 3.12 (elliptic version).

Moreover, in \parencite{DiezRudolphKuranishi}, we have shown that a wide class of equivariant maps can be brought into the following equivariant normal form.
\begin{defn}
	\label{def:normalFormEquivariantMap:abstractNormalFormEquivariantMap}
	An \emphDef{abstract equivariant normal form} is a tuple \( (H, X, Y, \hat{f}, f_\singularPart) \) consisting of a compact Lie group \( H \) and a normal form \( (X, Y, \hat{f}, f_\singularPart) \) which is \( H \)-equivariant in the sense that \( X, Y \) are endowed with smooth linear \( H \)-actions, the decompositions \( X = \ker \oplus \coimg \), \( Y = \coker \oplus \img \) are \( H \)-invariant, and the maps \( \hat{f}: \coimg \to \img \), \( f_\singularPart: X \supseteq U \to \coker \) are \( H \)-equivariant, where \( U \subseteq X \) is an \( H \)-invariant open neighborhood of \( 0 \).

	For an equivariant normal form \( (H, X, Y, \hat{f}, f_\singularPart) \) and a Lie group \( K \) with \( H \subseteq K \), we define the \emphDef{equivariant local normal form map} \( \normalFormGroup{K}{f}: K \times_{H} U \to K \times_{H} Y \) by
	\begin{equation}
		\normalFormGroup{K}{f} \bigl(\equivClass{k, x_1, x_2}\bigr) \defeq \equivClass[\big]{k, \normalForm{f}(x_1, x_2)} = \equivClass[\big]{k, \hat{f}(x_2) + f_\singularPart(x_1, x_2)},
	\end{equation}
	for \( k \in K \), \( x_1 \in U \intersect \ker \) and \( x_2 \in U \intersect \coimg \).
\end{defn}

\begin{defn}
	\label{def:normalFormEquivariantMap:normalFormEquivariantMap}
	Let \( f: M \to N \) be a smooth \( G \)-equivariant map.
	For \( m \in M \) write \( \mu = f(m) \in N \), and assume that the stabilizer \( G_\mu \) is a Lie subgroup of \( G \).
	We say that \( f \) can be \emphDef{brought into the equivariant normal form \( (H, X, Y, \hat{f}, f_\singularPart) \)} at \( m \) if \( H = G_m \) and there exist
	\begin{enumerate}
		\item 
			a slice\footnotemark{} \( S \) at \( m \) for the \( G_\mu \)-action,
			\footnotetext{See \cref{defn:slice:slice} for the definition of a slice.}
		\item 
			a \( G_m \)-equivariant diffeomorphism \( \iota_S: X \supseteq U \to S \subseteq M \), 
		\item 
			a \( G_m \)-equivariant chart \( \rho: N \supseteq V' \to V \subseteq Y \) at \( \mu \) with \( f(S) \subseteq V' \)
	\end{enumerate}
	such that the following diagram commutes:
	\begin{equationcd}[label=eq:normalFormEquivariantMap:bringIntoNormalForm]
		M \to[r, "f"] 
			& N
			\\
		\mathllap{G_\mu \times_{G_m} X \supseteq {}} G_\mu \times_{G_m} U
			\to[r, "\normalFormGroup{G_\mu}{f}"]
			\to[u, "{\chi^\tube} \, \circ \, {(\id_{G_\mu} \times \iota_S)}"]
			& G_\mu \times_{G_m} V \mathrlap{{}\subseteq G_\mu \times_{G_m} Y,}
			\to[u, "\rho^\tube" swap] 
	\end{equationcd}
	where \( \chi^\tube: G_\mu \times_{G_m} S \to M, \chi^\tube(\equivClass{g, s}) = g \cdot s \) is the tube map associated with \( S \) and \( \rho^\tube: G_\mu \times_{G_m} V \to N \) is defined by \( \rho^\tube(\equivClass{g, y}) = g \cdot \rho^{-1}(y) \).
	For short, we say that \( f \) is \emphDef{locally \( G \)-equivalent} to \( \normalFormGroup{G_\mu}{f} \).
\end{defn}
Assume that the \( G \)-equivariant map \( f: M \to N \) can be brought into the equivariant normal form \( (G_m, X, Y, \hat{f}, f_\singularPart) \) by using diffeomorphisms \( \iota_S \) and \( \rho \) as above.
Then \( \tangent_m \iota_S: X \to \TBundle_m S \) identifies \( X \) with \( \TBundle_m S \isomorph \TBundle_m M \slash (\LieA{g}_\mu \ldot m) \), as \( G_m \)-representation spaces.
Similarly, \( \tangent_\mu \rho: \TBundle_\mu N \to Y \) is an isomorphism of \( G_m \)-representations.
Under these identifications, the abstract spaces \( \ker \) and \( \coker \) occurring in the decomposition of \( X \) and \( Y \) are identified with the first and second homology groups of the chain complex
\begin{equationcd}[label=eq:normalFormEquivariantMap:deformationChain]
	0 \to[r] & \LieA{g}_\mu \to[r, "\tangent_e \Upsilon_m"] & \TBundle_m M \to[r, "\tangent_m f"] & \TBundle_\mu N \to[r] & 0 \,,
\end{equationcd}
where \( \Upsilon_m: G_\mu \to M \) is the orbit map of the \( G_\mu \)-action at \( m \in M \).

By combining a slice for the \( G_\mu \)-action with a normal form of the restriction of \( f \) to that slice, we obtain an equivariant local normal form.
\begin{thm}[{General Equivariant normal form, \parencite[Theorem~4.6]{DiezRudolphKuranishi}}]
	\label{prop:normalFormEquivariantMap:abstract}
	Let \( f: M \to N \) be an equivariant map between Fréchet \( G \)-manifolds.
	Choose \( m \in M \) and set \( \mu = f(m) \).
	Assume that the following conditions hold:
	\begin{thmenumerate}
		\item
			The stabilizer subgroup \( G_\mu \) of \( \mu \) is a Lie subgroup of \( G \).
		\item 
			The induced \( G_\mu \)-action on \( M \) is proper and admits a slice \( S \) at \( m \).
		\item
			The induced \( G_m \)-actions \( N \) can be linearized\footnotemark{} at \( \mu \).
			\footnotetext{A \( G \)-action can be linearized at a fixed point \(  m \in M \) if there exist a \( G \)-invariant open neighborhood \( U' \) of \( m \) in \( M \), a chart \( \rho: M \supseteq U' \to U \subseteq X \) and a smooth linear \( G \)-action on \( X \) such that \( \rho \) is \( G \)-equivariant. By Bochner's Linearization Theorem \parencites[Theorem~1]{Bochner1945}, every action of a compact Lie group on a finite-dimensional manifold can be linearized near a fixed point of the action.}
		\item
			{
			\providecommand{\creflastconjunction}{}
			\renewcommand{\creflastconjunction}{~or~}
			The restriction \( f^S \equiv \restr{f}{S}: S \to N \) of \( f \) to \( S \) can be brought into a normal form at \( m \) using \( G_m \)-equivariant charts.
			}
	\end{thmenumerate}
	Then, \( f \) can be brought into an equivariant normal form at \( m \) relative to the \( G \)-action.
\end{thm}

According to \textcite{Palais1961}, proper actions of finite-dimensional Lie groups on finite-dimensional manifolds always admit slices.
In infinite dimensions, this may no longer be true and additional assumptions have to be made.
We refer to \parencite{DiezSlice,Subramaniam1986} for general slice theorems in infinite dimensions and \parencite{AbbatiCirelliEtAl1989,Ebin1970,CerveraMascaroEtAl1991} for constructions of slices in concrete examples.
All assumptions of \cref{prop:normalFormEquivariantMap:abstract} are automatically satisfied in finite dimensions, and we obtain the following.
\begin{thm}[{Equivariant normal form --- finite dimensions, \parencite[Theorem~4.8]{DiezRudolphKuranishi}}]
	\label{prop:normalFormEquivariantMap:finiteDim}
	Let \( G \) be a finite-dimensional Lie group and let \( f: M \to N \) be a smooth \( G \)-equivariant map between finite-dimensional \( G \)-manifolds.
	If the \( G \)-action on \( M \) is proper, then \( f \) can be brought into an equivariant normal form at every point.
\end{thm}

Using the Nash--Moser Inverse Function Theorem, we get a version of \cref{prop:normalFormEquivariantMap:abstract} in the tame Fréchet category. 
\begin{thm}[{Equivariant normal form --- Tame Fréchet, \parencite[Theorem~4.11]{DiezRudolphKuranishi}}]
	\label{prop:normalFormEquivariantMap:tame}
	Let \( G \) be a tame Fréchet Lie group and let \( f: M \to N \) be an equivariant map between tame Fréchet \( G \)-manifolds.
	Choose \( m \in M \) and set \( \mu = f(m) \).
	Assume that the following conditions hold:
	\begin{thmenumerate}
		\item
			The stabilizer subgroup \( G_\mu \) of \( \mu \) is a tame Fréchet Lie subgroup of \( G \).
		\item 
			The induced \( G_\mu \)-action on \( M \) is proper and admits a tame slice \( S \) at \( m \).
		\item
			The induced \( G_m \)-action on \( N \) can be linearized at \( \mu \).
		\item
			The chain
			\begin{equationcd}[label=eq:normalFormEquivariantMap:tame:chain]
				0 \to[r] 
					& \LieA{g}_\mu
						\to[r]
					& \TBundle_s M
						\to[r, "\tangent_s f"]
					& \TBundle_{f(s)} N
						\to[r]
					& 0
			\end{equationcd}
			of linear maps parametrized by \( s \in S \) is uniformly\footnotemark{} tame regular at \( m \).
			Here, the first map is the Lie algebra action given by \( \xi \mapsto \xi \ldot s \) for \( \xi \in \LieA{g}_\mu \).
			\footnotetext{Roughly speaking, a family of operators is uniformly tame regular if a certain family of extended operators is invertible and the inverses form a tame family again, see \parencite[Definition~2.1 and~2.13]{DiezRudolphKuranishi} for the precise definition. For example, the chain~\eqref{eq:normalFormEquivariantMap:tame:chain} is uniformly tame regular if it is an elliptic complex at the point \( s = m \), see \parencite[Theorem~2.15]{DiezRudolphKuranishi}.}
	\end{thmenumerate}
	Then, \( f \) can be brought into an equivariant normal form at \( m \). 
\end{thm}

In particular, this theorem applies to nonlinear differential operators that are equivariant under groups of diffeomorphisms or gauge transformations.
To state the theorem, we need the notion of a geometric Fréchet manifold.
A tame Fréchet manifold \( M \) is said to be \emphDef{geometric} if it is locally modeled on the space of smooth sections of some vector bundle over a compact manifold (or a subspace thereof).
Moreover, roughly speaking, a smooth map \( f: M \to N \) between geometric Fréchet manifolds is called \emphDef{geometric} if its linearization \( \tangent_m f \) is a family of linear differential operators whose coefficients depend tamely on \( m \in M \),
see \parencite[Definition~3.13]{DiezRudolphKuranishi} for the precise definition.

\begin{thm}[Equivariant normal form --- elliptic]
	\label{prop:normalFormEquivariantMap:elliptic}
	Let \( G \) be a tame Fréchet Lie group and let \( f: M \to N \) be an equivariant map between tame Fréchet \( G \)-manifolds.
	Let \( m \in M \) and \( \mu = f(m) \).
	Assume that the following conditions hold:
	\begin{thmenumerate}
		\item
			The stabilizer subgroup \( G_\mu \) of \( \mu \) is a geometric tame Fréchet Lie subgroup of \( G \).
		\item 
			The induced \( G_\mu \)-action on \( M \) is proper and admits a geometric slice \( S \) at \( m \).
		\item
			The induced \( G_m \)-action on \( N \) can be linearized at \( \mu \).
		\item
			The chain
			\begin{equationcd}[label=eq:normalFormEquivariantMap:elliptic:chain]
				0 \to[r] 
					& \LieA{g}_\mu
						\to[r]
					& \TBundle_s M
						\to[r, "\tangent_s f"]
					& \TBundle_{f(s)} N
						\to[r]
					& 0
			\end{equationcd}
			is a chain of geometric linear maps parametrized by \( s \in S \), which is an elliptic complex at \( m \).
	\end{thmenumerate}
	Then, \( f \) can be brought into an equivariant normal form at \( m \).
\end{thm}

\section{Symplectic functional analysis}
\label{sec:symplecticFunctionalAnalysis}

In this section, we study infinite-dimensional symplectic vector spaces.
Specifically, we investigate symplectic structures defined on locally convex topological vector spaces.
In the Banach space setting, the theory of symplectic geometry splits into two branches depending on whether the map \( \omega^\flat: X \to X' \) induced by the symplectic form \( \omega \) on \( X \) is an isomorphism or merely an injection.
The former forms are called strongly symplectic and the latter are referred to as weakly symplectic. 
The well-consolidated building of finite-dimensional symplectic geometry generalizes almost without changes to strongly symplectic structures, but it is confronted with serious problems if weakly symplectic forms are considered.
For example, the Darboux theorem holds for strongly symplectic forms but fails for weak ones, see \parencite{Marsden1972}.

Many examples of symplectic structures on Banach spaces are only weakly symplectic.
Moreover, in the case of a Fréchet space, a \( 2 \)-form cannot be strongly symplectic, because the dual of a Fréchet space is never a Fréchet space (except when it is a Banach space).
Thus, weakly symplectic forms are the norm rather than the exception and we need to find a way to address the problems that originate from the failure of \( \omega^\flat \) to be surjective.
At the root of our approach lies the observation that the symplectic form induces a natural topology on \( X \), which we call a symplectic topology. 
For a strongly symplectic form, this topology is equivalent to the original vector space topology; while for a weakly symplectic form the original topology is finer.
Relative to the symplectic topology, the symplectic form behaves as if it were strongly symplectic, enabling the extension of standard finite-dimensional results to this context.
Functional analysis tools, particularly the theory of dual pairs, then allow to transfer these conclusions to results relative to the original topology.
From a different angle, we translate algebraic questions into a functional analytic setting and then utilize the powerful tool of dual pairs to efficiently solve these problems.  
We are not aware of any previous systematic study of infinite-dimensional linear symplectic geometry, though some elements can be found in \parencite{Booss-BavnbekZhu2014,ChernoffMarsden1974}. 
We refer the reader to \parencite{Koethe1983,Schaefer1971,Jarchow1981,NariciBeckenstein2010,RobertsonRobertson1973} for the relevant material concerning dual pairs.

\begin{defn}
	\label{defn::symplecticFunctionalAnalysis:definitionSymplectic}
	A \emphDef{symplectic vector space} \( (X, \omega) \) is a locally convex space \( X \) endowed with a jointly continuous, antisymmetric bilinear form \( \omega \) which is non-degenerate in the sense that the induced map 
	\begin{equation}
		\omega^\flat: X \to X', \quad x \mapsto \omega(x, \cdot)
	\end{equation}
	is injective, where \( X' \) denotes the topological dual of \( X \).
	If \( \omega^\flat \) is a bijection, then \( \omega \) is called \emphDef{strongly symplectic}.
\end{defn}

\begin{example}
	\label{ex::diffOneFormsOnSurface:symplecticStructure}
	Consider the Fréchet space \( X = \DiffFormSpace^1(M) \) of differential one-forms on a closed two-dimensional surface \( M \).
	The integration pairing yields a symplectic form \( \omega \) on \( \DiffFormSpace^1(M) \) by setting
	\begin{equation}
		\omega(\alpha, \beta) = \int_M \alpha \wedge \beta
	\end{equation}
	for \( \alpha, \beta \in \DiffFormSpace^1(M) \).
	Indeed, given a Riemannian metric \( g \) on \( M \) with associated Hodge star operator \( \hodgeStar \), the value
	\begin{equation}
		\omega(\alpha, \hodgeStar \alpha) = \int_M \norm{\alpha}^2_g
	\end{equation}
	 is positive for every non-vanishing \( \alpha \in \DiffFormSpace^1(M) \) and thus \( \omega \) is non-degenerate.  
	Moreover, the symplectic structure \( \omega \) is not strongly non-degenerate, because the image of \( \omega^\flat \) consists of regular functionals and not of all distributional \( 1 \)-forms. 

	This example will reappear throughout the next sections serving as an illustration of the abstract theory.
	The series consists of \cref{ex::diffOneFormsOnSurface:diracDistribution,ex::diffOneFormsOnSurface:generalizationToConnections,ex::diffOneFormsOnSurface:pointedGaugeTransf,ex::diffOneFormsOnSurface:momentumMapNormalForm,ex::diffOneFormsOnSurface:gaugeAction,ex::diffOneFormsOnSurface:symplecticReduction}.
\end{example}

If we ignore for a moment that \( X \) already carries a locally convex topology, we are left with a vector space endowed with a bilinear form \( \omega \). 
This setting is well-studied in functional analysis, where such a pair \( (X, \omega) \) is called a \emphDef{dual pair}, see \eg \parencite{Koethe1983,Schaefer1971,Jarchow1981} for background information.
The bilinear form \( \omega \) singles out certain topologies \( \tau \) on \( X \) for which \( \omega^\flat: X \to (X, \tau)' \) is surjective, \ie, every \( \tau \)-continuous functional on \( X \) is of the form \( \omega(x, \cdot) \) for some \( x \in X \).
In the general theory such topologies are called compatible with the dual pair \( (X, \omega) \).
Here and in the following, we will use the notation \( (X, \tau)' \) for the space of all \( \tau \)-continuous functionals on \( X \) to emphasize the dependence on the topology \( \tau \) of \( X \).
\begin{defn}
	Let \( (X, \omega) \) be a symplectic vector space.
	A locally convex topology \( \tau \) on \( X \) is called \emphDef{compatible} with \( \omega \) or a \emphDef{symplectic topology}\footnotemark{} if \( (X, \tau)' = \img \omega^\flat \).
	\footnotetext{We hope this naming will not lead to confusion with the discipline of \textquote{symplectic topology}, which is concerned with the study of global properties of finite-dimensional symplectic manifolds.}  
\end{defn}
The Mackey--Arens Theorem \parencite[Theorem~IV.3.3]{Schaefer1971}, applied to the symplectic form,
shows that there is always a symplectic topology on a given symplectic vector space.
Moreover, it yields an upper and a lower bound for the symplectic topologies (namely, the symplectic version of the weak and Mackey topology).
In order to distinguish symplectic topologies from the original one in which \( \omega \) is jointly continuous, we call the latter the \emphDef{original topology}.
Note that, for a weakly symplectic form, any symplectic topology is strictly coarser than the original topology, because the latter has more continuous functionals.
\begin{prop}
	\label{prop::symplecticFunctionalAnalysis:strongSymplecticOnlyInNormed}
	Let \( (X, \omega) \) be a symplectic vector space which is complete as a topological vector space.
	The original topology on \( X \) is symplectic if and only if \( X \) is a Banach space and \( \omega \) is strongly symplectic.
\end{prop}
\begin{proof}
	For a strongly symplectic form, the original topology is clearly symplectic.
	Conversely, suppose that the original topology is symplectic.
	Then, by definition, \( \omega^\flat: X \to X' \) is a bijection and it remains to show that \( X \) is normable.
	For this purpose, endow \( X' \) with the topology \( \tau_\omega \) by declaring \( \omega^\flat \) to be a homeomorphism.
	Accordingly, the inverse \( \omega^\sharp: X' \to X \) of \( \omega^\flat \) is a continuous linear map with respect to the topology \( \tau_\omega \) on \( X' \) and the original topology on \( X \).
	Note that the canonical evaluation map \( X' \times X \to \R \) can be written as
	\begin{equation}
		(\alpha, x) \mapsto \alpha(x) = \omega (\omega^\sharp(\alpha), x).
	\end{equation}
	Thus, it is jointly continuous with respect to \( \tau_\omega \) on \( X' \) and the original topology on \( X \).
	However, the evaluation map is only jointly continuous (for some vector space topology on \( X' \)) if \( X \) is normable according to \parencite{Maissen1963}.
\end{proof}
In other words, the difference between the original topology and any symplectic topology is a measure of how much \( \omega^\flat: X \to X' \) fails to be surjective.

By definition, symplectic topologies are closely tied to the symplectic form and thus they are often able to detect symplectic phenomena, which are hard to describe in the original topology.
As we will see now, problems involving subspaces of symplectic spaces can be conveniently dealt within the framework of symplectic topologies.
A first hint of this interplay can be gathered from the fact that all symplectic topologies have the same closed linear subspaces according to \parencite[Corollary~8.3.6 and Theorem~8.4.1]{Wilansky1978}.
This observation allows us to introduce the following notion.
\begin{defn}
	Let \( (X, \omega) \) be a symplectic vector space.
	A linear subspace \( V \subseteq X \) is called \emphDef{symplectically closed}, or simply \( \omega \)-closed, if it is closed with respect to some (and hence all) symplectic topology on \( X \).
	In a similar vein, \( V \) is said to be \emphDef{symplectically dense} if it is dense relative to any symplectic topology.
\end{defn}
Let \( (X, \omega) \) be a symplectic vector space and let \( V \subseteq X \) be a linear subspace.
The \emphDef{symplectic orthogonal} of \( V \) is defined by 
\begin{equation}
	V^\omega \defeq \set{x \in X \given \omega(x, v) = 0 \text{ for all } v \in V }.
\end{equation}
The following topological description of the double orthogonal generalizes previous results of \textcite[Lemma~1.4]{Booss-BavnbekZhu2014}.
\begin{prop}
	\label{prop:symplecticFunctionalAnalysis:symplectiallyClosedIffDoubleOrthogonal}
	Let \( (X, \omega) \) be a symplectic vector space.
	For every linear subspace \( V \subseteq X \), the symplectic double orthogonal \( V^{\omega \omega} \) coincides with the closure of \( V \) with respect to any symplectic topology.
	In particular, \( V \) is symplectically closed if and only if \( V^{\omega \omega} = V \).
\end{prop}
\begin{proof}
	Note that \( V^\omega \) is the symplectic counterpart of the polar (or annihilator) of \( V \) in the general theory of dual pairs, see \eg \parencite[Section~II.8.2]{Jarchow1981}.
	Hence the claim follows directly from the Bipolar Theorem \parencite[Proposition~20.3.2]{Koethe1983}.
\end{proof}
This results in a shift of perspective as we may complement the algebraic approach to symplectic double orthogonals by powerful tools from topology.
Illustrating this shift in philosophy, the following result is almost trivial from a topological point of view but not so straightforward to prove in an algebraic fashion.
\begin{lemma}
	\label{prop:symplecticFunctionalAnalysis:finiteDimSubspaceSymplectiallyClosed}
	Every finite-dimensional subspace of a symplectic vector space is symplectically closed.
\end{lemma}
\begin{proof}
	Every finite-dimensional vector space \( V \subseteq X \) is closed for whatever vector space topology we put on \( X \), see \parencite[Proposition~15.5.2]{Koethe1983}.
\end{proof}
Often it is convenient to relate the symplectic orthogonal to a metric orthogonal.
\begin{lemma}
	\label{prop:symplecticFunctionalAnalysis:orthogonalToSymplecticOrthogonal}
	Let \( (X, \omega) \) be a symplectic vector space with compatible complex structure \( j \) and associated inner product \( g = \omega(\cdot, j \cdot) \).
	For every linear subspace \( V \subseteq X \), the symplectic double orthogonal \( V^{\omega \omega} \) coincides with the double orthogonal \( V^{\perp \perp} \) of \( V \) with respect to the inner product \( g \).
\end{lemma}
\begin{proof}
	The relation \( g = \omega(\cdot, j \cdot) \) implies that \( j \) is the adjoint of the identity map \( \id: X \to X \) with respect to the dual pairs \( \omega \) and \( g \).
	Hence, the identity map is a topological isomorphism between the weak topologies of \( (X, \omega) \) and \( (X, g) \).
	In particular, weak closures (and hence double orthogonals) coincide.

	Alternatively, the claim follows from the fact that \( V^\omega = (j V)^\perp = j V^\perp \).
\end{proof}

In the literature, it is often assumed or even \enquote{proven} that in weakly symplectic Banach spaces every linear subspace \( V \) which is closed with respect to the original topology on \( X \) satisfies \( V^{\omega\omega} = V \), see for example \parencite[Lemma~7.5.9]{Kobayashi1987} or \parencite[Lemma~3.2]{Bambusi1999}.
Extending a counterexample of \textcite[Example~1.6]{Booss-BavnbekZhu2014} we can, however, show that every genuinely weakly symplectic space has at least one closed subspace that is not symplectically closed.
The following result is analogous to the fact that the dual of a non-reflexive Banach space contains subspaces that are both norm-closed and weak*-dense, see \parencite[Fact~4.1.6]{DeVito1978}.
\begin{prop}
	\label{prop:sympleticFunctionalAnalysis:allClosedAreSymplecticallyClosedImpliesStrong}
 	Let \( (X, \omega) \) be a symplectic vector space.
 	If \( \omega \) is not strongly symplectic, then there exists a closed proper linear subspace that is symplectically dense.
 	In particular, every closed linear subspace of \( X \) is symplectically closed if and only if \( X \) is normable and \( \omega \) is strongly symplectic.
\end{prop}
\begin{proof}
 	Endow \( X' \) with the weak topology determined by the dual pair \( \dualPairR{X'}{X} \), see \eg \parencite[Section~IV.3]{Schaefer1971}.
 	Suppose that \( \omega \) is not strongly non-degenerate.
 	Then, there exists a non-zero continuous linear functional \( \alpha \in X' \setminus \img \omega^\flat \).
 	The subspace \( A \defeq \R \cdot \alpha \subseteq X' \) is finite-dimensional and hence closed, see \parencite[Proposition~15.5.2]{Koethe1983}.
 	Consider the subspace \( V \defeq A^\polar \subseteq X \), where the polar \( A^\polar \) is taken with respect to the dual pair \( \dualPairR{X'}{X} \).
 	By definition, \( V \) coincides with the kernel of \( \alpha \) and thus it is a proper closed subspace of \( X \).
 	The Bipolar Theorem \parencite[Proposition~20.3.2]{Koethe1983} with respect to the dual pair \( \dualPairR{X'}{X} \) and closedness of \( A \) imply \( V^\polar = A^{\bipolar} = A = \R \cdot \alpha \).
 	Since \( \alpha \) is not contained in the image of \( \omega^\flat \), we have \( (\omega^\flat)^{-1}(V^\polar) = \set{0} \).
 	We can view \( \omega^\flat: X \to X' \) as a weakly continuous map between the dual pairs \( \omega(X, X) \) and \( \dualPairR{X'}{X} \) with the identity on \( X \) as the adjoint map.
 	Hence, \parencite[Proposition~IV.2.3]{Schaefer1971} implies 
 	\begin{equation}
 		V^\omega = (\id (V))^\omega = (\omega^\flat)^{-1}(V^\polar) = \set{0}.
 	\end{equation}
 	Thus, \( V^{\omega\omega} = X \).
 	In summary, every genuinely weakly symplectic space has at least one closed proper subspace whose symplectic closure is the whole space.
 	In other words, if every closed subspace is symplectically closed then the symplectic form has to be strongly symplectic.
 	The latter is only possible if \( X \) is normable according to \cref{prop::symplecticFunctionalAnalysis:strongSymplecticOnlyInNormed}.
 	Conversely, if \( (X, \omega) \) is a strongly symplectic space, then the original topology on \( X \) is symplectic and thus every closed subspace is also symplectically closed as a consequence of \parencite[Corollary~8.3.6 and Theorem~8.4.1]{Wilansky1978}. 
\end{proof}
\begin{example}
	\label{ex::diffOneFormsOnSurface:diracDistribution}
	Continuing \cref{ex::diffOneFormsOnSurface:symplecticStructure}.
	Every tangent vector \( X_m \in \TBundle_m M \) yields a continuous Dirac-like functional \( \delta_{X_m}: \DiffFormSpace^1(M) \to \R \) by evaluation of a \( 1 \)-form on \( X_m \).
	Note that \( \delta_{X_m} \) is singular in the sense of distributions and thus does \emph{not} lie in the image of \( \omega^\flat \). 
	Consider the closed subspace
	\begin{equation}
		V \defeq (\vspan \delta_{X_m})^\polar = \set{\alpha \in \DiffFormSpace^1(M) \given \alpha (X_m) = 0}.
	\end{equation}
	Since integration is not sensitive to the behavior at a single point, we find \( V^\omega = \set{0} \) and thus \( V^{\omega\omega} = \DiffFormSpace^1(M) \).
	Hence, by \cref{prop:symplecticFunctionalAnalysis:symplectiallyClosedIffDoubleOrthogonal}, the closure of \( V \) with respect to any symplectic topology is the whole space \( \DiffFormSpace^1(M) \).
	In summary, \( V \) is a closed but symplectically dense subspace.
\end{example}

In their original article on symplectic reduction, \textcite{MarsdenWeinstein1974} considered symplectic forms \( \omega \) on reflexive Banach spaces whose associated musical isomorphism \( \omega^\flat \) had a closed image.
This setting, a priori, lies between weakly and strongly symplectic Banach spaces.
However, under these assumptions, they showed that every closed subspace is also symplectically closed, see \parencite[Lemma on p.~123]{MarsdenWeinstein1974}.
Hence, \cref{prop:sympleticFunctionalAnalysis:allClosedAreSymplecticallyClosedImpliesStrong} implies that such symplectic forms are automatically strongly non-degenerate.
Let us record this observation.
\begin{prop}
	Let \( X \) be a reflexive Banach space endowed with a symplectic form \( \omega \).
	If \( \omega^\flat \) has closed image in \( X' \) (relative to the dual norm topology), then \( \omega \) is strongly symplectic.
\end{prop}

The restriction \( \omega_V \) of the symplectic structure \( \omega \) to a closed subspace \( V \subseteq X \) is in general degenerate, the kernel of \( \omega_V^\flat: V \to V' \) being \( V \intersect V^\omega \).
\begin{defn}
	\label{defn:sympleticFunctionalAnalysis:symplecticSubspace}
	A closed subspace \( V \) of a symplectic vector space \( (X, \omega) \) is called \emphDef{symplectic} if \( V \intersect V^\omega = \set{0} \).
\end{defn}
Accordingly, the restriction \( \omega_V \) of \( \omega \) to a symplectic subspace \( V \) yields a symplectic form on \( V \).
We emphasize that the notions \enquote{symplectically closed subspace} and \enquote{symplectic subspace} should not be confused.
\begin{example}
	Let \( (X, \omega) \) be a symplectic vector space.
	Assume that \( \omega \) is not strongly symplectic and consider a closed, symplectically dense, proper subspace \( V \subset X \) (which always exists according to \cref{prop:sympleticFunctionalAnalysis:allClosedAreSymplecticallyClosedImpliesStrong}).
	For every \( x \in V^\omega \), the functional \( \omega(x, \cdot) \) on \( X \) is continuous with respect to any symplectic topology on \( X \) but it also vanishes on the symplectically dense subspace \( V \).
	Thus, it has to vanish on the whole space \( X \), which implies \( x = 0 \).
	Hence, \( V^\omega = \set{0} \) and so \( V \intersect V^\omega = \set{0} \).
	In other words, every closed, symplectically dense, proper subspace is a symplectic subspace.
\end{example}
In finite dimensions, every symplectic subspace \( V \subseteq X \) induces a direct sum decomposition \( X = V \oplus V^\omega \).
The previous example shows that this is no longer the case in infinite dimensions (for a proper symplectically dense subspace \( V \) we have \( V^\omega = \set{0} \), and thus the subspace \( V \oplus V^\omega = V \) is a proper subspace of \( X \)).
The following phenomenon, peculiar for the infinite-dimensional setting, is related.
Given a subspace \( V \subseteq X \) and \( v \in V \), the functional \( \omega(v, \cdot) \) on \( X \) is continuous with respect to any symplectic topology \( \tau_\omega \) on \( X \).
Hence, its restriction to a functional on \( V \) is continuous relative to the subspace topology, the latter being denoted by \( \tau_\omega \) as well. 
This furnishes a linear map
\begin{equation}
	\label{eq:sympleticFunctionalAnalysis:symplecticSubspaceGammaMap}
	\Gamma_V: V \mapsto (V, \tau_\omega)', \qquad v \mapsto \omega(v, \cdot).
\end{equation}
Clearly, \( \Gamma_V \) is injective if and only if \( V \) is a symplectic subspace.
In finite dimensions, a count of dimensions shows that \( \Gamma_V \) is bijective if \( V \) is a symplectic subspace.
However, when passing to the infinite-dimensional setting, the above example of a symplectically dense subspace \( V \subseteq X \) shows that \( \Gamma_V \) is in general not surjective (for every non-zero \( y \in X \setminus V \), the \( \tau_\omega \)-continuous functional \( \omega(y, \cdot) \) on \(  V \) is not representable by some \( v \in V \)).

\begin{prop}
	\label{prop:symplecticFunctionalAnalysis:strongSymplecticSubspace}
	Let \( (X, \omega) \) be a symplectic vector space.
	For a closed subspace \( V \subseteq X \), the following are equivalent:
	\begin{enumerate}
		\item
			\label{prop:symplecticFunctionalAnalysis:strongSymplecticSubspace:restriction}
			\( V \) is a symplectic subspace and the restriction of every symplectic topology on \( X \) yields a symplectic topology on \( (V, \omega_V) \).
		\item
			\label{prop:symplecticFunctionalAnalysis:strongSymplecticSubspace:directSum}
			There exists an algebraic direct sum decomposition \( X = V \oplus V^\omega \).
		\item
			\label{prop:symplecticFunctionalAnalysis:strongSymplecticSubspace:surjectiveGamma}
			The linear map \( \Gamma_V \) defined in~\eqref{eq:sympleticFunctionalAnalysis:symplecticSubspaceGammaMap} is surjective.
			\qedhere
	\end{enumerate}
\end{prop}
\begin{proof}
	The implication \iref{prop:symplecticFunctionalAnalysis:strongSymplecticSubspace:restriction} \textrightarrow{} \iref{prop:symplecticFunctionalAnalysis:strongSymplecticSubspace:surjectiveGamma} is clear.
	Now assume that \( \Gamma_V \) is surjective.
	Then, we have
	\begin{equation}\begin{split}
		V \intersect V^\omega
			&= \set{x \in V \given \Gamma_V(v)(x) = \omega(v, x) = 0 \text{ for all } v \in V}
			\\
			&= \set{x \in V \given \alpha(x) = 0 \text{ for all } \alpha \in (V, \tau_\omega)'}
			\\
			&= \set{0},
	\end{split}\end{equation}
	where the last equality is a consequence of the Hahn--Banach Theorem \parencite[Proposition~20.1.2]{Koethe1983}, which implies that \( \tau_\omega \)-continuous functionals on \( V \) separate points of \( V \); that is, for each non-zero \( x \in V \) there exists \( \alpha \in (V, \tau_\omega)' \) such that \( \alpha(x) = 1 \).
	Moreover, by surjectivity of \( \Gamma_V \), for every \( x \in X \), there exists \( v \in V \) such that the functionals \( \omega(x, \cdot) \) and \( \omega(v, \cdot) \) coincide on \( V \).
	Thus, \( \omega(x-v, \cdot) \) vanishes on \( V \), that is, \( x - v \in V^\omega \).
	Hence, \( X \) is the algebraic direct sum of \( V \) and \( V^\omega \), which establishes the implication \iref{prop:symplecticFunctionalAnalysis:strongSymplecticSubspace:surjectiveGamma} \textrightarrow{} \iref{prop:symplecticFunctionalAnalysis:strongSymplecticSubspace:directSum}.
	For the last implication \iref{prop:symplecticFunctionalAnalysis:strongSymplecticSubspace:directSum} \textrightarrow{} \iref{prop:symplecticFunctionalAnalysis:strongSymplecticSubspace:restriction}, suppose now that \( X = V \oplus V^\omega \) is an algebraic direct sum.
	Then, \( V \) is a symplectic subspace, because the sum is direct.
	Moreover, the Hahn--Banach Theorem \parencite[Proposition~20.1.1]{Koethe1983} implies that the restriction map \( (X, \tau_\omega)' \to (V, \tau_\omega)' \) is surjective.
	Hence, using the definition of the symplectic topology, every \( \tau_\omega \)-continuous functional on \( V \) is obtained as the restriction to \( V \) of \( \omega(x, \cdot) \) for some \( x \in X \).
	Write \( x \) as \( x = v + w \) with \( v \in V \) and \( w \in V^\omega \).
	Then, the restrictions to \( V \) of \( \omega(x, \cdot) \) and \( \omega(v, \cdot) \) coincide.
	In other words, every \( \tau_\omega \)-continuous functional on \( V \) is of the form \( \omega(v, \cdot) \) for some \( v \in V \), which completes the proof.
\end{proof}

Finally, we come to what can be considered the linear toy example of symplectic reduction.
Later on, the nonlinear case of a symplectic action of a Lie group \( G \) on a symplectic manifold \( M \)  will be reduced to this simple setting by considering the action of the stabilizer \( G_m \) on a symplectic slice at \( m \in M \).
Assume that a compact Lie group \( G \) acts continuously and linearly on a symplectic vector space \( (X, \omega) \).
We say that the symplectic form \( \omega \) is preserved by the \( G \)-action if
\begin{equation}
	\omega(g \cdot x, g \cdot y) = \omega(x, y)
\end{equation}
holds for all \( g \in G \) and \( x,y \in X \).
\begin{prop}
	\label{prop:sympleticFunctionalAnalysis:invariantSubspaceSymplectic}
 	Let \( (X, \omega) \) be a symplectic vector space and let a compact Lie group \( G \) act continuously and linearly on \( X \) in such a way that the symplectic form \( \omega \) is preserved.
 	Then, the subspace \( X_G \) of \( G \)-invariant elements is a symplectic subspace of \( X \) and \( X = (X_G) \oplus (X_G)^\omega \) is an algebraic direct sum.
\end{prop}
\begin{proof}
	Let \( \tau_\omega \) be a symplectic topology on \( X \).
	We will denote the induced subspace topology on \( X_G \) by \( \tau_\omega \) as well.
	According to \cref{prop:symplecticFunctionalAnalysis:strongSymplecticSubspace}, we have to show that the linear map
	\begin{equation}
		\Gamma_{X_G}: X_G \mapsto (X_G, \tau_\omega)', \qquad v \mapsto \omega(v, \cdot)
	\end{equation}
	is surjective.
	For this purpose, let \( \alpha \in (X_G, \tau_\omega)' \).
	The functional \( \alpha \) can be extended to a \( \tau_\omega \)-continuous functional \( \bar{\alpha} \) defined on the whole of \( X \) according to the Hahn--Banach Theorem \parencite[Proposition~20.1.1]{Koethe1983}.
	By taking the average over the compact Lie group \( G \), we may assume that the extension \( \bar{\alpha} \) is \( G \)-invariant.
	Since \( \bar{\alpha} \) is \( \tau_\omega \)-continuous, there exists \( v \in X \) such that \( \bar{\alpha} = \omega(v, \cdot) \).
	As \( \bar{\alpha} \) and \( \omega \) are \( G \)-invariant, \( v \) has to be an element of \( X_G \).
	Clearly, \( \Gamma(v) = \alpha \), which shows that \( \Gamma \) is surjective.
\end{proof}

\section{Symplectic manifolds and momentum maps}
In this section, we will introduce the notion of symplectic structure on an infinite-dimensional manifold modelled on a locally convex space and we will study symplectic actions of infinite-dimensional Lie groups.
While some parts will be devoted to certain general properties of symplectic manifolds, our main focus lies on the momentum map geometry.
Although this topic is well-studied in finite dimensions and is the subject of many textbooks, we are not aware of any previous systematic treatments of the infinite-dimensional case, especially such going beyond the Banach realm.
Some aspects of the theory of Hamiltonian dynamics on symplectic Banach manifolds can be found in \parencite{ChernoffMarsden1974,MarsdenRatiuEtAl2002,RatiuMarsden2002}. 

\label{sec:momentumMaps}
\subsection{Group-valued momentum maps}
We start by defining the notion of a symplectic structure on an infinite-dimensio\-nal manifold.
\begin{defn}
	Let \( M \) be a smooth manifold.
	A differential \( 2 \)-form \( \omega \) on \( M \) is called a \emphDef{symplectic form} if it closed and, for every \( m \in M \), the induced bilinear form \( \omega_m: \TBundle_m M \times \TBundle_m M \to \R \) is a symplectic structure on \( \TBundle_m M \) in the sense of \cref{defn::symplecticFunctionalAnalysis:definitionSymplectic}.
\end{defn}
By \cref{defn::symplecticFunctionalAnalysis:definitionSymplectic}, for a symplectic form \( \omega \) on \( M \), the associated map
\begin{equation}
	\omega_m^\flat: \TBundle_m M \to (\TBundle_m M)', \qquad v \mapsto \omega_m(v, \cdot)
\end{equation}
is injective for all \( m \in M \).
If \( \omega_m^\flat \) is a topological isomorphism, then we say that \( \omega \) is a strongly symplectic form.
According to \cref{prop::symplecticFunctionalAnalysis:strongSymplecticOnlyInNormed}, this may be the case only when \( M \) is a Banach manifold.

\begin{example}
	\label{ex::diffOneFormsOnSurface:generalizationToConnections}
	We continue with the series of \cref{ex::diffOneFormsOnSurface:symplecticStructure} and include from now on the following geometric structure.
	Let \( P \to M \) be a  \( \UGroup(1) \)-bundle on the closed surface \( M \).
	The space \( \ConnSpace(P) \) of connections on \( P \) is an affine space modeled on the Fréchet space \( \DiffFormSpace^1(M) \).
	The \( 2 \)-form \( \omega \) on \( \ConnSpace(P) \) defined by the pairing
	\begin{equation}
		\label{eq::diffOneFormsOnSurface:symplecticForm}
		\omega_A (\alpha, \beta) = \int_M \alpha \wedge \beta
	\end{equation}
	for \( A \in \ConnSpace(P) \) and \( \alpha, \beta \in \DiffFormSpace^1(M) \isomorph \TBundle_A \ConnSpace(P) \) is a symplectic form.
	Indeed, \( \omega \) is closed, because the right-hand side of~\eqref{eq::diffOneFormsOnSurface:symplecticForm} is independent of \( A \), and non-degeneracy follows from the same arguments as in \cref{ex::diffOneFormsOnSurface:symplecticStructure}.
\end{example}

Let \( (M, \omega) \) be a symplectic manifold and let \( \Upsilon: G \times M \to M \) be an action of a Lie group \( G \) on \( M \) by symplectic diffeomorphisms, \ie \( \Upsilon_g^* \omega = \omega \) for all \( g \in G \).
We will refer to this setting by saying that \( (M, \omega, G) \) is a \emphDef{symplectic \( G \)-manifold}.
In finite dimensions, the conserved quantities corresponding to the \( G \)-symmetry of the system are encoded in the momentum map, which is a map from \( M \) to the dual space of the Lie algebra \( \LieA{g} \) of \( G \).
To make sense of the notion of a momentum map in our infinite-dimensional setting, we first need to explain what we will understand by the dual space of the Lie algebra.
Guided by the theory of dual pairs, we say that a
\emphDef{dual pair} \( \kappa(\LieA{h}, \LieA{g}) \) consists of a locally convex space \( \LieA{h} \) and a weakly non-degenerate jointly continuous bilinear map \( \kappa: \LieA{h} \times \LieA{g} \to \R \).
In this case, we say that \( \LieA{h} \) is the \emphDef{\( \kappa \)-dual space} of \( \LieA{g} \) and write \( \LieA{g}^*_\kappa \equiv \LieA{h} \).
Note that, in general, \( \LieA{g}^*_\kappa \) is not the topological dual \( \LieA{g}' \) of \( \LieA{g} \), but it is realized as a subspace of \( \LieA{g}' \) using the map \( \LieA{g}^*_\kappa \ni \mu \mapsto \kappa(\mu, \cdot) \in \LieA{g}' \).
We emphasize that the dual space is not unique and the choice of a dual space and a pairing is part of the data, and can be adapted to the problem at hand.
\begin{defn}
	\label{def:momentumMap:definitionVectorValued}
	Let \( (M, \omega, G) \) be a symplectic \( G \)-manifold and let \( \kappa(\LieA{g}^*_\kappa, \LieA{g}) \) be a dual pair.
	A map \( J: M \to \LieA{g}^*_\kappa \) is called a \emphDef{momentum map} if 
	\begin{equation}
		\label{eq:momentumMap:definingVectorValued}
		\xi^* \contr \omega + \kappa(\dif J, \xi) = 0
	\end{equation}
	holds for all \( \xi \in \LieA{g} \), where \( \xi^* \) denotes the fundamental vector field induced by the action of the Lie algebra \( \LieA{g} \) on \( M \).
\end{defn}
\begin{example}[finite-dimensional case]
	If \( \LieA{g} \) is finite-dimensional, then the dual space is unique and coincides with the algebraic dual \( \LieA{g}^* \) of \( \LieA{g} \) relative to the canonical pairing \( \kappa: \LieA{g}^* \times \LieA{g} \to \R \).
	Thus, the notion of a momentum map with respect to \( \kappa(\LieA{g}^*, \LieA{g}) \) is the same as the usual definition in the finite-dimensional case.
\end{example}
In the following examples, the spaces involved may be infinite-dimensional.
\begin{example}[Lie algebra-valued momentum maps]
	Let \( \kappa: \LieA{g} \times \LieA{g} \to \R \) be a jointly continuous, non-degenerate, symmetric bilinear form.
	Hence, the \( \kappa \)-dual space of \( \LieA{g} \) is \( \LieA{g} \) itself, and this leads to the notion of a Lie algebra-valued momentum map.
	Although this concept can be found in the literature since the mid-seventies, it was recently formalized by \textcite[Definition~4.3]{NeebSahlmannEtAl2014}:
	a Lie algebra-valued momentum map is a smooth map \( J: M \to \LieA{g} \) such that, for all \( \xi \in \LieA{g} \), the component functions \( J_\xi: M \to \R \) defined by \( J_\xi (m) = \kappa(J(m), \xi) \) for \(m \in  M\) satisfy
	\begin{equation}
		\xi^* \contr \omega + \dif J_\xi = 0.
	\end{equation}
	It is immediately clear from the definition that such a Lie algebra-valued momentum map can be regarded as a \( \LieA{g} \)-valued momentum map with respect to the dual pair \( \kappa(\LieA{g}, \LieA{g}) \).
\end{example}

\begin{example}
	\label{ex:momentumMap:linearCase}
	Let \( (X, \omega) \) be a symplectic vector space endowed with a continuous linear action of the compact finite-dimensional Lie group \( G \).
	Assume that the symplectic form is preserved by the action, which in the linear setting simply means that
	\begin{equation}
		\omega(g \cdot x, g \cdot y) = \omega(x, y)
	\end{equation}
	for all \( x, y \in X \) and \( g \in G \).
	As a consequence, the Lie algebra action of \( \LieA{g} \) is skew-symmetric with respect to \( \omega \).
	A straightforward calculation shows that the quadratic map \( J: X \to \LieA{g}^* \) defined by
	\begin{equation}
		\kappa(J(x), \xi) = \frac{1}{2} \omega(x, \xi \ldot x),
	\end{equation}
	for \( \xi \in \LieA{g} \), is a momentum map for the \( G \)-action relative to the canonical dual pair \( \kappa(\LieA{g}^*, \LieA{g}) \).
\end{example}

\begin{example}
	\label{ex::diffOneFormsOnSurface:gaugeAction}
	Continuing in the setting of \cref{ex::diffOneFormsOnSurface:generalizationToConnections}, let \( P \to M \) be a principal \( \UGroup(1) \)-bundle on the closed surface \( M \).
	The group \( \GauGroup(P) \) of gauge transformations of \( P \) is identified with the space \( \sFunctionSpace(M, \UGroup(1)) \) and thus is a Fréchet Lie group with Lie algebra \( \GauAlgebra(P) = \sFunctionSpace(M) \).
	The natural pairing
	\begin{equation}
		\kappa(\alpha, \phi) = \int_M \phi \, \alpha,
	\end{equation}
	for \( \alpha \in \DiffFormSpace^2(M) \) and \( \phi \in \GauAlgebra(P) \), puts \( \GauAlgebra(P) \) in duality to \( \DiffFormSpace^2(M) \).

	The gauge group \( \GauGroup(P) \) acts on \( \ConnSpace(P) \) via gauge transformations
	\begin{equation}
		\GauGroup(P) \times \ConnSpace(P) \to \ConnSpace(P), \qquad (\lambda, A) \mapsto A - \difLog^R \lambda,
	\end{equation}
	where \( \difLog^R \lambda \in \DiffFormSpace^1(M) \) denotes the right logarithmic derivative of \( \lambda \in \sFunctionSpace(M, \UGroup(1)) \) defined by \( \difLog^R \lambda (v) = \tangent_m \lambda (v) \ldot \lambda^{-1}(m) \) for \( v \in \TBundle_m M \). 
	The infinitesimal action of \( \GauAlgebra(P) \) coincides with minus the exterior differential \( \dif: \sFunctionSpace(M, \R) \to \DiffFormSpace^1(M) \).
	The curvature map
	\begin{equation}
		\label{eq::diffOneFormsOnSurface:gaugeAction:momentumMap}
		\SectionMapAbb{J}: \ConnSpace(P) \to \DiffFormSpace^2(M), \qquad A \mapsto - F_A
	\end{equation}
	is the momentum map for this action with respect to the pairing \( \kappa \) introduced above.
	Indeed, we have \( \tangent_A \SectionMapAbb{J} = - \dif \) as \( F_{A + \alpha} = F_A + \dif \alpha \) for every \( \alpha \in \DiffFormSpace^1(M) \), and the calculation
	\begin{equation}\begin{split}
		\omega(\psi \ldot A, \alpha)
			&= \omega(- \dif \psi, \alpha) 
			\\
			&= - \int_M \dif \psi \wedge \alpha
			= \int_M \psi \wedge \dif \alpha
			\\
			&= \kappa(\dif \alpha, \psi)
			= - \kappa(\tangent_A \SectionMapAbb{J} (\alpha), \psi)
	\end{split}\end{equation}
	for \( \psi \in \sFunctionSpace(M) \) and \( \alpha \in \DiffFormSpace^1(M) \) verifies the momentum map relation~\eqref{eq:momentumMap:definingVectorValued}.
\end{example}

Even in finite dimensions, a momentum map for a symplectic action does not need to exist.
This is the case if the \( 1 \)-form \( \xi^* \contr \omega \) is only closed and not exact for some \( \xi \in \LieA{g} \).
Since \( \xi^* \contr \omega \) is closed, it gives rise to the period homomorphism
\begin{equation}
	\per_\xi: \sHomology^1(M, \Z) \to \R, \qquad \equivClass{\gamma} \mapsto \int_\gamma (\xi^* \contr \omega),
\end{equation}
where \( \gamma \) is a closed curve in \( M \).
One could argue that the topological data encoded in the period homomorphism is conserved by the action and such conservation laws should be encoded in the momentum map as well.
However, the classical momentum map takes values in a continuous vector space and thus there is no space to store (discrete) topological information such as the winding numbers of closed curves.
In order to capture these additional data, in \parencite{DiezRatiuAutomorphisms} the concept of a group-valued momentum map has been introduced.
At the heart of this generalized notion of a momentum map lies the observation that if the periods of \( \xi^* \contr \omega \) are integral, \ie \( \per_\xi(\equivClass{\gamma}) \in \Z \) for all closed curves \( \gamma \) in \( M \), then there exists a smooth map \( J_\xi: M \to \UGroup(1) \) such that the left logarithmic derivative of \( J_\xi \) equals \( \xi^* \contr \omega \).
The \( \UGroup(1) \)-valued component function \( J_\xi \) can be viewed as a generalized primitive of \( \xi^* \contr \omega \).
This extends the real-valued maps components of an ordinary momentum map.
The additional topological information is encoded in the winding number of \( J_\xi \).

In order to formalize these ideas, we say that a \emphDef{dual pair\footnotemark{} of Lie groups} \( \kappa(H, G) \) consists of a pair of Lie groups \( G \) and \( H \) and a jointly continuous, non-degenerate bilinear map \( \kappa: \LieA{h} \times \LieA{g} \to \R \).
\footnotetext{The notion of a dual pair of Lie groups should not be confused with the notion of a Howe pair.}
As for Lie algebras, we often write \( G^*_\kappa \defeq H \).
In the following, we assume for simplicity that \( G^*_\kappa \) is abelian, write its group operation as addition and use \( 0 \in G^*_\kappa \) for the identity element.
The reader is referred to \parencite{DiezRatiuAutomorphisms} for the general non-abelian case.
\begin{defn}
	\label{def:momentumMap:definitionGroupValued}
	Let \( (M, \omega, G) \) be a symplectic \(  G \)-manifold.
	A \emphDef{group-valued momentum map} is a tuple \( (J, \kappa) \), where \( \kappa(G^*_\kappa, G) \) is a dual pair of Lie groups with abelian \( G^*_\kappa \) and \( J: M \to G^*_\kappa \) is a smooth map satisfying
	\begin{equation}
		\label{eq::momentumMap:defEq}
		\xi^* \contr \omega + \kappa(\difLog J, \xi) = 0
	\end{equation}
	for all \( \xi \in \LieA{g} \), where \( \difLog J \in \DiffFormSpace^1(M, \LieA{g}^*_\kappa) \) denotes the (left) logarithmic derivative of \( J \) defined by \( \difLog J (v) = J(m)^{-1} \ldot \tangent_m J (v) \) for \( v \in \TBundle_m M \).
\end{defn}
The notion of a group-valued momentum map is inspired by the theory of Poisson Lie groups \parencite{LuWeinstein1990}.
However, we do not require \( G \) to be a Poisson Lie group.

\begin{example}
	\label{ex:momentumMap:classicalAsGroupValued}
	Let \( \kappa(\LieA{g}^*_\kappa, \LieA{g}) \) be a dual pair of Lie algebras.
	We can consider \( \LieA{g}^*_\kappa \) as a Lie group with group multiplication given by addition, and then \( \kappa(\LieA{g}^*_\kappa, G) \) is a dual pair of Lie groups, that is, \( G^*_\kappa \defeq \LieA{g}^*_\kappa \).
	Moreover, a group-valued momentum map \( J: M \to \LieA{g}^*_\kappa \) relative to \( \kappa(\LieA{g}^*_\kappa, G) \) in the sense of \cref{def:momentumMap:definitionGroupValued} is the same as an ordinary momentum map relative to \( \kappa(\LieA{g}^*_\kappa, \LieA{g}) \) in the sense of \cref{def:momentumMap:definitionVectorValued}.
\end{example}
Thus, the notion of a group-valued momentum map is a generalization of the classical momentum map taking values in the dual space of the Lie algebra.
However, a group-valued momentum map might exist even if a classical momentum map does not exist.

\begin{example}
	Consider the action of \( \UGroup(1) \) on the torus \( T^2 = \UGroup(1) \times \UGroup(1) \) by multiplication in the first factor.
	This action is symplectic with respect to the natural volume form on \( T^2 \), but it does not have a classical momentum map.
	On the other hand, the projection onto the second factor yields a \( \UGroup(1) \)-valued momentum map \( J: T^2 \to \UGroup(1) \).
\end{example}

\begin{example}
	\label{ex::diffOneFormsOnSurface:pointedGaugeTransf}
    \NewDocumentCommand { \pnt } { }{
	    m_0
    }

	Continuing with the setting of \cref{ex::diffOneFormsOnSurface:gaugeAction}, let \( P \to M \) be a principal \( \UGroup(1) \)-bundle on the compact connected surface \( M \).
	Choose a point \( \pnt \in M \).
	The evaluation \( \ev_{\pnt}: \GauGroup(P) \to \UGroup(1) \) at \( \pnt \) is a morphism of Lie groups.
	The kernel of \( \ev_{\pnt} \) is called the group of pointed gauge transformations and is denoted by \( \GauGroup_{\pnt}(P) \). 
	This group is a normal, locally exponential Lie subgroup of \( \GauGroup(P) \) due to \parencite[Proposition~IV.3.4]{Neeb2006}.
	The Lie algebra \( \GauAlgebra_{\pnt}(P) \) of \( \GauGroup_{\pnt}(P) \) consists of all \( \phi \in \GauAlgebra(P) \) vanishing at \( \pnt \).
	The integration pairing\footnotemark{} \( \kappa \) realizes \( \GauAlgebra_{\pnt}(P)^* = \dif \DiffFormSpace^1(M) \) as a natural dual of \( \GauAlgebra_{\pnt}(P) \).
	\footnotetext{The map \( \phi \mapsto (\phi - \phi(\pnt), \phi(\pnt)) \) gives a topological decomposition \( \GauAlgebra(P) \isomorph \GauAlgebra_{\pnt}(P) \oplus \R \) which is dual to the Hodge isomorphism \( \DiffFormSpace^2(M) = \dif \DiffFormSpace^1(M) \oplus \deRCohomology^2(M) \).}
	However, the momentum map \( \SectionMapAbb{J} \) for the \( \GauGroup(P) \)-action defined in~\eqref{eq::diffOneFormsOnSurface:gaugeAction:momentumMap} does not yield a \( \GauAlgebra_{\pnt}(P)^* \)-valued momentum map for the \( \GauGroup_{\pnt}(P) \)-action, because the curvature \( F_A \) of \( A \in \ConnSpace(P) \) is in general not exact.
	Instead, the curvature is a closed \( 2 \)-form with integral periods, \ie \( F_A \in \clZDiffFormSpace^2(M) \).
	Thus, we get a well-defined map
	\begin{equation}
		\SectionMapAbb{J}_{\pnt}: \ConnSpace(P) \to \clZDiffFormSpace^2(M), \qquad A \mapsto - F_A
	\end{equation}
	and it can easily be verified that \( \SectionMapAbb{J}_{\pnt} \) is a \( \clZDiffFormSpace^2(M) \)-valued momentum map for the action of \( \GauGroup_{\pnt}(P) \).
	Here, we view the abelian Lie group \( \clZDiffFormSpace^2(M) \) with group multiplication given by addition as a dual group \( \GauGroup_{\pnt}(P)^*_\kappa \) of \( \GauGroup_{\pnt}(P) \), because the exact sequence
	\begin{equationcd}
		0
			\to[r]
		& \dif \DiffFormSpace^1(M)
			\to[r]
		& \clZDiffFormSpace^2(M)
			\to[r]
		& \deRCohomology^2(M, \Z)
			\to[r]
		& 0
	\end{equationcd}
	identifies the Lie algebra of \( \clZDiffFormSpace^2(M) \) with \( \dif \DiffFormSpace^1(M) = \GauAlgebra_{\pnt}(P)^* \).
	The \( \clZDiffFormSpace^2(M) \)-valued momentum map \( \SectionMapAbb{J}_{\pnt} \) remembers the topological type of \( P \) in form of the Chern class of \(  P \):
	\begin{equation}
		\chernClass(P) = \int_M F_A = - \int_M \SectionMapAbb{J}_{\pnt}(A) \in \Z
	\end{equation}
	for some \( A \in \ConnSpace(P) \).
	The group-valued momentum map \( \SectionMapAbb{J}_{\pnt} \) simplifies to a classical momentum map if and only if the bundle \( P \) is trivial.
	In summary, for non-trivial bundles \( P \), the group-valued momentum map \( \SectionMapAbb{J}_{\pnt} \) carries the topological structure of \( P \).
\end{example}
The example of the action of the group of pointed gauge transformations shows that valuable topological information is contained in a group-valued momentum map.
This is especially important for actions of diffeomorphism groups, which by their very nature are sensitive to topological properties of the manifold.
We refer the reader to \parencite{DiezRatiuAutomorphisms} for further information.

Despite its general nature, a group-valued momentum map still captures conserved quantities of the dynamical system, \ie, it has the \emph{Noether property} (see \parencite[Definition~4.3.1]{OrtegaRatiu2003}).
In finite dimensions, every Hamiltonian \( h \) on a symplectic manifold \( (M, \omega) \) induces a Hamiltonian flow.
This no longer holds true in an infinite-dimensional context.
For one thing, the map \( \omega^\flat: \TBundle M \to \CotBundle M \) induced by the symplectic form \( \omega \) on \( M \) is, in general, only injective and not surjective.
Hence, a Hamiltonian vector field \( X_h \) associated with a Hamiltonian \( h: M \to \R \) by the relation
\begin{equation}
	X_h \contr \omega + \dif h = 0	
\end{equation}
may not exist\footnote{For example, consider a symplectic vector space \( (X, \omega) \). The Hamiltonian vector field associated with a continuous linear functional \( h: X \to \R \) exists if and only if \( h \in X' \) lies in the image of \( \omega^\flat: X \to X' \).}.
Even if \( X_h \) exists, it may not have a unique flow.
The construction of a flow requires the solution of an ordinary differential equation on \( M \), which a priori is not guaranteed to exist and to be unique in infinite dimensions.
Nonetheless, for concrete examples, one can often show that a unique flow exists.
\begin{prop}[Noether's Theorem]
	\label{prop:momentumMap:noetherTheorem}
	Let \( (M, \omega, G) \) be a symplectic \( G \)-manifold and let \( \kappa(G^*_\kappa, G) \) be a dual pair of Lie groups.
	Assume that the \( G \)-action on \( M \) has a \( G^*_\kappa \)-valued momentum map \( J: M \to G^*_\kappa \).
	Let \( h \in \sFunctionSpace(M) \) be a smooth function for which the Hamiltonian vector field \( X_h \) exists and has a unique local flow.
	If \( h \) is \( G \)-invariant, then \( J \) is constant along the integral curves of \( X_h \). 
\end{prop}
\begin{proof}
	Let \( \xi \in \LieA{g} \) and \( m \in M \).
	Using the defining equation for the momentum map, we have
	\begin{equation}\begin{split}
	\kappa\bigl((\difLog J)_m (X_h), \xi\bigr)
		&= - \omega_m \bigl(\xi^*, X_h\bigr) 
		= - (\dif h)_m (\xi^*) \\
		&= - \difFracAt{}{\varepsilon}{0} h(\exp(\varepsilon \xi) \cdot m)
		= 0
	\end{split}
	\end{equation}
	by \(G\)-invariance of \(h\).
	Since \( \xi \in \LieA{g} \) is arbitrary and the pairing \( \kappa \) is non-degenerate, we conclude 
	\( \difLog J (X_h) = 0 \).
	Hence, \( J \) is constant along integral curves of \( X_h \).
\end{proof}

Fix a dual pair \( \kappa(G^*_\kappa, G) \) of Lie groups.
Let \(  G \) act on the symplectic manifold \( (M, \omega) \) and assume that the action has a group-valued momentum map \( J: M \to G^*_\kappa \).
A natural question to ask is in which sense \( J \) is equivariant.
Recall that the coadjoint action is defined with respect to the duality pairing \( \kappa \) by
\begin{equation}
	\kappa\bigl(\CoAdAction_g \mu, \xi\bigr) = \kappa\bigl(\mu, \AdAction_{g^{-1}} \xi\bigr),
\end{equation}
for \( g \in G \), \( \xi \in \LieA{g} \) and \(\mu \in \LieA{g}^*_\kappa \).
In infinite dimensions, this relation only ensures uniqueness of \( \CoAdAction \) but not its existence, because \( \kappa \) is in general only weakly non-degenerate.
In the following, we will assume that \( \CoAdAction \) exists.
\begin{defn}
	Let \( \kappa(G^*_\kappa, G) \) be a dual pair of Lie groups.
	A left action \( \Upsilon: G \times G^*_\kappa \to G^*_\kappa \) of \( G \) on \( G^*_\kappa \) is called a \emphDef{coconjugation action} if it integrates the coadjoint action, that is, 
	\begin{equation}
		\difLog_{\eta} \Upsilon_g (\eta \ldot \mu) = \CoAdAction_{g} \mu
	\end{equation}
	holds for all \( g \in G \), \( \eta \in G^*_\kappa\) and \(\mu \in \LieA{g}^*_\kappa \), where \( \eta \ldot \mu \in \TBundle_\eta G^*_\kappa \) is the derivative at the identity of the (left) translation by \(\eta \in  G^*_\kappa\) on \( G^*_\kappa \) in the direction \( \mu \in  \LieA{g}^*_\kappa \) and \( \difLog_{\eta} \Upsilon_g (\eta \ldot \mu) \in \LieA{g}^*_\kappa \) denotes the (left) logarithmic derivative at \( \eta \) of the map \( \Upsilon_g: G^*_\kappa \to  G^*_\kappa \) in the direction \( \eta \ldot \mu \in \TBundle_\eta G^*_\kappa \).
	If, in addition, \( \Upsilon_g (\zeta + \eta) = \Upsilon_g (\zeta) + \Upsilon_g (\eta) \) holds for all \( \zeta, \eta \in G^*_\kappa \), then we say that the coconjugation action is \emphDef{standard}.

	For a given coconjugation action, we say that the group-valued momentum map \( J \) is \emphDef{equivariant} if it is \( G \)-equivariant as a map \( J: M \to G^*_\kappa \).
\end{defn}
\begin{example}
	The coadjoint representation is a standard coconjugation action of \( G \) on \( G^*_\kappa \defeq \LieA{g}^*_\kappa \), \cf \cref{ex:momentumMap:classicalAsGroupValued}.
	Moreover, every \( 1 \)-cocycle \( c: G \to \LieA{g}^*_\kappa \) defines a (non-standard) coconjugation action by
	\begin{equation}
		(g, \mu) \mapsto \CoAdAction_{g} \mu + c(g).
	\end{equation}
	 Recall that such affine actions play an important role for ordinary non-equivariant momentum maps (see \parencite[Definition~4.5.23]{OrtegaRatiu2003}).
\end{example}
Every standard coconjugation action \( \Upsilon \) of \( G \) on \( G^*_\kappa \) defines a Poisson Lie structure \( \Lambda: G^*_\kappa \times \LieA{g} \to \LieA{g}^*_\kappa \) on \( G^*_\kappa \) by
\begin{equation}
	\label{eq:momentumMap:poissonLie}
	\Lambda(\eta, \xi) = \eta^{-1} \ldot \tangent_e \Upsilon_\eta (\xi),
\end{equation}
where \( \Upsilon_\eta: G \to G^*_\kappa \) is the orbit map through \( \eta \in G^*_\kappa \).
Moreover, non-standard coconjugation actions are related to affine Poisson structures.
We refer to \parencite{DiezRatiuAutomorphisms} for more details on coconjugation actions and their relation to Poisson structures.
A group-valued momentum map is equivariant with respect to a given standard coconjugation action \( \Upsilon \) if and only if it is a Poisson map with respect to the induced Poisson structure \( \Lambda \).
This is essentially a consequence of the infinitesimal equivariance property 
\begin{equation}
	\label{eq:momentumMap:equivarianceInfinitesimal}
	\difLog_m J (\xi \ldot m) = J(m)^{-1} \ldot \bigl(\xi \ldot J(m)\bigr) = \Lambda(J(m), \xi),
\end{equation}
 see \parencite{DiezRatiuAutomorphisms} for details.
In the case when \( G^*_\kappa = \LieA{g}^*_\kappa \) endowed with the coadjoint representation as the standard coconjugation action, the Poisson Lie structure defined in~\eqref{eq:momentumMap:poissonLie} is simply given by
\begin{equation}
	\label{eq:momentumMap:poissonLie:usual}
	\Lambda(\eta, \xi) = \CoadAction_\xi \eta
\end{equation}
and the equivariance condition~\eqref{eq:momentumMap:equivarianceInfinitesimal} takes the usual form
\begin{equation}
	\label{eq:momentumMap:equivarianceInfinitesimal:usual}
	\tangent_m J (\xi \ldot m) = \CoadAction_\xi \bigl(J(m)\bigr).
\end{equation}

\subsection{Bifurcation Lemma}
\label{sec:momentumMap:bifurcationLemma}

Let \( (M, \omega, G) \) be a symplectic \( G \)-manifold with group-valued momentum map \( J: M \to G^*_\kappa \).
Throughout this section we assume that \( J \) is \( G \)-equivariant with respect to a given coconjugation action on \( G^*_\kappa \).
The momentum map relation~\eqref{eq::momentumMap:defEq} can be written in the form
\begin{equation}
 	\kappa\bigl(\difLog_m J (v), \xi \bigr) = \omega_m \bigl(v, \tangent_e \Upsilon_m (\xi) \bigr)
\end{equation}
for all \( m \in M \), \( v \in \TBundle_m M \) and \( \xi \in \LieA{g} \).
Phrased in the language of dual pairs, this identity shows that \( \difLog_m J \) is the adjoint of the infinitesimal action \( \tangent_e \Upsilon_m \) with respect to the dual pairs \( \omega(\TBundle_m M,\TBundle_m M) \) and \( \kappa(\LieA{g}^*_\kappa, \LieA{g}) \), \cf \parencite[Section~IV.2]{Schaefer1971}.
In particular, \( \tangent_e \Upsilon_m \) and \( \difLog_m J \) are continuous with respect to the weak topologies induced by the dual pairs \( \omega \) and \( \kappa \).
We can represent this situation diagrammatically as follows:
\begin{equationcd}[tikz=adjoint diagram]
	\TBundle_m M
		\arrow[d, phantom, "\times_{\mathrlap{\omega_m}}"]
		& \LieA{g}
		\to[l, swap, "\tangent_e \Upsilon_m"]
		\arrow[d, phantom, "\times_{\mathrlap{\kappa}}"] \\
	\TBundle_m M
		\to[r, swap, "\difLog_m J"]
		& \LieA{g}^*_\kappa.
\end{equationcd}
Taken with a grain of salt, the momentum map \( J \) \enquote{integrates} this diagram to
\begin{equationcd}[tikz=adjoint diagram]
	M
		\arrow[d, phantom, "\times_{\mathrlap{\omega}}"]
		& G
		\to[l, swap, "\Upsilon_m"]
		\arrow[d, phantom, "\times_{\mathrlap{\kappa}}"] \\
	M
		\to[r, swap, "J"]
		& G^*_\kappa.
\end{equationcd}
In this sense, we may view a (group-valued) momentum map as the symplectic adjoint of the group action.

The theory of adjoints allows us to generalize the following important result concerning the infinitesimal momentum map geometry to infinite dimensions (see, \eg, \parencite[Proposition~4.5.12]{OrtegaRatiu2003} for the finite-dimensional version).
The polar, or annihilator, with respect to the dual pair \( \kappa \) will be denoted by \( \perp \), \cf~\parencite[Section~II.8.2]{Jarchow1981}.
\begin{lemma}[Bifurcation Lemma]
	\label{prop:bifurcationLemma}
	Let \( (M, \omega, G) \) be a symplectic \( G \)-manifold with equivariant momentum map \( J: M \to G^*_\kappa \).
	Then, for every \( m \in M \), the following holds:
	\begin{thmenumerate}
		\item (weak version)
			\label{prop::bifurcationLemma:weak}
			\begin{subequations}\label{eq:bifurcationLemma:weak}\begin{align}
				\ker \difLog_m J &= (\LieA{g} \ldot m)^\omega, \\
				(\img \difLog_m J)^\perp &= \LieA{g}_m.
			\end{align}\end{subequations}
		\item (strong version)
			\label{prop::bifurcationLemma:strong}
			If, additionally, \( \LieA{g} \ldot m = (\LieA{g} \ldot m)^{\omega\omega} \) holds, then we have
			\begin{subequations}
			\begin{equation}
				\label{eq:bifurcationLemma:strong:ker}
				(\ker \difLog_m J)^\omega = \LieA{g} \ldot m
			\end{equation}
			and
			\begin{equation}
				\ker \difLog_m J \intersect (\ker \difLog_m J)^\omega = \LieA{g}_\mu \ldot m,
			\end{equation}
			where \( \mu = J(m) \in G^*_\kappa \).
			\sideRemark{The momentum map fails to be a submersion exactly at points with non-trivial stabilizer.}
			Moreover, if \( \img \difLog_m J = (\img \difLog_m J)^{\perp\perp} \), then
			\begin{equation}
				\label{eq:bifurcationLemma:strong:img}
				\img \difLog_m J = \LieA{g}_m^\perp.
				\qedhere		
			\end{equation}
			\end{subequations}
	\end{thmenumerate}
\end{lemma}
\begin{proof}
	The formulae~\eqref{eq:bifurcationLemma:weak} follow from the general relation \( \ker T = (\img T^*)^\perp \) relating the image of a linear operator \( T \) and the kernel of its adjoint \( T^* \), see \parencite[Proposition~IV.2.3]{Schaefer1971}, applied to \( T = \difLog_m J \) and \( T = \tangent_e \Upsilon_m \), respectively.
	In a similar vein, the general identity \( (\img T)^{\perp\perp} = (\ker T^*)^\perp \) of \parencite[Proposition~IV.2.3]{Schaefer1971} implies~\eqref{eq:bifurcationLemma:strong:ker} and~\eqref{eq:bifurcationLemma:strong:img}.
	Using~\eqref{eq:bifurcationLemma:strong:ker}, we find
	\begin{equation}
		\ker \difLog_m J \intersect (\ker \difLog_m J)^\omega = \ker \difLog_m J \intersect \LieA{g} \ldot m \, .
	\end{equation}
	The equivariance property~\eqref{eq:momentumMap:equivarianceInfinitesimal} of \( J \) implies that \( \xi \ldot m \in \ker \difLog_m J \) if and only if \( \xi \ldot \mu = 0 \), \ie, \( \xi \in \LieA{g}_\mu \).
	Hence, \( \ker \difLog_m J \intersect \LieA{g} \ldot m = \LieA{g}_\mu \ldot m \).
\end{proof}
Recall from \cref{prop:symplecticFunctionalAnalysis:symplectiallyClosedIffDoubleOrthogonal} that the condition \( (\LieA{g} \ldot m)^{\omega\omega} = \LieA{g} \ldot m \) is equivalent to \( \LieA{g} \ldot m \) being symplectically closed.
In particular, by \cref{prop:symplecticFunctionalAnalysis:finiteDimSubspaceSymplectiallyClosed}, every infinitesimal orbit \( \LieA{g} \ldot m \) of the action of a finite-dimensional Lie group \( G \) is symplectically closed.

The Bifurcation Lemma is a helpful tool to study the bifurcations of Hamiltonian flows in the neighborhood of points with non-trivial stabilizer.
With view towards symplectic reduction, its main significance is the existence of a certain symplectic space normal to the orbit.
\begin{prop}
	\label{prop:bifurcationLemma:kernelMomentumMapIsSymplectic}
	Let \( (M, \omega, G) \) be a symplectic \( G \)-manifold with equivariant momentum map \( J: M \to G^*_\kappa \).
	Let \( m \in M \) and \( \mu = J(m) \in G^*_\kappa \).
	Assume that the stabilizer \( G_\mu \) of \( \mu \) is a Lie subgroup of \( G \) and that \( \LieA{g} \ldot m \) is symplectically closed.
	Moreover, assume that a slice \( S \) at \( m \) for the induced \( G_\mu \)-action exists. 
	Then,
	\begin{equation}
		\label{eq:bifurcationLemma:kernelMomentumMapIsSymplectic}
		E \defeq \TBundle_m S \intersect \ker \difLog_m J
	\end{equation}
	is a symplectic subspace of \( (\TBundle_m M, \omega_m) \) in the sense of \cref{defn:sympleticFunctionalAnalysis:symplecticSubspace}.
\end{prop}
\begin{proof}
	Using~\eqref{eq:bifurcationLemma:weak} and the general identity \( (Y_1 + Y_2)^\omega = Y_1^\omega \intersect Y_2^\omega \), \cf \parencite[Proposition~8.2.1]{Jarchow1981}, we obtain
	\begin{equation}\begin{split}
		E \intersect E^\omega
			&= \TBundle_m S \intersect \ker \difLog_m J \intersect E^\omega \\
			&= \TBundle_m S \intersect (\LieA{g} \ldot m)^\omega \intersect E^\omega \\
			&= \TBundle_m S \intersect (\LieA{g} \ldot m + E)^\omega.
	\end{split}\end{equation}
	Since \( \ker \difLog_m J \subseteq \LieA{g} \ldot m + E \), the anti-monotony of the symplectic orthogonal (for subsets \( A_1 \subseteq A_2 \), we have \( A_1^\omega \supseteq A_2^\omega \), see \parencite[Proposition~8.2.1]{Jarchow1981}), implies \( E \intersect E^\omega \subseteq (\ker \difLog_m J)^\omega \).
	Thus, using the Bifurcation \cref{prop::bifurcationLemma:strong}, we obtain
	\begin{equation}
		E \intersect E^\omega
			\subseteq \TBundle_m S \intersect \ker \difLog_m J \intersect (\ker \difLog_m J)^\omega
			= \TBundle_m S \intersect \LieA{g}_\mu \ldot m
			= \set{0},
	\end{equation}
	because \( S \) is a slice for the \( G_\mu \)-action.
	Hence, \( E \) is a symplectic subspace of \( (\TBundle_m M, \omega_m) \).
\end{proof}
By \( G \)-equivariance of \( J \), we have a chain complex
\begin{equationcd}[label=eq:bifurcationLemma:chainComplex]
	0
		\to[r]
	& \LieA{g}_\mu
		\to[r, "\tangent_e \Upsilon_m"]
	& \TBundle_m M
		\to[r, "\difLog_m J"]
	& \LieA{g}^*_\kappa
		\to[r]
	& 0 .
\end{equationcd}
Since, for every \( G_\mu \)-slice \( S \) at \( m \), \( \TBundle_m S \) is a topological complement of \( \LieA{g}_\mu \ldot m \) in \( \TBundle_m M \), the subspace \( E \) defined in~\eqref{eq:bifurcationLemma:kernelMomentumMapIsSymplectic} is identified with the middle homology group of this complex.
If the assumptions of the Strong Bifurcation Lemma hold, then the other homology groups are given by \( \LieA{g}_m \) and \( \LieA{g}_m^* \) according to~\eqref{eq:bifurcationLemma:strong:img}.
Thus, if~\eqref{eq:bifurcationLemma:chainComplex} is a Fredholm complex, then its Euler characteristic is given by
\begin{equation}
	\dim \LieA{g}_m - \dim E + \dim \LieA{g}_m^* = 2 \dim \LieA{g}_m - \dim E. 
\end{equation}
According to \cref{prop:bifurcationLemma:kernelMomentumMapIsSymplectic}, this number is always even.
As we will see below, the Euler characteristic of~\eqref{eq:bifurcationLemma:chainComplex} coincides (up to a sign) with the virtual dimension of the symplectic quotient. This observation is a first indication that the latter carries a symplectic structure.
\begin{remark}[Witt--Artin Decomposition]
	\label{rem:symplectic:bifurcationLemma:decomposition}
	Let us connect \cref{prop:bifurcationLemma:kernelMomentumMapIsSymplectic} to the classical Witt--Artin decomposition.
	In the setting of \cref{prop:bifurcationLemma:kernelMomentumMapIsSymplectic}, assume additionally that \( G_\mu \) is a split Lie subgroup of \( G \), that is, a Lie subgroup for which there exists a topological complement \( \LieA{q} \) of \( \LieA{g}_\mu \) in \( \LieA{g} \).
	Since \( S \) is a \( G_\mu \)-slice at \( m \), we have a topological decomposition
	\begin{equation}
		\TBundle_m M = \LieA{g}_\mu \ldot m \oplus \TBundle_m S.
	\end{equation}
	Assume that \( E \) has a topological complement \( F' \) in \( \TBundle_m S \).
	Since \( \LieA{q} \ldot m \intersect E = \set{0} \), \( \LieA{q} \ldot m \) is a subspace of \( F' \).
	Moreover, assume that \( \LieA{q} \ldot m \) has a topological complement \( F \) in \( F' \).
	Then we obtain the topological decomposition:
	\begin{equation}
		\label{eq:symplectic:bifurcationLemma:wittArtinDecomposition}
		\TBundle_m M = 
			\mathrlap{\overbrace{\phantom{\LieA{q} \ldot m \oplus \LieA{g}_\mu \ldot m}}^{\LieA{g} \ldot m}}
			\LieA{q} \ldot m \oplus
			\mathrlap{\underbrace{\phantom{\LieA{g}_\mu \ldot m \oplus X}}_{\ker \difLog_m J}}
			\LieA{g}_\mu \ldot m \oplus E \oplus F.
	\end{equation}
	For a classical momentum map, this decomposition is often referred to as the Witt--Artin decomposition (see \eg \parencite[Equation~10.2.11]{RudolphSchmidt2012}).
	We emphasize that, in addition to the conditions of \cref{prop:bifurcationLemma:kernelMomentumMapIsSymplectic}, the derivation of the Witt--Artin decomposition in infinite dimensions relied on several additional assumptions concerning the existence of topological complements. 

	Let us now pass to the corresponding decomposition of \( \LieA{g}^*_\kappa \).
	For this purpose, suppose that \( \LieA{g}_m \) is topologically complemented in \( \LieA{g}_\mu \) by a subspace \( \LieA{m} \), so that we get the topological decomposition
	\begin{equation}
		\LieA{g} = \LieA{q} \oplus \underbrace{\LieA{m} \oplus \LieA{g}_m}_{\LieA{g}_\mu}.
	\end{equation}
	If this decomposition is weakly continuous with respect to the dual pair \( \kappa \), then, by \parencite[Section~20.5]{Koethe1983}, we obtain a dual decomposition of \( \LieA{g}^*_\kappa \):
	\begin{equation}
		\LieA{g}^*_\kappa = \LieA{q}^* \oplus \LieA{m}^* \oplus \LieA{g}_m^*,
	\end{equation}
	where \( \LieA{q}^* \) is the annihilator of \( \LieA{g}_\mu \) in \( \LieA{g} \) and \( \LieA{m}^* \) is the annihilator of \( \LieA{g}_m \) in \( \LieA{g}_\mu^* \).
	Moreover, assume that \( \img \difLog_m J = (\img \difLog_m J)^{\perp\perp} \).
	Then, by \cref{prop::bifurcationLemma:strong}, we have \( \img \difLog_m J = \LieA{g}_m^\perp = \LieA{q}^* \oplus \LieA{m}^* \).
	Hence, in summary, we obtain the decomposition
	\begin{equation}
		\LieA{g}^*_\kappa = \underbrace{\LieA{q}^* \oplus \LieA{m}^*}_{\img \difLog_m J}{} \oplus \LieA{g}_m^*.
	\end{equation}
	Accordingly, \( \difLog_m J \) yields, by restriction, a bijection between \( \LieA{q} \ldot m \oplus F \) and \( \LieA{q}^* \oplus \LieA{m}^* \), \cf~\eqref{eq:symplectic:bifurcationLemma:wittArtinDecomposition}.
	Due to the equivariance property~\eqref{eq:momentumMap:equivarianceInfinitesimal}, the restriction of \( \difLog_m J \) to \( \LieA{q} \ldot m \) yields the map
	\begin{equation}
		I_\mu: \LieA{q} \ldot m \to \LieA{q}^* \oplus \LieA{m}^*, \qquad \xi \ldot m \mapsto \mu^{-1} \ldot (\xi \ldot \mu) = \Lambda (\mu, \xi).
	\end{equation}
	Since \( \LieA{q} \) is a complement of \( \LieA{g}_\mu \), the map \( I_\mu \) is injective.
	In the case when \( G^*_\kappa = \LieA{g}^*_\kappa \) is endowed with the coadjoint action as the coconjugation action, \( I_\mu \) is given by
	\begin{equation}
		I_\mu: \LieA{q} \ldot m \to \LieA{q}^* \oplus \LieA{m}^*, \qquad \xi \ldot m \mapsto \CoadAction_\xi \mu,
	\end{equation}
	see~\eqref{eq:momentumMap:equivarianceInfinitesimal:usual}.
	Then, \( I_\mu \) takes values in \( \LieA{q}^* \), because \( \CoadAction_\xi \mu \in \LieA{g}_\mu^\perp = \LieA{q}^* \) for every \( \xi \in \LieA{g} \).
	If \( \LieA{q} \) is finite-dimensional, then a dimension count shows that \( I_\mu \) yields an isomorphism between \( \LieA{q} \ldot m \) and \( \LieA{q}^* \).
	Moreover, in this case, \( \difLog_m J \) restricts to a bijection between \( F \) and \( \LieA{m}^* \).
	We do not know if this observation generalizes to infinite-dimensional \( \LieA{q} \) and/or to other coconjugation actions.
\end{remark}

\begin{example}
	\label{ex::diffOneFormsOnSurface:gaugeActionMomentumMap}
	For the action of the group of gauge transformations on the space of connections on a principal \( \UGroup(1) \)-bundle \( P \) discussed in \cref{ex::diffOneFormsOnSurface:gaugeAction}, we found that the infinitesimal action \( \dif: \DiffFormSpace^0(M) \to \DiffFormSpace^1(M) \) has \( \difLog_\alpha \SectionMapAbb{J} = \dif: \DiffFormSpace^1(M) \to \DiffFormSpace^2(M) \) as its adjoint.
	The \( L^2 \)-orthogonal Hodge decompositions
	\begin{align}
		\label{eq:diffOneFormsOnSurface:gaugeActionMomentumMap:hodge1}
		\DiffFormSpace^1(M) &= 
			\dif \sFunctionSpace(M) \oplus \dif^* \DiffFormSpace^2(M) \oplus \deRCohomology^1(M, \R), \\
		\DiffFormSpace^2(M) &= 
			\dif \DiffFormSpace^1(M) \oplus \deRCohomology^2(M, \R)
	\end{align}
	show that the closedness conditions \( \LieA{g} \ldot m = (\LieA{g} \ldot m)^{\omega\omega} \) and \( \img \difLog_m J = (\img \difLog_m J)^{\perp\perp} \) hold for the case under consideration, \cf \cref{prop:symplecticFunctionalAnalysis:orthogonalToSymplecticOrthogonal}.
	The identities~\eqref{eq:bifurcationLemma:strong:ker} and~\eqref{eq:bifurcationLemma:strong:img} of the Strong Bifurcation Lemma are equivalent to
	\begin{align}
		(\dif \sFunctionSpace(M) \oplus \deRCohomology^1(M, \R))^\omega &= \dif \sFunctionSpace(M), \\
		\dif \DiffFormSpace^1(M) &= \deRCohomology^0(M, \R)^\perp.
	\end{align}
	These identities can also be verified by direct inspection using Hodge theory.

	Let us now determine the symplectic subspace \( \SectionSpaceAbb{E} \) of \cref{prop:bifurcationLemma:kernelMomentumMapIsSymplectic} for this example.
	The orbit through \( A \in \ConnSpace(P) \) of the \( \GauGroup(P) \)-action is given by
	\begin{equation}
		\GauGroup(P) \cdot A = A + \clZDiffFormSpace^1(M),
	\end{equation}
	where \( \clZDiffFormSpace^1(M) \) denotes the space of closed \( 1 \)-forms with integral periods.
	According to the Hodge decomposition~\eqref{eq:diffOneFormsOnSurface:gaugeActionMomentumMap:hodge1}, a natural choice of a slice \( \SectionSpaceAbb{S} \) at \( A \) is given by \( \SectionSpaceAbb{S} = A + \SectionSpaceAbb{U} \), where \( \SectionSpaceAbb{U} \) is a sufficiently small neighborhood of \( 0 \) in \( \dif^* \DiffFormSpace^2(M) \oplus \deRCohomology^1(M, \R) \).
	For the subspace defined in~\eqref{eq:bifurcationLemma:kernelMomentumMapIsSymplectic}, we hence obtain 
	\begin{equation}
		\SectionSpaceAbb{E} = \TBundle_A \SectionSpaceAbb{S} \intersect \ker \difLog_A \SectionMapAbb{J}
			= \bigl(\dif^* \DiffFormSpace^2(M) \oplus \deRCohomology^1(M, \R)\bigr) \intersect \ker \dif
			\isomorph \deRCohomology^1(M, \R).
	\end{equation}
	It is a simple exercise to verify that \( \SectionSpaceAbb{E} \) is a symplectic subspace, in accordance with \cref{prop:bifurcationLemma:kernelMomentumMapIsSymplectic}.
	Moreover, the restriction of the symplectic structure \( \omega \) as given in~\eqref{eq::diffOneFormsOnSurface:symplecticForm} on \( \TBundle_A \bigl(\ConnSpace(P)\bigr) \) to \( \SectionSpaceAbb{E} \) coincides with the intersection form on \( \deRCohomology^1(M, \R) \):
	\begin{equation}
		\label{eq:diffOneFormsOnSurface:gaugeActionMomentumMap:symplecticStructureOnE}
		\bigl(\equivClass{\alpha}, \equivClass{\beta}\bigr) \mapsto \int_M \alpha \wedge \beta.
		\qedhere
	\end{equation}
\end{example}

\subsection{Normal form of a momentum map}
\label{sec:momentumMap:normalFormMomentumMap}

Starting with the work of \textcite{Arms1981} on the Yang--Mills equation and of \textcite{FischerMarsdenEtAl1980} on the Einstein equation, it became clear that the solution spaces of field theories have singularities at points with internal symmetry.
Shortly thereafter, a similar relation between internal symmetries and singularities of the momentum map was established in \parencite{ArmsMarsdenEtAl1981}.
The occurrence of these singularities of momentum map level sets was later explained by \textcite{Marle1983,Marle1985} and \textcite{GuilleminSternberg1984}, who proved normal form theorems for momentum maps.
These considerations were in a finite-dimensional setting or had a formal nature in the sense that the problems of the infinite-dimensional context were largely ignored.
The aim of this section is to use and refine the construction of the equivariant normal form in \cref{sec:normalFormEquivariantMap} to establish a normal form result for equivariant momentum maps in infinite dimensions.

At this point, the reader might want to recall the definition of an equivariant normal form, \cf  \cref{def:normalFormEquivariantMap:abstractNormalFormEquivariantMap,def:normalFormEquivariantMap:normalFormEquivariantMap}.
In a similar spirit, we define the concept of a normal form of an equivariant momentum map.
\begin{defn}
	\label{def:momentumMap:abstractMGSnormalForm}
	Let \( \LieA{g} \) be a Lie algebra and let \( \kappa(\LieA{g}^*_\kappa, \LieA{g}) \) be a dual pair.
	An \emphDef{abstract Marle--Guillemin--Sternberg (MGS) normal form} relative to these data is a tuple \( (H, X, \LieA{g}^*_\kappa, \hat{J}, J_\singularPart; \bar{\omega}) \), where
	\begin{enumerate}
		\item 
			\( H \) is a Lie group whose Lie algebra \( \LieA{h} \) is a Lie subalgebra of \( \LieA{g} \),
		\item 
			\( (H, X, \LieA{g}^*_\kappa, \hat{J}, J_\singularPart) \) is an abstract equivariant normal form in the sense of \cref{def:normalFormEquivariantMap:abstractNormalFormEquivariantMap} such that \( \kappa \) restricts to a dual pair \( \kappa(\coker, \LieA{h}) \) between the subspace \( \coker \subseteq \LieA{g}^*_\kappa \) (occurring in~\cref{def:normalFormEquivariantMap:abstractNormalFormEquivariantMap}) and \( \LieA{h} \subseteq \LieA{g} \),
		\item 
			\( \bar{\omega} \) is a closed \( H \)-invariant \( 2 \)-form\footnotemark{} on \( U \intersect \ker \) such that, at the point \( 0 \in U \), the form \( \bar{\omega}_{0} \) on \( \ker \) is symplectic, where \( U \) and \( \ker \) are as in \cref{def:normalFormEquivariantMap:abstractNormalFormEquivariantMap},
			\footnotetext{That is, \( \bar{\omega}_x \) is an antisymmetric \( H \)-invariant bilinear form on \( \ker \) for every \( x \in U \intersect \ker \).}
		\item
			the momentum map identity
			\begin{equation}
				\label{eq:momentumMap:MGSnormalForm:momentumMapRel}
				\bar{\omega}_{x} \bigl(\xi \ldot x, w \bigr) + \kappa\bigl(\tangent_{x} J_\singularPart (w), \xi \bigr) = 0
			\end{equation}
			holds for all \( x \in U \intersect \ker \), \( w \in \ker \) and \( \xi \in \LieA{h} \).
	\end{enumerate}
	An abstract MGS normal form is called \emphDef{strong} if \( \bar{\omega}_x = \bar{\omega}_0 \) holds and \( J_\singularPart \) satisfies
	\begin{equation}
		\label{eq:momentumMap:MGSnormalForm:quadraticMomentumMap}
		\kappa\bigl(J_\singularPart(x), \xi\bigr) = \frac{1}{2} \bar{\omega}_{0} (x, \xi \ldot x)
	\end{equation}
	for all \( x \in U \intersect \ker \) and \( \xi \in \LieA{h} \).
\end{defn}
At this point, the elements of an MGS normal form are not connected to a symplectic action on a symplectic manifold; such a connection will be established in \cref{def:momentumMap:MGSnormalForm}.

The important additional property of an abstract MGS normal form in comparison to an abstract equivariant normal form is the fact that the singular part \( J_\singularPart \) takes values in the \( \kappa \)-dual of the Lie algebra of \( H \) and that it is tamed by the \( 2 \)-form \( \bar{\omega} \) using the momentum map identity~\eqref{eq:momentumMap:MGSnormalForm:momentumMapRel}.
For a constant form \( \bar{\omega}_x = \bar{\omega}_0 \), the equation~\eqref{eq:momentumMap:MGSnormalForm:quadraticMomentumMap} is the integrated version of~\eqref{eq:momentumMap:MGSnormalForm:momentumMapRel}.
Note that we do not require \( \bar{\omega} \) to be symplectic away from the origin.
Nonetheless, this is automatically the case if \( \ker \) is finite-dimensional, because then the space of invertible operators is open in the space of all linear maps.
Moreover, in this case, one can use the Darboux Theorem to pass from an MGS normal form to a strong one in the following sense.
\begin{prop}
	\label{prop:momentumMap:upgradeMGSnormalFormToStrong}
	Let \( (H, X, \LieA{g}^*_\kappa, \hat{J}, J_\singularPart; \bar{\omega}) \) be an abstract MGS normal form relative to a dual pair \( \kappa(\LieA{g}^*_\kappa, \LieA{g}) \).
	If the subspace \( \ker \) occurring in \cref{def:normalFormEquivariantMap:abstractNormalFormEquivariantMap} is finite-dimensional, then there exists an \( H \)-equivariant local diffeomorphism \( \psi \) of \( X \) such that \( (H, X, \LieA{g}^*_\kappa, \hat{J}, J_\singularPart \circ \psi; \psi^* \bar{\omega}) \) is a strong MGS normal form.
\end{prop}
\begin{proof}
	Since \( \ker \) is finite-dimensional and \( \bar{\omega}_{0} \) is symplectic, we may shrink \( U \) in such a way that \( \bar{\omega}_{x} \) is a symplectic form on \( \ker \) for all \( x \in U \intersect \ker \).
	The classical Darboux Theorem for finite-dimensional symplectic manifolds yields a local diffeomorphism \( \psi \) of \( U \intersect \ker \) satisfying \( \psi (0) = 0 \) and \( \tangent_0 \psi = \id_{\ker} \) such that
	\begin{equation}
		\psi^* \bar{\omega} = \bar{\omega}_{0}.
	\end{equation}
	By extending \( \psi \) to the whole of \( X \) using the identity map on \( \coimg \), we may regard \( \psi \) as a local diffeomorphism of \( X \).
	As \( H \) is a compact Lie group, we may furthermore choose \( \psi \) to be \( H \)-equivariant.
	The momentum map relation behaves naturally with respect to equivariant symplectomorphisms and thus \( J_\singularPart \circ \psi \) is a momentum map for the linear \( H \)-action on \( \ker \) with respect to the constant symplectic form \( \bar{\omega}_{0} \).
	Since, in finite dimensions, the momentum map is unique (up to a constant which is fixed by the condition \( J_\singularPart \circ \psi(0) = J_\singularPart(0) = 0 \)), we have
	\begin{equation}
		\kappa\bigl(J_\singularPart \circ \psi(x), \xi \bigr) = \frac{1}{2} \bar{\omega}_{0} \bigl(x, \xi \ldot x\bigr)
	\end{equation}
	for every \( x \in \ker \) and \( \xi \in \LieA{h} \).
	In summary, \( J_\singularPart \circ \psi \) is in the strong MGS normal form according to~\eqref{eq:momentumMap:MGSnormalForm:quadraticMomentumMap}.
\end{proof}
For an abstract MGS normal form \( (H, X, \LieA{g}^*_\kappa, \hat{J}, J_\singularPart; \bar{\omega}) \) and a Lie group \( K \) with \( H \subseteq K \), we define the \emphDef{local MGS normal form} map \( \normalForm{J}: K \times_{H} U \to K \times_{H} \LieA{g}^*_\kappa \) by
\begin{equation}
	\label{eq:momentumMap:MGSnormalForm}
	\normalForm{J}\bigl(\equivClass{k, x_1, x_2}\bigr) \defeq \equivClass[\big]{k, \hat{J}(x_2) + J_\singularPart(x_1, x_2)}
\end{equation}
for \( k \in K \), \( x_1 \in U \intersect \ker \) and \( x_2 \in U \intersect \coimg \), where \( \ker, \coimg \) and \( U \) are as in \cref{def:normalFormEquivariantMap:abstractNormalFormEquivariantMap}.

\begin{defn}
	\label{def:momentumMap:MGSnormalForm}
	Let \( G \) be a Lie group with a dual pair \( \kappa(G^*_\kappa, G) \), and let \( (M, \omega, G) \) be a symplectic \( G \)-manifold with equivariant momentum map \( J: M \to G^*_\kappa \).
	Let \( m \in M \).
	Assume that the stabilizer \( G_\mu \) of \( \mu = J(m) \in G^*_\kappa \) is a Lie subgroup of \( G \).
	We say that \( J \) can be \emphDef{brought into the MGS normal form \( (H, X, \LieA{g}^*_\kappa, \hat{J}, J_\singularPart; \bar{\omega}) \)} at \( m \) if \( H = G_m \) and if there exists a linear slice \( S \) at \( m \) for the \( G_\mu \)-action, a \( G_m \)-equivariant diffeomorphism \( \iota_S: X \supseteq U \to S \subseteq M \) and a \( G_m \)-equivariant chart \( \rho: G^*_\kappa \supseteq V' \to V \subseteq \LieA{g}^*_\kappa \) at \( \mu \), which bring \( J \) into the local MGS normal form map \( \normalForm{J} \) according to \cref{def:normalFormEquivariantMap:normalFormEquivariantMap} and for which the restriction of \( \iota_S^* \omega \) to \( \ker \) coincides with \( \bar{\omega} \).
\end{defn}
If \( J \) can be brought into the MGS normal form \( (G_m, X, \LieA{g}^*_\kappa, \hat{J}, J_\singularPart; \bar{\omega}) \) at \( m \in M \), then according to \cref{def:normalFormEquivariantMap:normalFormEquivariantMap} there exists a commutative diagram of the type
\begin{equationcd}[label=eq:momentumMap:bringIntoNormalForm]
	M
		\to[r, "J"] 
	& G^*_\kappa 
	\\
	\mathllap{G_\mu \times_{G_m} X \supseteq {}} G_\mu \times_{G_m} U
		\to[r, "\normalForm{J}"]
		\to[u, "{\chi^\tube} \, \circ \, {(\id_{G_\mu} \times \iota_S)}"]
	& G_\mu \times_{G_m} V \mathrlap{{}\subseteq G_\mu \times_{G_m} \LieA{g}^*_\kappa,}
		\to[u, "\rho^\tube" swap]
\end{equationcd}
where \( \normalForm{J} \) was defined in~\eqref{eq:momentumMap:MGSnormalForm} with \( K = G_\mu \).

\begin{remark}
	Assume that the equivariant momentum map \( J: M \to G^*_\kappa \) can be brought into an MGS normal form \( (H, X, \LieA{g}^*_\kappa, \hat{J}, J_\singularPart; \bar{\omega}) \) at \( m \in M \) using the maps \( \chi^T: G_\mu \times_{G_m} S \to M \), \( \iota_S: X \supseteq U \to S \) and \( \rho: G^*_\kappa \supseteq V' \to V \subseteq \LieA{g}^*_\kappa \).
	Then, under the natural isomorphisms \( \tangent_0 \iota_S: X \to \TBundle_m S \) and \( \tangent_\mu \rho: \TBundle_\mu G^*_\kappa \to \LieA{g}^*_\kappa \), the Witt--Artin decomposition (\cf \cref{rem:symplectic:bifurcationLemma:decomposition}) yields the following identification of the spaces occurring in the MGS normal form:
	\begin{equation}
		\TBundle_m M = 
			\LieA{g}_\mu \ldot m
			\oplus
			\overbrace{
				\underbrace{E_{}}_{\ker}
				\oplus
				\underbrace{
					F
					\oplus
					\LieA{q} \ldot m
					}_{\coimg}
				}^{X}
		\end{equation}
	and
	\begin{equation}
		\LieA{g}^*_\kappa = \underbrace{\LieA{q}^* \oplus \LieA{m}^*_{}}_{\img} \oplus \underbrace{\LieA{g}_m^*}_{\coker}.
	\end{equation}
	Moreover, under these identifications, suppressing \( \iota_S \), the diagram~\eqref{eq:momentumMap:bringIntoNormalForm} takes the form
	\begin{equationcd}[label=eq:momentumMap:MGSnormalForm:bringIntoNormalForm]
		M
			\to[r, "J"]
			& G^*_\kappa
		\\
		G_\mu \times_{G_m} (E \oplus F \oplus \LieA{q} \ldot m)
			\to[u, "\chi^\tube"]
			\to[r, "\normalForm{J}"]
			& G_\mu \times_{G_m} \LieA{g}^*_\kappa \, ,
			\to[u, "\rho^\tube", swap]
	\end{equationcd}
	where \( \normalForm{J} \) is given by \( \normalForm{J}(\equivClass{g, e, f, \xi \ldot m}) = \equivClass{g, \hat{J}(f, \xi \ldot m) + J_\singularPart(e, f, \xi)} \).
	Furthermore, the restriction of the symplectic form \( \omega_m \) on \( \TBundle_m M \) to \( E \) is identified with the symplectic form \( \bar{\omega}_0 \) on \( \ker \).
\end{remark}

\begin{remark}[Classical Marle--Guillemin--Sternberg normal form]
	\label{rem:momentumMap:MGSnormalForm:comparison}
	The classical Marle--Guillemin--Sternberg normal form in finite dimensions (see, \eg, \parencite[Theorem~7.5.5]{OrtegaRatiu2003}) has a slightly different form and can be summarized in the following commutative diagram:
	\begin{equationcd}
		M
			\to[r, "J"]
			& \LieA{g}^*_\kappa
		\\
		G \times_{G_m} (E \times \LieA{m}^*)
			\to[u]
			\to[r, "J_{\textnormal{MGS}}"]
			& G \times_{G_m} \LieA{g}^*_\kappa \, ,
			\to[u, "\Lambda", swap] 
	\end{equationcd}
	where \( J_{\textnormal{MGS}} \) and \( \Lambda \) are defined by 
	\begin{equation}
		J_{\textnormal{MGS}}(\equivClass{g, e, \eta}) = \equivClass{g,  \eta + J_E (e)}
	\end{equation}
	and \( \Lambda(\equivClass{g, \nu}) = \CoAdAction_{g^{-1}} (\mu + \nu) \), respectively.
	Here, \( J_E: E \to \LieA{g}_m^* \) is the quadratic momentum map for the symplectic linear \( G_m \)-action on \( E \), see \cref{ex:momentumMap:linearCase}.

	The most important difference is that the classical MGS normal form splits off the \( G \)-action whereas the focus of our normal form lies on the \( G_\mu \)-action. 
	With view towards symmetry reduction, our formulation has advantages over the traditional approach, because in the context of symplectic reduction the group \( G_\mu \) and not \( G \) is the major player.
	In the classical MGS normal form \( \LieA{m}^* \) is identified with \( F \) via the map \( \LieA{m}^* \ni \nu \mapsto f \in F \) implicitly defined by
	\begin{equation}
		\kappa(\nu, \xi) = \omega_m (\xi \ldot m, f)
	\end{equation}
	for \( \xi \in \LieA{m} \).
	As already noted in \cref{rem:symplectic:bifurcationLemma:decomposition}, we do not know whether this isomorphism generalizes to infinite dimensions.
	Furthermore, in the classical MGS normal form the nonlinear term \( J_E \) only depends on points in \( \ker \) while our normal form is weaker in this regard, allowing the singular part \( J_\singularPart \) to additionally depend on points in \( \LieA{q} \) and \( F \).
	Finally, the classical Marle--Guillemin--Sternberg normal form brings both the symplectic structure and the momentum map simultaneously into a normal form, while we are mostly concerned with the momentum map.
	In fact, we do not have any control over the symplectic form in the \( \coimg \)-direction.
\end{remark}

The following result is of fundamental importance for the construction of an MGS normal form.
\begin{lemma}
	\label{prop:momentumMap:upgradeEquivNormalFormToMGS}
	Let \( (M, \omega, G) \) be a symplectic \( G \)-manifold with equivariant momentum map \( J: M \to G^*_\kappa \).
	Let \( m \in M \) be such that \( \LieA{g} \ldot m \) is symplectically closed.
	Assume that \( J \) can be brought into the \( G_\mu \)-equivariant normal form \( (G_m, X, \LieA{g}^*_\kappa, \hat{J}, J_\singularPart) \) at \( m \) using the maps \( \chi^\tube: G_\mu \times_{G_m} S \to M \) and \( \rho: G^*_\kappa \supseteq V' \to \LieA{g}^*_\kappa \) in the sense of \cref{def:normalFormEquivariantMap:normalFormEquivariantMap}.
	If, for all \( \nu \in V' \subseteq G^*_\kappa \), the left derivative
	\begin{equation}
		\tangentLeft_\nu \rho: \LieA{g}^*_\kappa \to \LieA{g}^*_\kappa, \qquad \eta \mapsto \tangent_\nu \rho (\nu \ldot \eta)
	\end{equation}
	of \( \rho \) at \( \nu \) restricts to the identity on \( \LieA{g}_m^* \), then \( J \) can be brought into the MGS normal form \( (G_m, X, \LieA{g}^*_\kappa, \hat{J}, J_\singularPart; \bar{\omega}) \), where \( \bar{\omega} = \restr{(\iota_S^* \omega)}{\ker} \).
\end{lemma}
\begin{proof}
	First, note that the commutative diagram~\eqref{eq:momentumMap:bringIntoNormalForm} yields by restriction the following commutative diagram:
	\begin{equationcd}[label=momentumMap:upgradeEquivNormalFormToMGS:commutativeOnS]
		M
			\to[r, "J"] 
		& G^*_\kappa 
		\\
		U \intersect \ker
			\to[r, "J_\singularPart"]
			\to[u, "\iota_S"]
		& \LieA{g}_m^* \,.
			\to[u, "\rho^{-1}" swap]
	\end{equationcd}
	Moreover, a straightforward calculation shows that \( \difLog_{\rho(\nu)} (\rho^{-1}) \circ \tangentLeft_\nu \rho = \id_{\LieA{g}^*_\kappa} \). 
	Thus, \( \difLog_{\rho(\nu)} (\rho^{-1}): \LieA{g}^*_\kappa \to \LieA{g}^*_\kappa \) restricts to the identity on \( \LieA{g}_m^* \) for all \( \nu \in V' \).
	Hence, for every \( x \in U \intersect \ker \), \( w \in \ker \) and \( \xi \in \LieA{g}_m \), we obtain
	\begin{equation}\begin{split}
		\kappa\bigl(\tangent_x J_\singularPart (w), \xi \bigr)
			&= \kappa\bigl( \difLog_{\rho \, \circ J \circ \iota_S (x)} (\rho^{-1}) \circ \tangent_x J_\singularPart (w), \xi \bigr)
			\\
			&= \kappa\bigl( \difLog_{\iota_S (x)} J \circ \tangent_x \iota_S (w), \xi \bigr)
			\\
			&= \omega_{\iota_S (x)} \bigl(\tangent_x \iota_S (w), \xi \ldot \iota_S (x) \bigr)
			\\
			&= \omega_{\iota_S (x)} \bigl(\tangent_x \iota_S (w), \tangent_x \iota_S (\xi \ldot x) \bigr)
			\\
			&= (\iota_S^* \omega)_{x} \bigl(w, \xi \ldot x \bigr),
	\end{split}\end{equation}
	where we used the momentum map relation and \( G_m \)-equivariance of \( \iota_S \).
	Thus, if we let \( \bar{\omega} = \restr{(\iota_S^* \omega)}{\ker} \), then~\eqref{eq:momentumMap:MGSnormalForm:momentumMapRel} holds.
	Moreover, \( \bar{\omega}_0 \) coincides with the restriction of \( \omega_m \) to the subspace \( E = \TBundle_m S \intersect \ker \difLog_m J \) (under the isomorphism \( \tangent_0 \iota_S: X \to \TBundle_m S \)).
	Hence, \( \bar{\omega}_0 \) is a symplectic form according to \cref{prop:bifurcationLemma:kernelMomentumMapIsSymplectic}.
	In summary, \( J \) is brought into the MGS normal form \( (G_m, X, \LieA{g}^*_\kappa, \hat{J}, J_\singularPart; \bar{\omega}) \) using the maps \( \chi^\tube \) and \( \rho \).
\end{proof}
The assumption in \cref{prop:momentumMap:upgradeEquivNormalFormToMGS} concerning the chart \( \rho \) is rather easy to satisfy.
For example, in the case when \( G^*_\kappa = \LieA{g}^*_\kappa \), we can simply take \( \rho \) to be the translation by \( \mu \).
More generally, if \( G^*_\kappa \) has an exponential map \( \exp_{G^*_\kappa} \) which is a local diffeomorphism, then \( \rho \defeq \exp_{G^*_\kappa}^{-1} \) satisfies \( \tangentLeft_\nu \rho = \id_{\LieA{g}^*_\kappa} \), because \( G^*_\kappa \) is abelian.
This covers the cases we are interested in and thus, for simplicity, we will assume in the sequel that \( \tangentLeft_\nu \rho = \id_{\LieA{g}^*_\kappa} \).
\begin{remark}
	The proof of \cref{prop:momentumMap:upgradeEquivNormalFormToMGS} shows that the momentum map relation~\eqref{eq:momentumMap:MGSnormalForm:momentumMapRel} holds even for all \( x \in U \) and not just \( x \in U \intersect \ker \) (with \( \bar{\omega} \) being a \( 2 \)-form on \( U \)).
	If \( \ker \) is finite-dimensional, this observation can be combined with a slightly modified version of \cref{prop:momentumMap:upgradeMGSnormalFormToStrong} to show that \( J_\singularPart \) satisfies
	\begin{equation}
		\kappa\bigl(J_\singularPart(x_1, x_2), \xi\bigr) = \frac{1}{2} \bar{\omega}_{(0, x_2)} (x_1, \xi \ldot x_1)
	\end{equation}
	for all \( (x_1, x_2) \in U \) and \( \xi \in \LieA{g}_m \).
	Here, we used the property \( J_\singularPart(0, x_2) = 0 \) coming from the equivariant normal form to fix the constant of integration. 
	We see that \( J_\singularPart(x_1, x_2) \) is quadratic in \( x_1 \) for all \( x_2 \).
	In comparison to the classical MGS normal form, the only difference is that the quadratic form \( x_1 \mapsto \bar{\omega}_{(0, x_2)} (\xi \ldot x_1, x_1) \) on \( \ker \) still depends on \( x_2 \).
\end{remark}

With the help of \cref{prop:momentumMap:upgradeEquivNormalFormToMGS}, we can upgrade an equivariant normal form of \( J \) in the sense of \cref{def:normalFormEquivariantMap:normalFormEquivariantMap} to an MGS normal form.
\begin{thm}[Marle--Guillemin--Sternberg Normal Form]
	\label{prop:momentumMap:MGSnormalForm}
	Let \( (M, \omega, G) \) be a symplectic \( G \)-manifold with equivariant momentum map \( J: M \to G^*_\kappa \).
	Let \( m \in M \) and \( \mu = J(m) \in G^*_\kappa \).
	Assume that there exists a chart \( \rho: G^*_\kappa \supseteq V' \to \LieA{g}^*_\kappa \) at \( \mu \) which linearizes the \( G_m \)-action on \( G^*_\kappa \) and which satisfies \( \tangentLeft_\nu \rho = \id_{\LieA{g}^*_\kappa} \) for every \( \nu \in V' \).
	Moreover, assume that \( \LieA{g} \ldot m \) is symplectically closed and that the image of \( \difLog_m J \) is weakly closed.
	If \( J \) satisfies the assumptions of \cref{prop:normalFormEquivariantMap:abstract} (or of one of the other equivariant normal form theorems of \cref{sec:normalFormEquivariantMap}), then \( J \) can be brought into an MGS normal form.
\end{thm}
\begin{proof}
	\Cref{prop:normalFormEquivariantMap:abstract} (or its other variations in \cref{sec:normalFormEquivariantMap}) shows that \( J \) can be brought into an equivariant normal form in a neighborhood of \( m \) by deforming the chart \( \rho \) on \( G^*_\kappa \) using the local diffeomorphism \( \phi \) defined in~\parencite[Equation~3.7]{DiezRudolphKuranishi}.
	For completeness, we sketch the construction of this diffeomorphism.
	By choosing a chart on the \( G_\mu \)-slice \( S \), we can locally identify \( S \) with a neighborhood \( U \) of \( 0 \) in a locally convex space \( X \) and think of \( J \) as a smooth map \( J: U \to \LieA{g}^*_\kappa \).
	Then, \( \difLog_m J \) induces a linear map \( T: X \to \LieA{g}^*_\kappa \).	
	Since \( T \) is regular by assumption, there exist topological decompositions \( X = \ker T \oplus \coimg T \) and \( Y = \coker T \oplus \img T \).
	Moreover, the restriction \( \hat{T}: \coimg T \to \img T \) of \( T \) is a topological isomorphism.
	Define the smooth map \( \psi: X \supseteq U \to X \) by 
	\begin{equation}
		\label{eq:normalFormMap:banach:deformDomain}
	 	\psi(x_1, x_2) = \bigl(x_1, \hat{T}^{-1} \circ \pr_{\img T} \circ J(x_1, x_2)\bigr),
	\end{equation}
	with \( x_1 \in \ker T \) and \( x_2 \in \coimg T \).
	The assumptions ensure that \( \psi: U \to \psi(U) \) is a diffeomorphism after possibly shrinking \( U \).
	Let \( V = \rho(V') \subseteq \LieA{g}^*_\kappa \) with \( V' \) as in the assumptions of the theorem.
	Define the smooth map \( \phi: Y \supseteq V \to Y \) by
	\begin{equation}
		\label{eq:normalFormMap:banach:deformTarget}
		\phi(y_1, y_2) = \Bigl(y_1 + \pr_{\coker T} \circ J \circ \psi^{-1}(0, \hat{T}^{-1} y_2), y_2\Bigr)
	\end{equation}
	with \( y_1 \in \coker T \) and \( y_2 \in \img T \).
	A simple inspection shows that the restriction of \( \tangent_{(y_1, y_2)} \phi: Y \to Y \) to \( \coker T \) is the identity map for all \( (y_1, y_2) \in V \).
	Since \( \img \difLog_m J = \img T \) is a weakly closed subspace of \( \LieA{g}^*_\kappa \), the Strong Bifurcation \cref{prop::bifurcationLemma:strong} shows that in our setting \( \coker T \) is identified with \( \LieA{g}_m^* \).
	Hence, the chart \( \rho' = \rho \circ \phi^{-1} \) satisfies the assumptions of \cref{prop:momentumMap:upgradeEquivNormalFormToMGS}, which implies that \( J \) can be brought into an MGS normal form.
\end{proof}

We note that, in finite dimensions and for classical momentum maps, all assumptions of \cref{prop:momentumMap:MGSnormalForm} are met automatically except for properness of the action.
\begin{remark}
	\label{rem:momentumMap:MGSnormalForm:notDarboux}
	The proof of \cref{prop:momentumMap:MGSnormalForm} is constructive in the sense that the diffeomorphisms bringing \( J \) into the MGS normal form have explicit and relatively simple expressions.
	In contrast, the usual proof of the classical MGS normal form theorem relies on the relative Darboux Theorem, which makes it difficult\footnote{Even in simple cases, it is usually impossible to integrate in closed-form the differential flow equation underlying Moser's method.} to determine the deforming diffeomorphism (both analytically as well as numerically).
	Thus, one would expect that our proof of the MGS normal form theorem can be used to design new numerical discretization algorithms which preserve the momentum map geometry. 
	This will be explored in further work.
\end{remark}

Recall from \cref{prop:bifurcationLemma:kernelMomentumMapIsSymplectic} that \( \ker \isomorph E \) can be identified with the middle homology group of the complex
\begin{equationcd}
	0 \to[r] & \LieA{g}_\mu \to[r, "\tangent_e \Upsilon_m"] & \TBundle_m M \to[r, "\difLog_m J"] & \LieA{g}^*_\kappa \to[r] & 0,
\end{equationcd}
where \( \Upsilon_m: G \to M \) denotes the orbit map as before.
If this complex is Fredholm, then \( \ker \isomorph E \) is finite-dimensional and every MGS normal form can be upgraded to a strong MGS normal form according to \cref{prop:momentumMap:upgradeMGSnormalFormToStrong}.
This is, in particular, the case for elliptic actions so that \cref{prop:momentumMap:MGSnormalForm,prop:normalFormEquivariantMap:elliptic} yield the following MGS normal form theorem.
For the definition of geometric tame Fréchet manifolds, see the remark before \cref{prop:normalFormEquivariantMap:elliptic}.
\begin{thm}[MGS Normal Form --- elliptic version]
	\label{prop:momentumMap:MGSnormalForm:elliptic}
	Let \( G \) be a tame Fréchet Lie group, let \( G^*_\kappa \) be a geometric dual Lie group and let \( (M, \omega, G) \) be a symplectic geometric tame Fréchet \( G \)-manifold with equivariant momentum map \( J: M \to G^*_\kappa \).
	Let \( m \in M \) and \( \mu = J(m) \).
	Assume that the following conditions hold:
	\begin{thmenumerate}
		\item
			The stabilizer subgroup \( G_\mu \) of \( \mu \) is a geometric tame Fréchet Lie subgroup of \( G \).
		\item 
			The induced \( G_\mu \)-action on \( M \) is proper and admits a slice \( S \) at \( m \).
		\item
			The induced \( G_m \)-action on \( G^*_\kappa \) can be linearized at \( \mu \) using a chart \( \rho: G^*_\kappa \supseteq V' \to \LieA{g}^*_\kappa \) satisfying \( \tangentLeft_\nu \rho = \id_{\LieA{g}^*_\kappa} \) for every \( \nu \in V' \).
		\item
			The chain
			\begin{equationcd}[label=eq:momentumMap:MGSnormalForm:elliptic:chain]
				0 \to[r] 
					& \LieA{g}_\mu
						\to[r, "\tangent_e \Upsilon_s"]
					& \TBundle_s M
						\to[r, "\difLog_s J"]
					& \LieA{g}^*_\kappa
						\to[r]
					& 0
			\end{equationcd}
			of linear maps parametrized by \( s \in S \) is a chain of differential operators, which constitute an elliptic complex at \( m \).
		\item
			The subspace \( \LieA{g} \ldot m \subseteq \TBundle_m M \) is symplectically closed and the image of \( \difLog_m J \) is weakly closed in \( \LieA{g}^*_\kappa \). 
	\end{thmenumerate}
	Then, \( J \) can be brought into a strong MGS normal form at \( m \). 
\end{thm}
\begin{example}
	\label{ex::diffOneFormsOnSurface:momentumMapNormalForm}
	Let us use \cref{prop:momentumMap:MGSnormalForm:elliptic} to show that the momentum map for the action of the group of gauge transformations on the space of connections on a principal \( \UGroup(1) \)-bundle \( P \to M \) over a surface \( M \) can be brought into an MGS normal form.
	By \cref{ex::diffOneFormsOnSurface:gaugeAction}, the map \( \SectionMapAbb{J}: \ConnSpace(P) \to \DiffFormSpace^2(M) \) assigning the curvature \( F_A \) to a connection \( A \in \ConnSpace(P) \) is the momentum map for the \( \GauGroup(P) \)-action on \( \ConnSpace(P) \).
	Let \( \mu \in \DiffFormSpace^2(M) \).
	Since \( \GauGroup(P) \) acts trivially on \( \DiffFormSpace^2(M) \), the stabilizer group of \( \mu \) coincides with the whole group \( \GauGroup(P) \), which clearly is a geometric tame Fréchet Lie group.
	As discussed in \cref{ex::diffOneFormsOnSurface:gaugeActionMomentumMap}, the \( \GauGroup(P) \)-action is proper and admits a slice \( \SectionSpaceAbb{S} \) at every \( A \in \ConnSpace(P) \) of the form \( \SectionSpaceAbb{S} = A + \SectionSpaceAbb{U} \), where \( \SectionSpaceAbb{U} \) is a sufficiently small neighborhood of \( 0 \) in \( \dif^* \DiffFormSpace^2(M) \oplus \deRCohomology^1(M, \R) \).
	In this case, the chain~\eqref{eq:momentumMap:MGSnormalForm:elliptic:chain} is independent of \( B \in \SectionSpaceAbb{S} \) and coincides (up to a sign) with the elliptic de Rham complex
	\begin{equationcd}
		0
			\to[r]
		& \DiffFormSpace^0(M)
			\to[r, "\dif"]
		& \DiffFormSpace^1(M)
			\to[r, "\dif"]
		& \DiffFormSpace^2(M)
			\to[r]
		& 0.
	\end{equationcd}
	The closedness assumptions of the subspaces \( \LieA{g} \ldot m \) and \( \img \difLog_m J \) were checked in \cref{ex::diffOneFormsOnSurface:gaugeActionMomentumMap} for the model under consideration.
	In summary, \cref{prop:momentumMap:MGSnormalForm:elliptic} implies that \( J \) can be brought into a strong MGS normal form.
	In the present setting, the normal form is particularly simple.
	Indeed, for every connection \( A \), the curvature map restricted to the slice \( \SectionSpaceAbb{S} \) at \( A \) is already in the MGS normal form 
	\begin{equation}
		\SectionMapAbb{J}: \SectionSpaceAbb{S} = A + \SectionSpaceAbb{U} \mapsto \DiffFormSpace^2(M), \qquad A + \alpha_1 + \alpha_2 \mapsto F_A + \dif \alpha_1,
	\end{equation}
	where \( \alpha_1 \in \dif^* \DiffFormSpace^2(M) \) and \( \alpha_2 \in \deRCohomology^1(M, \R) \).
	We read off that the linear part of \( \SectionMapAbb{J} \) is given by the isomorphism \( \dif: \dif^* \DiffFormSpace^2(M) \to \dif \DiffFormSpace^1(M) \) and that the singular part \( \SectionMapAbb{J}_\singularPart \) vanishes.
	The latter should not come as a surprise as the \( \GauGroup(P) \)-action on \( \ConnSpace(P) \) has only one orbit type, namely \( \UGroup(1) \).
\end{example}

The linear action of a compact group on a symplectic vector space provides another situation where we can show the existence of an MGS normal form.
To see this, let \( (X, \omega) \) be a symplectic Fréchet space endowed with a continuous linear symplectic action of the compact finite-dimensional Lie group \( G \).
Recall from \cref{ex:momentumMap:linearCase}, that in this case, the momentum map \( J: X \to \LieA{g}^*_\kappa \) for the \( G \)-action is given by
\begin{equation}
	\label{eq:momentumMap:linearActionQuadraticMomentumMap}
	\kappa(J(x), \xi) = \frac{1}{2} \omega(x, \xi \ldot x),
\end{equation}
for \( \xi \in \LieA{g} \).
Moreover, \( J \) is equivariant with respect to the coadjoint action of \( G \) on \( \LieA{g}^*_\kappa \).
\begin{thm}
	\label{prop:momentumMap:MGSnormalForm:linear}
	Let \( (X, \omega) \) be a symplectic Fréchet space endowed with a continuous linear symplectic action of the compact finite-dimensional Lie group \( G \).
	Then, the equivariant momentum map \( J: X \to \LieA{g}^*_\kappa \) defined in~\eqref{eq:momentumMap:linearActionQuadraticMomentumMap} can be brought into a strong MGS normal form at every point of \( X \).
\end{thm}
\begin{proof}
	Let \( x \in X \) and let \( \mu = J(x) \).
	The stabilizer subgroup \( G_\mu \) is a compact Lie subgroup of \( G \), because \( G \) is compact and finite-dimensional.
	Therefore, the \( G_\mu \)-action on \( X \) is proper and, by \parencite[Theorem~3.15]{DiezSlice}, it admits a linear slice at \( x \). 
	Since \( \LieA{g} \) is finite-dimensional, the image of \( \tangent_x J \) is (weakly) closed and the orbit \( \LieA{g} \ldot x \) is symplectically closed due to \cref{prop:symplecticFunctionalAnalysis:finiteDimSubspaceSymplectiallyClosed}.
	Define \( \rho: \LieA{g}^*_\kappa \to \LieA{g}^*_\kappa \) by \( \rho(\nu) = \nu - \mu \).
	Clearly, \( \rho \) is \( \CoAdAction_{G_\mu} \)-invariant and satisfies \( \tangent_\nu \rho = \id_{\LieA{g}^*_\kappa} \) for every \( \nu \in \LieA{g}^*_\kappa \).
	Now, \cref{prop:momentumMap:MGSnormalForm} and \parencite[Theorem~3.7]{DiezRudolphKuranishi} imply that \( J \) can be brought into an MGS normal form at \( x \).
	Since the symplectic form \( \omega \) on \( X \) is constant, an argument similar to the one in the proof of \cref{prop:momentumMap:upgradeMGSnormalFormToStrong} shows that \( J_\singularPart \) satisfies~\eqref{eq:momentumMap:MGSnormalForm:quadraticMomentumMap}.
\end{proof}

\section{Symplectic reduction}
\label{sec:symplecticReduction}

In this section, we are concerned with the geometric structure of the symplectic quotient
\begin{equation}
	\check{M}_\mu \defeq J^{-1}(\mu) \slash G_\mu.
\end{equation}
In finite dimensions, the geometry of \( \check{M}_\mu \) attracted a lot of attention starting with the work of \textcite{MarsdenWeinstein1974,Meyer1973} on the regular case and of \textcite{SjamaarLerman1991,ArmsGotayEtAl1990} on the singular case.
By extending these classical structure theorems to our infinite-dimensional context, we will show that the decomposition of \( \check{M}_\mu \) into orbit types is a stratification and that each stratum carries a natural symplectic structure.
These results are obtained under the standing assumption that \( J \) can be brought into an MGS normal form.
Moreover, we will see that the frontier condition of the stratification relies on the strong MGS normal form combined with a certain approximation property.
At the end of the section, we will comment on the regular case.

\subsection{Singular symplectic reduction}

We continue to work in the setting of the previous section.
Let \( (M, \omega, G) \) be a symplectic \( G \)-manifold with equivariant momentum map \( J: M \to G^*_\kappa \), where \( \kappa(G^*_\kappa, G) \) is a dual pair and \( G^*_\kappa \) carries a given coconjugation action.
Recall from \cref{sec:slices} that \( M \) decomposes into the orbit type subsets \( M_{(H)} \) and that \( M_{(H)} \) are submanifolds of \( M \) under the assumption that the group action admits a slice at every point.
The following is the symplectic version of this fact.
\begin{prop}
	\label{prop:symplecticReduction:orbitTypeManifold}
	Let \( (M, \omega, G) \) be a symplectic \( G \)-manifold with equivariant momentum map \( J: M \to G^*_\kappa \).
	Let \( \mu \in G^*_\kappa \), and assume that \( G_\mu \) is a Lie subgroup of \( G \) acting properly on \( M \).
	If \( J \) can be brought into an MGS normal form at every point of \( M_\mu \defeq J^{-1}(\mu) \), then, for every orbit type \( (H) \) of the \( G_\mu \)-action on \( M_\mu \), the set
	\begin{equation}
		M_{(H), \mu} \defeq M_{(H)} \intersect J^{-1}(\mu)
	\end{equation}
	is a smooth submanifold of \( M \).
	Moreover, there exists a unique smooth manifold structure on the quotient
	\begin{equation}
		\check{M}_{(H), \mu} \defeq M_{(H), \mu} \slash G_\mu
	\end{equation}
	such that the natural projection \( \pi_{(H), \mu}: M_{(H), \mu} \to \check{M}_{(H), \mu} \) is a smooth surjective submersion.
	Furthermore, for every \( m \in M_{(H), \mu} \), the projection \( \tangent_m \pi_{(H), \mu} \) restricts to an isomorphism
	\begin{equation}
		\tangent_m \pi_{(H), \mu}: E_{G_m} \to \TBundle_{\equivClass{m}} \check{M}_{(H), \mu},
	\end{equation}
	where \( E_{G_m} \) denotes the set of fixed points of the normal space \( E = \TBundle_m S \intersect \ker \difLog_m J \subseteq \TBundle_m M \) under the \( G_m \)-action and \( S \) is the \( G_\mu \)-slice in the MGS normal form at \( m \). 
\end{prop}
\begin{proof}
	Let \( m \in M_{(H), \mu} \).
	By assumption, \( J \) can be brought into an MGS normal form \( (G_m, X, \LieA{g}^*_\kappa, \hat{J}, J_\singularPart; \bar{\omega}) \) using maps \( \chi^\tube: G_\mu \times_{G_m} S \to M \) and \( \rho: G^*_\kappa \supseteq V' \to \LieA{g}^*_\kappa \).
	Thus, we have
	\begin{equation}
		J \circ \chi^\tube \circ (\id_{G_\mu} \times \iota_S) \bigl(\equivClass{g, x_1, x_2}\bigr)
			= \rho^\tube \bigl(\equivClass{g, \hat{J}(x_2) + J_\singularPart(x_1, x_2)}\bigr)
	\end{equation}
	for \( g \in G_\mu \) and \( (x_1, x_2) \in U \subseteq X \), where \( \iota_S: X \supseteq U \to S \) is the slice diffeomorphism.
	Since \( (\rho^\tube)^{-1}(\mu) = G_\mu \times_{G_m} \set{0} \) and since \( \hat{J} \) is an isomorphism, we find that \( \chi^\tube \) identifies the level set \( M_\mu = J^{-1}(\mu) \) with the set
	\begin{equation}
		\label{eq:symplecticReduction:orbitTypeManifold:locallyMmu}
		G_\mu \times_{G_m} \set*{x_1 \in U \intersect \ker \given J_\singularPart(x_1) = 0}. 	
	\end{equation}
	Moreover, \( \chi^\tube \) is \( G_\mu \)-invariant and so \( M_{(H), \mu} \) is locally identified with
	\begin{equation}
		\label{eq:symplecticReduction:orbitTypeManifold:locallyMHmu}
		G_\mu \times_{G_m} \bigl( U_{(H)} \intersect \ker{} \intersect J_\singularPart^{-1}(0) \bigr),
	\end{equation}
	where for \( U_{(H)} \) the stabilizer is taken with respect to the linear \( G_m \)-action but conjugation is understood with respect to \( G_\mu \).
	Since the \( G_\mu \)-action is proper, we have \( S_{(H)} = S_{G_m} \), see \parencite[Proposition~2.13]{DiezSlice}.
	Therefore, \( U_{(H)} = U_{G_m} \).
	Hence, for every \( x_1 \in U_{(H)} \intersect \ker \), \( v \in \ker \) and \( \xi \in \LieA{g}_m \), we find
	\begin{equation}
		\kappa\bigl( \tangent_{x_1} J_\singularPart (v), \xi \bigr)
			= - \bar{\omega}_{x_1} (\xi \ldot x_1, v)
			= 0
	\end{equation}
	and so \( \tangent_{x_1} J_\singularPart (v) = 0 \).
	By possibly shrinking \( U \), we may assume that \( U_{G_m} \intersect \ker \) is convex.
	Now the Fundamental Theorem of Calculus \parencite[Proposition I.2.3.2]{Neeb2006} implies
	\begin{equation}
		J_\singularPart (x_1) 
			= J_\singularPart (0) + \int_0^1 \tangent_{(t x_1)} J_\singularPart (x_1) \, \dif t
			= 0,
	\end{equation}
	for every \( x_1 \in U_{G_m} \intersect \ker \).
	Thus, in summary, \( M_{(H), \mu}  \) is locally identified with
	\begin{equation}
		G_\mu \times_{G_m} \bigl( U_{(H)} \intersect \ker{} \intersect J_\singularPart^{-1}(0) \bigr)
			= G_\mu \times \bigl( U_{G_m} \intersect \ker \bigr),
	\end{equation}
	from which we conclude that \( M_{(H), \mu} \) is a smooth submanifold of \( M \).
	Moreover, \( \check{M}_{(H), \mu} \) is locally identified with \( U_{G_m} \intersect \ker \) and so it carries a smooth manifold structure modeled on \( \ker_{G_m} \).
	In these coordinates, the projection \( \pi_{(H), \mu}: M_{(H), \mu} \to \check{M}_{(H), \mu} \) corresponds to the projection onto the second factor and thus it is a smooth submersion.
	Finally, note that the \( G_m \)-equivariant isomorphism \( \tangent_0 \iota_S: X \to \TBundle_m S \) identifies \( \ker_{G_m} \) with \( E_{G_m} \).
\end{proof}

\begin{remark}
	\label{prop:symplecticReduction:normalFormLevelSetQuadraticSingularities}
	In contrast to the regular case, the level set \( M_\mu \) is in general not a smooth manifold.
	Indeed, as we have seen in~\eqref{eq:symplecticReduction:orbitTypeManifold:locallyMmu}, \( M_\mu \) is locally modeled on spaces of the type \( \ker \intersect J_\singularPart^{-1}(0) \).
	If the MGS normal form is strong, then \( \restr{(J_\singularPart)}{\ker} \) is a quadratic form and thus \( M_\mu \) has conic  singularities\footnote{The fact that the level set of the momentum map has conic singularities is well known in finite dimensions (see, \eg, \parencite[Proposition~8.1.2]{OrtegaRatiu2003}) and was first observed by \textcite[Theorem~5]{ArmsMarsdenEtAl1981}.}.
\end{remark}

\begin{prop}
	\label{prop:symplecticReduction:orbitTypeManifoldSymplectic}
	In the setting of \cref{prop:symplecticReduction:orbitTypeManifold}, assume additionally that \( \LieA{g} \ldot m \) is symplectically closed for every \( m \in M_{(H), \mu} \).
	Then, there exists a symplectic form \( \check{\omega}_{(H), \mu} \) on \( \check{M}_{(H), \mu} \) uniquely determined by
	\begin{equation}
		\label{eq:symplecticReduction:reducedSymplecticForm}
		\pi_{(H), \mu}^* \, \check{\omega}^{}_{(H), \mu} = \iota_{(H), \mu}^* \omega,
	\end{equation}
	where \( \iota_{(H), \mu}: M_{(H), \mu} \to M \) is the natural injection.
\end{prop}
\begin{proof}
	To prove that~\eqref{eq:symplecticReduction:reducedSymplecticForm} uniquely defines a \( 2 \)-form \( \check{\omega}_{(H), \mu} \) on \( \check{M}_{(H), \mu} \) it suffices to show that \( \iota_{(H), \mu}^* \omega \) is \( G_\mu \)-invariant and horizontal with respect to the smooth submersion \( \pi_{(H), \mu} \).
	Invariance under \( G_\mu \) is clear, because \( \iota_{(H), \mu} \) is \( G_\mu \)-equivariant and \( \omega \) is \( G \)-invariant.
	Furthermore, for every \( \xi \in \LieA{g}_\mu \), \( m \in M_{(H), \mu} \) and \( v \in \TBundle_m M_{(H), \mu} \), we find
	\begin{equation}
		\bigl(\iota_{(H), \mu}^* \omega\bigr)_m (\xi \ldot m, v)
			= \omega_m (\xi \ldot m, v)
			= - \kappa (\difLog_m J(v), \xi)
			= 0,
	\end{equation}
	because \( J \) is constant on \( M_{(H), \mu} \).
	In summary, \( \iota_{(H), \mu}^* \omega \) is \( G_\mu \)-invariant and horizontal, and thus descends to a smooth \( 2 \)-form \( \check{\omega}_{(H), \mu} \) on \( \check{M}_{(H), \mu} \) which, by definition, satisfies~\eqref{eq:symplecticReduction:reducedSymplecticForm}.
	Moreover, the identity~\eqref{eq:symplecticReduction:reducedSymplecticForm} shows that \( \check{\omega}_{(H), \mu} \) is closed.
	It remains to prove that \( \check{\omega}_{(H), \mu} \) is non-degenerate.
	For this purpose, recall from \cref{prop:symplecticReduction:orbitTypeManifold} that the projection \( \tangent_m \pi_{(H), \mu} \) restricts to an isomorphism of \( \TBundle_{\equivClass{m}} \check{M}_{(H), \mu} \) with \( E_{G_m} \), where \( E = \TBundle_m S \intersect \ker \difLog_m J \) and \( S \) is the \( G_\mu \)-slice in MGS normal form at \( m \).
	Equation~\eqref{eq:symplecticReduction:reducedSymplecticForm} shows that, under this isomorphism, \( (\check{\omega}_{(H), \mu})_{\equivClass{m}} \) coincides with the restriction of \( \omega_m \) to \( E_{G_m} \).
	By \cref{prop:bifurcationLemma:kernelMomentumMapIsSymplectic}, \( E \) is a symplectic subspace of \( (\TBundle_m M, \omega) \).
	Since \( G_m \) is compact, \cref{prop:sympleticFunctionalAnalysis:invariantSubspaceSymplectic} implies that \( E_{G_m} \) is symplectic as well.
	Thus, we conclude that \( \check{\omega}_{(H), \mu} \) is non-degenerate.
\end{proof}

At this point, we know that \( \check{M}_\mu \) decomposes into orbit type manifolds, every one of which carries a symplectic structure.
We will now see that the pieces fit together in a particularly nice way.
The reader might want to recall from \cref{sec:slices} the notion of a stratified space.
\begin{prop}
	\label{prop:symplecticReduction:orbitTypeDecompositionStratification}
	Let \( G \) be a Lie group with dual pair \( \kappa(G^*_\kappa, G) \), and let \( (M, \omega, G) \) be a symplectic \( G \)-manifold with equivariant momentum map \( J: M \to G^*_\kappa \).
	Let \( \mu \in G^*_\kappa \), and assume that \( G_\mu \) is a Lie subgroup of \( G \) acting properly on \( M \).
	Assume that \( J \) can be brought into a strong MGS normal form \( (H, X, \LieA{g}, \hat{J}, J_\singularPart; \bar{\omega}) \) at every point \( m \in M_\mu \) such that the intersection
	\begin{equation}
		U_{(K)} \intersect \ker{} \intersect J_\singularPart^{-1}(0)
	\end{equation}
	is non-empty for every orbit type \( (K) \) of the \( G_\mu \)-action on \( M_\mu \) satisfying \( (K) \leq (G_m) \).
	Then, the decomposition of \( M_\mu \) and \( \check{M}_\mu \) into orbit type subsets \( M_{(H), \mu} \) and \( \check{M}_{(H), \mu} \), respectively, is a stratification.
\end{prop}
A strong MGS normal form \( (H, X, \LieA{g}, \hat{J}, J_\singularPart; \bar{\omega}) \) satisfying the assumption of \cref{prop:symplecticReduction:orbitTypeDecompositionStratification} is said to \emphDef{have the approximation property}.
\begin{proof}
	Let \( m \in M_\mu \) and let \( (K) \) be an orbit type of the \( G_\mu \)-action on \( M_\mu \) with \( (K) \leq (H) \), where \( H = G_m \).
	By assumption, there exists a strong MGS normal form \( (H, X, \LieA{g}, \hat{J}, J_\singularPart; \bar{\omega}) \) at \( m \) such that \( Y \equiv U_{(K)} \intersect \ker{} \intersect J_\singularPart^{-1}(0) \) is non-empty.
	A point \( x \in U \intersect \ker \) lies in \( Y \) if and only if \( (H_x) = (K) \) and \( J_\singularPart(x) = 0 \).
	Since the \( H \)-action on \( X \) is linear, we have \( H_x = H_{\alpha x} \) for all \( \alpha \in \R_{>0} \).
	Moreover, by~\eqref{eq:momentumMap:MGSnormalForm:quadraticMomentumMap}, we obtain
	\begin{equation}
	 	\kappa\bigl(J_\singularPart(\alpha x), \xi\bigr) 
	 		= \frac{1}{2} \bar{\omega}_0 \bigl(\xi \ldot (\alpha x), \alpha x\bigr)
	 		= \alpha^2 \, \kappa\bigl(J_\singularPart(x), \xi\bigr).
	\end{equation}
	Thus, for every \( x \in Y \) and \( \alpha \in \R_{>0} \), the point \( \alpha x \) lies in \( Y \) as well.
	By letting \( \alpha \to 0 \), we conclude that the point \( m \) lies in the closure of \( Y \) in \( X \).
	Thus, the claim follows from \parencite[Proposition~5.7]{DiezRudolphKuranishi}.
\end{proof}

In summary, we obtain the following result concerning the structure of the symplectic quotient \( \check{M}_\mu \).
\begin{thm}[Singular Reduction Theorem]
	\label{prop:symplecticReduction:mainTheorem}
	Let \( (M, \omega, G) \) be a symplectic \( G \)-manifold with equivariant momentum map \( J: M \to G^*_\kappa \).
	Let \( \mu \in G^*_\kappa \), and assume that \( G_\mu \) is a Lie subgroup of \( G \) acting properly on \( M \).
	Moreover, assume that \( \LieA{g} \ldot m \) is symplectically closed for every \( m \in M_\mu = J^{-1}(\mu) \).
	If  \( J \) can be brought into an MGS normal form at every point of \( M_\mu \), then the following holds:
	\begin{thmenumerate}
		\item
			For every orbit type \( (H) \) of the \( G_\mu \)-action on \( M_\mu \), the orbit type subset \( M_{(H), \mu} \) is a smooth submanifold of \( M \).
			Moreover, there exists a unique smooth manifold structure on the quotient
			\begin{equation}
				\check{M}_{(H), \mu} \defeq M_{(H), \mu} \slash G_\mu
			\end{equation}
			such that the natural projection \( \pi_{(H), \mu}: M_{(H), \mu} \to \check{M}_{(H), \mu} \) is a smooth submersion.
		\item
			If, additionally, the MGS normal forms can be chosen to be strong and to have the approximation property, then the decompositions of \( M_\mu \) and \( \check{M}_\mu \) into orbit type subsets \( M_{(H), \mu} \) and \( \check{M}_{(H), \mu} \), respectively, are stratifications.
		\item
			For every orbit type \( (H) \) of the \( G_\mu \)-action on \( M_\mu \), there exists a symplectic form \( \check{\omega}_{(H), \mu} \) on \( \check{M}_{(H), \mu} \) uniquely determined by
			\begin{equation}
				\pi_{(H), \mu}^* \, \check{\omega}^{}_{(H), \mu} = \iota_{(H), \mu}^* \omega,
			\end{equation}
			where \( \iota_{(H), \mu}: M_{(H), \mu} \to M \) is the natural injection.
			\qedhere
	\end{thmenumerate}
\end{thm}

\begin{example}
	\label{ex::diffOneFormsOnSurface:symplecticReduction}
	Continuing in the setting of \cref{ex::diffOneFormsOnSurface:gaugeAction}, let \( P \to M \) be a principal \( \UGroup(1) \)-bundle on the closed surface \( M \).
	As we have seen, the momentum map \( \SectionMapAbb{J}: \ConnSpace(P) \to \DiffFormSpace^2(M) \) for the \( \GauGroup(P) \)-action on \( \ConnSpace(P) \) is given by (minus) the curvature.
	Thus, the symplectic quotient at \( 0 \),
	\begin{equation}
		\check{\ConnSpace}_0(P) \equiv \SectionMapAbb{J}^{-1}(0) \slash \GauGroup(P),
	\end{equation}
	coincides with the moduli space of flat connections.
	Moreover, by the discussion in \cref{ex::diffOneFormsOnSurface:momentumMapNormalForm}, \( \SectionMapAbb{J} \) can be brought into a strong MGS normal form at every \( A \in \ConnSpace(P) \).
	Since the \( \GauGroup(P) \)-action is has only the single orbit type \( \UGroup(1) \), \cref{prop:symplecticReduction:mainTheorem} implies that \( \check{\ConnSpace}_0(P) \) is a symplectic manifold.

	In the present setting, we can use the method of invariants to get a direct description of \( \check{\ConnSpace}_0(P) \).
	Let \( (\gamma_i) \) be a family of closed piecewise smooth paths in \( M \) generating \( \pi_1(M) \). 
	The holonomy \( \Hol_A(\gamma_i) \) of \( \gamma_i \) with respect to a connection \( A \) furnishes a map
	\begin{equation}
		\SectionMapAbb{K}: \ConnSpace(P) \to \sCohomology^1(M, \UGroup(1)), \quad A \mapsto \bigl(\gamma_i \mapsto \Hol_A(\gamma_i)\bigr),
	\end{equation}
	where we used the Universal Coefficient Theorem and the Hurewicz Theorem to identify \( \sCohomology^1(M, \UGroup(1)) \) with \( \Hom(\pi_1(M), \UGroup(1)) \).
	Since \( \SectionMapAbb{K} \) is \( \GauGroup(P) \)-invariant, it descends to a map \( \check{\SectionMapAbb{K}} \) from \( \check{\ConnSpace}_0(P) \) to \( \sCohomology^1(M, \UGroup(1)) \).
	Moreover, by standard arguments, \( \check{\SectionMapAbb{K}} \) is a diffeomorphism.
	Under the identification given by \( \check{\SectionMapAbb{K}} \), the reduced symplectic form on \( \check{\ConnSpace}_0(P) \) coincides with the intersection form on \( \deRCohomology^1(M, \R) \) by~\eqref{eq:diffOneFormsOnSurface:gaugeActionMomentumMap:symplecticStructureOnE}.
	The corresponding story for a principal \( G \)-bundle \( P \) with non-abelian structure group \( G \) will be discussed below in \cref{sec:yangMillsSurface}.
	It turns out that, in this case, the moduli space \( \check{\ConnSpace}_0(P) \) of flat connections has singularities and is a stratified symplectic space, see \cref{prop:yangMillsSurface:moduliSpaceCentralYMSymplecticStrata}.
\end{example}
It is quite remarkable that, in the linear setting, the usual finite-dimensional result about symplectic strata directly generalizes to the infinite-dimensional realm without any further assumptions.
\begin{coro}
	\label{prop:singularReduction:linear}
	Let \( (X, \omega) \) be a symplectic Fréchet space endowed with a continuous linear symplectic action of the compact Lie group \( G \).
	Then, the \( G \)-action has a unique (up to a constant) equivariant momentum map \( J: X \to \LieA{g}^*_\kappa \) given by
	\begin{equation}
		\label{eq:singularReduction:linear:momentumMap}
		\kappa(J(x), \xi) = \frac{1}{2} \omega(x, \xi \ldot x)
	\end{equation}	
	for \( x \in X \) and \( \xi \in \LieA{g} \).
	Moreover, for every orbit type \( (H) \) of the \( G \)-action on \( X \), the subset \( X_{(H), 0} \defeq X_{(H)} \intersect J^{-1}(0) \) is a smooth submanifold of \( X \) and there exists a unique smooth manifold structure on \( \check{X}_{(H), 0} \defeq X_{(H), 0} \slash G \) such that the natural projection \( \pi_{(H)}: X_{(H), 0} \to \check{X}_{(H), 0} \) is a smooth submersion.
	Furthermore, \( \check{X}_{(H), 0} \) carries a symplectic form \( \check{\omega}_{(H)} \) uniquely determined by
	\begin{equation}
		\label{eq:singularReduction:linear:reducedSymplectic}
		\pi_{(H)}^* \check{\omega}_{(H)} = \restr{\omega}{X_{(H), 0}}.
		\qedhere
	\end{equation}
\end{coro}
\begin{proof}
	Since \( G \) is finite-dimensional,~\eqref{eq:singularReduction:linear:momentumMap} defines a smooth map \( J: X \to \LieA{g}^*_\kappa \).
	That \( J \) is a momentum map, indeed, follows from a routine calculation, which we leave to the reader.
	As \( G \) is compact, the \( G \)-action on \( X \) is proper.
	Moreover, for every \( x \), the orbit \( \LieA{g} \ldot x \) is a finite-dimensional subspace of \( X \) and thus it is symplectically closed due to \cref{prop:symplecticFunctionalAnalysis:finiteDimSubspaceSymplectiallyClosed}.
	According to \cref{prop:momentumMap:MGSnormalForm:linear}, \( J \) can be brought into an MGS normal form.
	Hence, the claims follow from the Singular Reduction \cref{prop:symplecticReduction:mainTheorem}.
\end{proof}

Let us now pass from the kinematic picture presented so far to dynamics.
\begin{prop}
	\label{prop:singularReduction:dynamics}
	Under the assumptions of \cref{prop:symplecticReduction:mainTheorem}, let \( h \) be a \( G \)-invariant Hamiltonian on the symplectic \( G \)-manifold \( (M, \omega, G) \) with equivariant momentum map \( J: M \to G^*_\kappa \).
	Assume\footnote{Recall that the Hamiltonian vector field associated with the Hamiltonian \( h \) may not exist in infinite dimensions and that vector fields on Fréchet manifolds need not have flows.
	The latter is more or less equivalent to (in time) local solutions of the corresponding partial differential equation.} that the associated Hamiltonian vector field \( X_h \) exists and that it has a unique flow \( \flow^h_t \).
	Let \( (H) \) be an orbit type and let \( \mu \in G^*_\kappa \).
	Then,
	\begin{thmenumerate}
		\item
			the flow \( \flow^h_t \) is \( G \)-equivariant and leaves \( M_{(H), \mu} \) invariant and, hence, it projects to a flow \( \check{\flow}^h_t \) on \( \check{M}_{(H), \mu} \),
		\item
			the projected flow \( \check{\flow}^h_t \) is Hamiltonian with respect to the smooth function \( \check{h}_{(H)} \) on \( \check{M}_{(H), \mu} \) uniquely defined by
			\begin{equation}
				\pi_{(H), \mu}^* \check{h}^{}_{(H)} = \restr{h}{M_{(H), \mu}} \, .
				\qedhere
			\end{equation}
	\end{thmenumerate}
\end{prop}
\begin{proof}
	Since \( h \) is \( G \)-invariant, the associated Hamiltonian vector field \( X_h \) is invariant, too.
	The calculation
	\begin{equation}
		\difFracAt{}{t}{t} \flow^h_t (g \cdot m) = (X_h)_{g \cdot m} = g \ldot (X_h)_m = \difFracAt{}{t}{t} g \cdot \flow^h_t (m)
	\end{equation}
	shows that the flow \( \flow^h_t \) is \( G \)-equivariant (since, by assumption, it exists and is unique).
	Moreover, \cref{prop:momentumMap:noetherTheorem} implies that the flow \( \flow^h_t \) leaves \( M_{\mu} \) invariant.
	Hence, \( \flow^h_t \) preserves \( M_{(H), \mu} \) and so it projects onto a flow \( \check{\flow}^h_t \) on \( \check{M}_{(H), \mu} \).
	Denote the induced vector field on \( \check{M}_{(H), \mu} \) by \( \check{X}_{(H)} \).
	Since \( h \) is \( G \)-invariant and since \( \pi_{(H)} \) is a surjective submersion, \( h \) descends to a smooth function \( \check{h}_{(H)} \) on \( \check{M}_{(H), \mu} \).
	That \( \check{X}_{(H)} \) is  Hamiltonian with respect to \( \check{h}_{(H)} \),  indeed, is verified by a routine calculation.
\end{proof}

\subsection{Regular symplectic reduction}

A special case is regular symplectic reduction, where the \( G_\mu \)-action is free.
As a consequence of the Singular Reduction \cref{prop:symplecticReduction:mainTheorem} and of \cref{prop:singularReduction:dynamics}, we obtain the following result.

\begin{thm}[Regular Reduction Theorem]
	\label{prop:symplecticReduction:regular}
	Let \( (M, \omega, G) \) be a symplectic \( G \)-manifold with equivariant momentum map \( J: M \to G^*_\kappa \).
	Let \( \mu \in G^*_\kappa \), and assume that \( G_\mu \) is a Lie subgroup of \( G \) acting properly and freely on \( M \).
	Moreover, assume that \( \LieA{g} \ldot m \) is symplectically closed for every \( m \in M_\mu = J^{-1}(\mu) \).
	If  \( J \) can be brought into an MGS normal form at every point of \( M_\mu \), then the following holds:
	\begin{thmenumerate}
		\item
			The level set \( M_\mu \) is a smooth submanifold of \( M \).
			Moreover, there exists a unique smooth manifold structure on the quotient
			\begin{equation}
				\check{M}_\mu \defeq M_\mu \slash G_\mu
			\end{equation}
			such that the natural projection \( \pi_\mu: M_\mu \to \check{M}_\mu \) is a smooth submersion.
		\item
			There exists a symplectic form \( \check{\omega}_\mu \) on \( \check{M}_\mu \) uniquely determined by
			\begin{equation}
				\pi_\mu^* \, \check{\omega}^{}_\mu = \iota_\mu^* \omega,
			\end{equation}
			where \( \iota_\mu: M_\mu \to M \) is the natural injection.
		\item
			Let \( h \) be a \( G \)-invariant Hamiltonian function whose associated Hamiltonian vector field \( X_h \) exists and has a unique flow \( \flow^h_t \).
			Then the flow \( \flow^h_t \) is \( G \)-equivariant and leaves \( M_\mu \) invariant.
			Hence, it projects to a flow \( \check{\flow}^h_t \) on \( \check{M}_\mu \), which is Hamiltonian with respect to the smooth function \( \check{h}_\mu \) on \( \check{M}_\mu \) uniquely defined by
			\begin{equation}
				\pi_\mu^* \check{h}^{}_\mu = \restr{h}{M_\mu} \, .
				\qedhere
			\end{equation}
	\end{thmenumerate}
\end{thm}
We emphasize that the Regular Reduction Theorem is a direct consequence of the Singular Reduction Theorem.
This is in contrast to the finite-dimensional case, where the Regular Reduction Theorem is usually proved first and then used to derive the Singular Reduction Theorem, \cf \cref{sec:symplecticReductionFiniteDimensions}.

\begin{remark}[Orbit reduction]
	It is interesting to ask whether our methods apply to orbit reduction as well (in both the regular and singular setting).
	This question will be addressed in future work.
	In \cref{sec:applications:teichmuellerSpace}, we will deal with a special example of orbit reduction, namely the Teichmüller space.
\end{remark}


\section{Cotangent bundle reduction}
\label{sec:cotangentBundleReduction}


In many applications in physics, the phase space  is the cotangent bundle \( \CotBundle Q \) of the configuration space $Q$ of the 
system under consideration. 
This case has been studied in detail in \parencite{DiezRudolphReduction}. For completeness, we add a short summary of singular symplectic reduction  for that case here.  Moreover, we outline that the Yang--Mills-Higgs gauge field model fits into that reduction scheme. 
For a detailed analysis of this model we refer to \parencite{DiezRudolphReduction}.

If $Q$ is an infinite-dimensional  (non-Banach) Fréchet manifold, then there is no canonical notion of cotangent bundle, 
because the  topological dual bundle \( \TBundle' Q \) is not a smooth fibre bundle.  
As a substitute, we say that a smooth Fréchet bundle \( \CotBundleProj: \CotBundle Q \to Q \) is a cotangent bundle of $Q$ if there exists a fiberwise non-degenerate pairing with the tangent bundle  \( \TBundle Q \), see \parencite[Appendix~A.1]{DiezRudolphReduction} for the 
details. That is, to speak of a cotangent bundle over $Q$ requires the choice of a bundle \( \CotBundle Q \) and of a pairing \( \CotBundle Q \times_Q \TBundle Q \to \R \). Then, as in finite dimensions, the canonical $1$-form $\theta$  on \( \CotBundle Q \) 
is given by 
\begin{equation}
	\theta_p (X) = \dualPair*{p}{\tangent_p \CotBundleProj (X)}_p, \qquad X \in \TBundle_p (\CotBundle Q) 
\end{equation}
and \( \omega = \dif \theta \) is a symplectic form. 

Now, let  $Q$ be endowed with a smooth, proper action of a Fréchet Lie group $G$. Then, by linearization we obtain 
a $G$-action on $\TBundle Q$ and by requiring that the pairing 
\( \CotBundle Q \times_Q \TBundle Q \to \R \)  be left invariant,  the action on \( \TBundle Q \) induces a \( G \)-action on 
\( \CotBundle Q \). It is easy to see that this action preserves both $\theta $ and 
$\omega$. In order to exclude pathological cases, we make the following regularity assumptions:
\begin{enumerate}
\item[(1)] There exists a \( G \)-equivariant diffeomorphism between \( \TBundle Q \) and \( \CotBundle Q \). 
\item[(2)]
 In all constructions which rely on dualizing  the corresponding adjoint maps exist (and are surjective when the original map is injective). 
\end{enumerate}
Our construction of the normal form rests on the existence of a slice for the $G$-action\footnote{We refer to \parencite{DiezSlice} for 
the study of the Slice Theorem in the Fréchet context.} that is compatible with 
the cotangent bundle structure. Thus, we assume: 
\begin{enumerate}
\item[(3)] 
The $G$-action admits a slice $S$ and the associated  tube map \( \chi: G \times_{G_q} S \to Q \) fulfils the following:
the lift \( \cotangent \chi: \CotBundle Q \to \CotBundle \bigl(G \times_{G_q} S\bigr) \) of $\chi$ exists and is a diffeomorphism when restricted to some appropriate neighborhood of \( q \).
\end{enumerate}

Next, in order to define the infinite-dimensional counterpart \( J: \CotBundle Q \to \LieA{g}^* \) of the momentum map, we need to specify a dual space to the Lie algebra \( \LieA{g} \) of \( G \). Thus, in line with the strategy of \cref{sec:momentumMaps}, we have to 
 choose a Fréchet space \( \LieA{g}^* \) and a separately continuous non-degenerate bilinear form \( \kappa: \LieA{g}^* \times \LieA{g} \to \R \).
With respect to these data, the momentum map \( J: \CotBundle Q \to \LieA{g}^* \), if it exists, satisfies
\begin{equation}
	\label{eq:cotangentBundle:momentumMap}
 	\kappa(J(p), \xi) = \dualPair{p}{\xi \ldot q}, \quad p \in \CotBundle_q Q, \xi \in \LieA{g}.
\end{equation}
In general, in infinite dimensions, the momentum map may not exist. 
But Assumption (2) ensures that the functional on \( \LieA{g} \) defined by the right-hand-side of \cref{eq:cotangentBundle:momentumMap} is represented by an element of \( \LieA{g}^* \), and thus the lifted action of $G$ to $\CotBundle Q$ has a momentum map.

In the above described setting, there is a straightforward strategy for the construction of the normal form, which allows for 
omitting the more complicated procedure described in the present paper for the general case. It goes as follows:
\begin{enumerate}
	\item
		By using a slice \( S \) at \( q \in Q \), reduce the problem to \( \CotBundle (G \times_{G_q} S) \).
	\item
		Establish an equivariant diffeomorphism \( \CotBundle (G \times_{G_q} S) \isomorph G \times_{G_q} (\LieA{m}^* \times \CotBundle S) \), where \( \LieA{m} \) is a complement of \( \LieA{g}_q \) in \( \LieA{g} \).
	\item
		Calculate the momentum map under these identifications.
\end{enumerate}
Here,  \( \LieA{g}_q \)  is the  Lie algebra of the stabilizer $G_q$ of the $G$-action  at $q \in Q$. 

To accomplish the first step, let \( p \in \CotBundle Q \) be a point in the fiber over \( q \in Q \), let  \( S \) be a slice of the 
$G$-action at $q$ and let \( \chi: G \times_{G_q} S \to Q \) be the associated tube mapping. The latter is a $G$-equivariant 
diffeomorphism onto an open neighborhood of \( q \) in \( Q \). Then, by Assumption $(3)$ and \parencite[Proposition~2.3]{DiezRudolphReduction}, 
$\cotangent \chi$ is a symplectomorphism and 
\begin{equation}
\label{transp-J-1}
 J \circ  \cotangent \chi^{-1}: \CotBundle (G \times_{G_q} S) \to \LieA{g}^* 
\end{equation}
is a momentum map for the lifted \( G \)-action on \( \CotBundle (G \times_{G_q} S) \).

Next, let us explain how to accomplish the second step. As the $G$-action is proper, the stabilizer $G_q$ is compact and thus, 
by \parencite[Lemma~2.4]{DiezRudolphReduction}, there exists an \( \AdAction_{G_q} \)-invariant complement 
\( \LieA{m} \) of \(  \LieA{g}_q  \) in \( \LieA{g} \), together with a complementary decomposition 
 \( \LieA{g}^* = \LieA{m}^* \oplus   \LieA{g}_q^* \). The choice of \( \LieA{m} \) induces a homogeneous connection in 
$G \to G/G_q$ and the latter yields a $G$-equivariant diffeomorphism 
\begin{equation}
	\label{eq:cotangentBundle:diffeoAssociatedBundle}
	\TBundle (G \times_{G_q} S) \isomorph G \times_{G_q} (\LieA{m} \times \TBundle S).
	\qedhere
\end{equation}
By dualizing this isomorphism we obtain a \( G \)-equivariant diffeomorphism
\begin{equation}
	\phi: G \times_{G_q} (\LieA{m}^* \times \CotBundle S) \to \CotBundle (G \times_{G_q} S) \, ,
\end{equation}
where \( \CotBundle S \) is, by definition, the image of \( \TBundle S \) under the restriction of the  \( G \)-equivariant diffeomorphism 
\( \TBundle Q \to \CotBundle Q \).

Now, the third step is accomplished by pulling back $J$ to  $G \times_{G_q} (\LieA{m}^* \times \CotBundle S) $ via the 
map
\begin{equation}
\label{locIso-Phi}
 \Phi = \cotangent \chi^{-1} \circ \phi  : G \times_{G_q} (\LieA{m}^* \times \CotBundle S) \to \CotBundle Q\, ,
\end{equation}
which is a diffeomorphism onto an open neighborhood of \( p \in \CotBundle_q Q \) in \( \CotBundle Q \).
Under \( \Phi \), the momentum map \( J: \CotBundle Q \to \LieA{g}^* \) takes the following normal form:
	\begin{equation}\label{eq:cotangentBundle:normalFormMomentumMap}
		(J \circ \Phi) \bigl(\equivClass{a, (\nu, \alpha_s)}\bigr) = \CoAdAction_a \bigl(\nu + J_{G_q}(\alpha_s)\bigr),
	\end{equation}
	where \( J_{G_q}: \CotBundle S \to \LieA{g}_q^* \) is the momentum map for the lifted \( G_q \)-action on \( \CotBundle S \), 
see \parencite[Theorem~2.10]{DiezRudolphReduction} for the details. 

We stress that, in sharp contrast to the standard approach in finite dimensions, the semi-global diffeomorphism \( \Phi \) does not yield a \emph{symplectic slice} for the \( G \)-action.
In fact, it does not even yield a slice for the lifted \( G \)-action on \( \CotBundle Q \) because in \( G \times_{G_q} (\LieA{m}^* \times \CotBundle S) \) the quotient is taken by \( G_q \) and not by \( G_p \). Hence, we do not construct here a Marle--Guillemin--Sternberg normal form in the sense of \cref{def:momentumMap:abstractMGSnormalForm}.
But, it turns out that the simple normal form~\eqref{eq:cotangentBundle:normalFormMomentumMap} of the momentum map is sufficient for most questions concerning singular cotangent bundle reduction. We also note that the pull-back of the symplectic form $\omega$ by $\Phi$ may be easily 
calculated, see \parencite[Proposition~2.11]{DiezRudolphReduction}.

The reader may wonder, whether the above normal form can be derived from the general construction provided in the previous sections. 
Unfortunately, for the time being, this is an open question. 

Given the above results, one can prove  the following symplectic reduction theorem, which is similar to the general result \cref{prop:symplecticReduction:mainTheorem} based on the MGS normal form.
\begin{thm}
	\label{prop:cotangentBundle:singularSympRed}
	Let \( Q \) be a Fréchet \( G \)-manifold.
	Assume that the \( G \)-action is proper, that it admits at every point a slice compatible with the cotangent bundle structures and that the decomposition of \( Q \) into orbit types satisfies the frontier condition.
	Moreover, assume that \( \CotBundle Q \) is a Fréchet manifold, which is \( G \)-equivariantly diffeomorphic to \( \TBundle Q \), and that the lifted action on \( \CotBundle Q \), endowed with its canonical symplectic form \( \omega \), has a momentum map \( J \).
	Then, the following hold.
	\begin{thmenumerate}
		\item
			The set of orbit types of \( P \defeq J^{-1}(0) \) with respect to the lifted \( G \)-action coincides with the set of orbit types for the \( G \)-action on \( Q \). 
	\item
			The reduced phase space \( \check{P} \defeq J^{-1}(0) \slash G \) is stratified into orbit type manifolds \( \check{P}_{(K)} \defeq (J^{-1}(0))_{(K)} \slash G \).
	\item 
			Assume, additionally, that every orbit is symplectically closed, that is, the symplectic double orthogonal \( (\LieA{g} \ldot p)^{\omega \omega} \) coincides with \( \LieA{g} \ldot p \) for all \( p \in P \).
			Then, for every orbit type \( (K) \), the manifold \( \check{P}_{(K)} \) carries a symplectic form \( \check{\omega}_{(K)} \) uniquely determined by
			\begin{equation}
				\pi_{(K)}^* \check{\omega}_{(K)} = \restr{\omega}{P_{(K)}},
			\end{equation}
			where \( \pi_{(K)}: P_{(K)} \to \check{P}_{(K)} \) is the natural projection.
			\qedhere
	\end{thmenumerate}
\end{thm}
For the proof we refer to \parencite[Theorem~2.12]{DiezRudolphReduction}.
When the \( G \)-action on \( Q \) has only one orbit type\footnote{This assumption includes, of course, also the case of a free action.}, we obtain the infinite-dimensional counterpart to the well-known cotangent bundle reduction theorem for one orbit type.
\begin{coro}
	\label{prop:cotangentBundle:oneOrbitTypeReducedSpaceCotangentBundle}
	In the setting of \cref{prop:cotangentBundle:singularSympRed}, assume additionally that the \( G \)-action on \( Q \) has only one orbit type. 
	Then, \( \check{P} = J^{-1}(0) \slash G \) is symplectomorphic to \( \CotBundle (Q \slash G) \) with its canonical symplectic structure. 	
\end{coro}  

Note that in the cotangent bundle case we are dealing with two $G$-manifolds, namely, $Q$ endowed with the original $G$-action and 
$\CotBundle Q$ endowed with the lifted $G$-action.  Thus, let us consider the orbit types \( (K) \) and \( (H) \) of \( \CotBundle Q \) and \( Q \), respectively, and let us define  the subset
\begin{equation}
	\seam{(\CotBundle Q)}^{(K)}_{(H)} \defeq \set{p \in \CotBundle_q Q \given q \in Q_{(H)}, p \in (\CotBundle Q)_{(K)}}
\end{equation}
of the orbit type stratum \( (\CotBundle Q)_{(K)} \). 
Since the projection \( \CotBundle Q \to Q \) is \( G \)-equivariant, \( \seam{(\CotBundle Q)}^{(K)}_{(H)} \) is non-empty only if \( (K) \leq (H) \).
Moreover, the union of \( \seam{(\CotBundle Q)}^{(K)}_{(H)} \) over all orbit types \( (H) \) fulfilling this condition yields the orbit type stratum \( (\CotBundle Q)_{(K)} \). It is easy to show that  the sets \( \seam{(\CotBundle Q)}^{(K)}_{(H)} \) and \( \seam{(\CotBundle Q)}^{(K)}_{(H)} \slash G \) are submanifolds of \( (\CotBundle Q)_{(K)} \) and 
\( (\CotBundle Q)_{(K)} \slash G \), respectively, see \parencite[Lemma~2.22]{DiezRudolphReduction}.
We call the sets \( \seam{(\CotBundle Q)}^{(K)}_{(H)} \) the \emphDef{secondary strata} and the decomposition of \( \CotBundle Q \) into these secondary strata is referred to as the \emphDef{secondary orbit type stratification}.

Next, consider the sets 
\begin{equation}
	\seam{P}^{(K)}_{(H)} \defeq \seam{(\CotBundle Q)}^{(K)}_{(H)} \intersect J^{-1}(0) \, ,
\end{equation}
where \( (K) \leq (H) \). We call  $\seam{P}^{(K)}_{(H)}$ a \emphDef{preseam} and the quotient 
\( \seam{\check{P}}^{(K)}_{(H)} \defeq \seam{P}^{(K)}_{(H)} \slash G \) a \emphDef{seam}. One can prove that 
 the preseam \( \seam{P}^{(K)}_{(H)} \) is a smooth submanifold of \( \CotBundle Q \) and the seam \( \seam{\check{P}}^{(K)}_{(H)} \) is a smooth submanifold of \( \check{P}_{(K)} \) and of \( (\CotBundle Q)_{(K)} \slash G \).
Moreover, \( \seam{\check{P}}^{(K)}_{(H)} \) is a smooth fiber bundle over \( \check{Q}_{(H)} \), see 
\parencite[Lemma~2.24]{DiezRudolphReduction}. Here, \( \check{Q} = Q/G \).

To summarize, one can prove the following supplement to \cref{prop:cotangentBundle:singularSympRed}, see 
\parencite[Theorem~2.29]{DiezRudolphReduction}.
\begin{prop}[Secondary stratification]
	\label{prop:cotangentBundle:singularCotangentBundleRed}
Under the assumptions of \cref{prop:cotangentBundle:singularSympRed}, the following hold. 
	\begin{thmenumerate}
	     \item
			Every symplectic stratum \( \check{P}_{(K)} \) is further stratified as
			\begin{equation}
			 	\check{P}_{(K)} = \bigDisjUnion_{(H) \geq (K)} \seam{\check{P}}^{(K)}_{(H)} \,,
			\end{equation}
			where each seam \( \seam{\check{P}}^{(K)}_{(H)} \) is a smooth fiber bundle over \( \check{Q}_{(H)} \).
		\item
			For every orbit type \( (H) \), the principal seam \( \seam{\check{P}}^{(H)}_{(H)} \) endowed with the restriction of the symplectic form \( \check{\omega}_{(H)} \) is symplectomorphic to \( \CotBundle (\check{Q}_{(H)}) \) endowed with its canonical symplectic structure.
		\item
			The decomposition
			\begin{equation}
				\check{P} = \bigDisjUnion_{(H) \geq (K)} \seam{\check{P}}^{(K)}_{(H)}
			\end{equation}
			is a stratification of \( \check{P} \) called the \emph{secondary stratification}.
			Moreover, the projection \( \CotBundle Q \to Q \) induces a stratified surjective submersion \( \check{P} \to \check{Q} \) with respect to the secondary stratification of \( \check{P} \) and the orbit type stratification of \( \check{Q} \).
			\qedhere
	\end{thmenumerate}
\end{prop}

Finally, let us discuss dynamics. 
\begin{prop}
	Let \( h \) be a \( G \)-invariant Hamiltonian on \( \CotBundle Q \).
	Assume that the associated Hamiltonian vector field \( X_h \) exists and that it has a unique flow \( \flow^h_t \).
	Let \( (K) \) be an orbit type.
	Then,
	\begin{thmenumerate}
		\item
			the flow \( \flow^h_t \) is \( G \)-equivariant and leaves \( P_{(K)} \) invariant and, hence, it projects to a flow \( \check{\flow}^h_t \) on \( \check{P}_{(K)} \),
		\item
			the projected flow \( \check{\flow}^h_t \) is Hamiltonian with respect to the smooth function \( \check{h}_{(K)} \) on \( \check{P}_{(K)} \) defined by
			\begin{equation}
				\pi_{(K)}^* \check{h}_{(K)} = \restr{h}{P_{(K)}}.
				\qedhere
			\end{equation}
	\end{thmenumerate}
\end{prop}
For the proof see \parencite[Proposition~2.30]{DiezRudolphReduction}.

The interaction of dynamics with the secondary stratification is more complicated, cf. \parencite[Example~2.31]{DiezRudolphReduction}. In particular, the seams are in general not preserved by the Hamiltonian flow. 


In the remainder of this section we will show that the Yang--Mills--Higgs gauge field model fits into the above setting. 

Thus, let \( (M, g) \) be a \( 3 \)-dimensional oriented, compact manifold with time-dependent Riemannian metric, which plays the role of a Cauchy surface in the \(( 3+1 )\)-splitting. 
The geometry underlying our model is that of a principal \( G \)-bundle \( P \to M \), where \( G \) is a connected compact Lie group. In this picture, a gauge potential is described by a connection form $A$ on $P$.
The covariant derivative with respect to \( A \) is denoted by \( \dif_A \) and the curvature of \( A \) is written as \( F_A \).
A bosonic matter field is a section \( \varphi \) of the associated vector bundle \( F = P \times_G \FibreBundleModel{F} \), where the typical fiber \( \FibreBundleModel{F} \) carries a unitary \( G \)-representation.
Thus, the configuration  space  \( \SectionSpaceAbb{Q} \) of the system consists of pairs \( (A, \varphi) \).
It obviously is the product of the infinite-dimensional affine Fréchet space \( \ConnSpace \) of connections and the Fréchet space \( \SectionSpaceAbb{F} \) of sections of \( F \).

Since \( \SectionSpaceAbb{Q} \) is an affine space, its tangent bundle \( \TBundle \SectionSpaceAbb{Q} \) is trivial with fiber 
given by 
\( \DiffFormSpace^1(M, \AdBundle P) \times \sSectionSpace(F) \).
We denote points in \( \TBundle \SectionSpaceAbb{Q} \) by tuples \( (A, \alpha, \varphi, \zeta) \) with \( \alpha \in \DiffFormSpace^1(M, \AdBundle P) \) and \( \zeta \in \sSectionSpace(F) \).
A natural choice for the cotangent bundle $\CotBundle \SectionSpaceAbb{Q}$ is the trivial bundle over $\SectionSpaceAbb{Q}$ with fiber \( 	\DiffFormSpace^2(M, \CoAdBundle P) \times \DiffFormSpace^3(M, F^*)  \).
We denote elements of this fiber by pairs \( (D, \Pi) \).
Then, the natural pairing with \( \TBundle \SectionSpaceAbb{Q} \) is given by integration over \( M \),\footnote{Here, 
the Hodge dual of a vector-valued differential form \( \alpha \in \DiffFormSpace^k(M, E) \) is the 
\emph{dual-valued} differential form \( \hodgeStar \alpha \in \DiffFormSpace^{3-k}(M, E^*) \). Moreover, for \( \alpha \in \DiffFormSpace^k(M, E) \) and \( \beta \in \DiffFormSpace^{3-k}(M, E^*) \), we denote by 
\(\alpha \dot{\wedge} \beta \) the real-valued top-form which arises from combining the wedge 
product with the natural pairing \( \dualPairDot: E \times E^* \to \R \).}  
\begin{equation}
	\dualPair*{(D, \Pi)}{(\alpha, \zeta)} = \int_M D \dot{\wedge} \alpha + \int_M \Pi \dot{\wedge} \zeta .
\end{equation}
On \( \SectionSpaceAbb{Q} = \ConnSpace \times \SectionSpaceAbb{F} \) we have a left action of the group \( \GauGroup = \sSectionSpace(P \times_G G) \) of local gauge transformations, 
\begin{equation}
\label{LocGTr}
	A \mapsto \AdAction_{\lambda} A + \lambda \dif \lambda^{-1},
	\quad
	\varphi \mapsto \lambda \cdot \varphi,
\end{equation}
for \( \lambda \in \GauGroup \).
A straightforward calculation shows that the momentum map for the lifted action  with respect to the natural choice
 \( \GauAlgebra^* = \DiffFormSpace^3(M, \CoAdBundle P) \) is given by
\begin{equation}\label{eq:yangMillsHiggs:momentumMap}
	\SectionMapAbb{J}(A, D, \varphi, \Pi) = \dif_A D + \varphi \diamond \Pi \, ,
\end{equation}
where $\diamond $ is the diamond product, see \parencite[Section~III]{DiezRudolphReduction}. 
Thus,  the momentum map constraint \( \SectionMapAbb{J} = 0 \) coincides with the Gauß constraint. Clearly, this interpretation is well known, cf. \parencite{DiezRudolphClebsch,Sniatycki1999,ArmsMarsdenEtAl1981}.

We can show now that all assumptions made in the general discussion above are met for the Yang--Mills--Higgs system:
\begin{enumerate}
	\item 
		\( \SectionSpaceAbb{Q} \) is a Fréchet manifold, because it is an affine space modelled on the Fréchet vector space \( \DiffFormSpace^1(M, \AdBundle P) \times \sSectionSpace(F) \).
	\item
		\( \GauGroup \) is a Fréchet Lie group, because it is realized as the space of sections of the group bundle \( P \times_G G \), see \parencite{CirelliMania1985} for details.
	\item
		The cotangent bundle \( \CotBundle \SectionSpaceAbb{Q} = \SectionSpaceAbb{Q} \times \DiffFormSpace^2(M, \CoAdBundle P) \times \DiffFormSpace^3(M, F^*) \) is clearly a Fréchet manifold.
		The Hodge operator yields a fiber-preserving \( \GauGroup \)-equivariant diffeomorphism between \( \TBundle \SectionSpaceAbb{Q} \) and \( \CotBundle \SectionSpaceAbb{Q} \).
	\item 
		The \( \GauGroup \)-action on \( \SectionSpaceAbb{Q} \) is affine and thus smooth.
		Moreover, it is proper, see \parencite{DiezSlice,RudolphSchmidtEtAl2002}.
	\item 
		The \( \GauGroup \)-action on \( \SectionSpaceAbb{Q} \) admits a slice at every point.
		First, the slice \( \SectionSpaceAbb{S}_{A_0} \) at \( A_0 \in \ConnSpace \) is given by the Coulomb gauge condition.
		That is,\footnote{Here, as usual, \( \dif^*_{A} \alpha \defeq (-1)^k \hodgeStar \dif_{A} \hodgeStar \alpha \) for a \( k \)-form \( \alpha \).}
		\begin{equation}
			\SectionSpaceAbb{S}_{A_0} \defeq \set{A \in \SectionSpaceAbb{U} \given \dif_{A_0}^* (A - A_0) = 0},
		\end{equation}
		where \( \SectionSpaceAbb{U} \) is an open neighborhood of \( A_0 \) in \( \ConnSpace \).
		In order to see that \( \SectionSpaceAbb{S} \) is a slice, indeed, one uses the Nash--Moser inverse function theorem \parencite{Hamilton1982}, which amongst other things relies on the fact that \( \SectionSpaceAbb{Q} \) and \( \GauGroup \) are in fact tame Fréchet.
		The details can be found in \parencite{DiezSlice,AbbatiCirelliEtAl1989}.
		
		Note that this slice for the \( \GauGroup \)-action on \( \ConnSpace \) fixes the gauge transformations up to elements of the stabilizer \( \GauGroup_{A_0} \) of \( A_0 \).
		Thus, we are left with the \( \GauGroup_{A_0} \)-action on \( \SectionSpaceAbb{F} \).
		This is a linear action of a finite-dimensional compact group on a Fréchet space and hence has a slice \( \SectionSpaceAbb{S}_{\varphi_0} \) at every point \( \varphi_0 \in \SectionSpaceAbb{F} \), see \parencite[Theorem~A.7]{DiezRudolphReduction}.
		By \parencite[Proposition~A.8]{DiezRudolphReduction}, the product \( \SectionSpaceAbb{S}_{A_0} \times \SectionSpaceAbb{S}_{\varphi_0} \) is a slice at \( (A_0, \varphi_0) \) for the \( \GauGroup \)-action on \( \SectionSpaceAbb{Q} \).
		Moreover, one can show that this slice is compatible with the cotangent bundle structures, see \parencite[Section~III.A]{DiezRudolphReduction}.
	\item
		That every orbit of \( \GauGroup \) is symplectically closed has been shown in \parencite[Lemma~3.13]{DiezRudolphReduction}. 
		Thus, in particular, the strong version of the Bifurcation Lemma holds.
\end{enumerate}
As a consequence, the momentum map $J$ can be brought to the normal form as given by~\eqref{eq:cotangentBundle:normalFormMomentumMap}.
For \cref{prop:cotangentBundle:singularSympRed,prop:cotangentBundle:singularCotangentBundleRed} to hold we assume, additionally, that the frontier condition for the decomposition of \( \SectionSpaceAbb{Q} \) into gauge orbit types is satisfied.
The orbit type decomposition of \( \ConnSpace \) satisfies the frontier condition \parencite[Theorem~4.3.5]{KondrackiRogulski1986}.
However, including matter fields is a rather delicate issue, see \parencite{DiezRudolphReduction} for further details. 

In \parencite{DiezRudolphReduction}, the reader may find a detailed study of symplectic reduction applied to the Yang--Mills--Higgs model, including  
a discussion of the bosonic sector of the Glashow-Weinberg-Salam model.


\section{Applications}
\subsection{Yang--Mills equation over a Riemannian surface}
\label{sec:yangMillsSurface}

\NewDocumentCommand { \pnt } { }{
	m_0
}

In this section, we are concerned with the moduli space of Yang--Mills connections on a principal bundle over a closed Riemannian surface.
This moduli space was extensively studied both from the geometric and algebraic point of view by \textcite{AtiyahBott1983}.
They described the structure of this moduli space using infinite-dimensional techniques inspired by ideas from symplectic reduction.
Here, we rework and extend the approach of \citeauthor{AtiyahBott1983} in our rigorous framework of symplectic reduction of Fréchet manifolds.
As an application of the Reduction \cref{prop:symplecticReduction:mainTheorem}, we show that (a variant of) the Yang--Mills moduli space is a symplectic stratified space.

Let \( G \) be a compact connected Lie group and let \( P \to M \) be a principal \( G \)-bundle over a closed Riemannian surface \( (M, g) \).
Fix an \( \AdAction_G \)-invariant pairing on the Lie algebra \( \LieA{g} \) of \( G \).
We are interested in connections \( A \in \ConnSpace(P) \) satisfying the Yang--Mills equation
\begin{equation}
	\dif_A \hodgeStar F_A = 0,
\end{equation}
where \( \hodgeStar \) refers to the Hodge star operator associated with the Riemannian metric \( g \) on \( M \).
A special class of Yang--Mills connections is provided by connections \( A \) whose curvature is of the form
\begin{equation}
	\label{eq:yangMillsSurface:centralYM}
	F_A = \xi \cdot \vol_g,
\end{equation}
where \( \xi \) is an element of the Lie algebra \( \aCenter \) of the center of \( G \).
We call such a connection a \emphDef{central Yang--Mills connection} and refer to \( \xi \) as the charge of \( A \).
The importance of central Yang--Mills connections for the study of the solution space of the Yang--Mills equation on a Riemannian surface comes from the following observation.
\begin{prop}[\textnormal{\parencite[p.~560]{AtiyahBott1983}}]
	Every Yang--Mills connection \( A \) on \( P \) is reducible to a central Yang--Mills connection \( A_\xi \) on a principal subbundle \( P_\xi \subseteq P \). 
\end{prop}
Thus, the analysis of the moduli space of Yang--Mills connections is divided into two steps:
first, for all \( \xi \in \LieA{g} \), determine the possible reductions \( P_\xi \) of \( P \) to the stabilizer subgroup \( G_\xi \); second, investigate the moduli space of central Yang--Mills connections on \( P_\xi \).
As \( \xi \) is determined by the topological type of \( P_\xi \) (according to Chern--Weil theory), the first step has a topological flavor and has been extensively discussed in \parencite[Section~6]{AtiyahBott1983}.
In the following, we focus on the second step from the point of view of symplectic reduction.

Recall from \cref{ex::diffOneFormsOnSurface:symplecticReduction}, that in the case \( G = \UGroup(1) \) the moduli space of flat connections was realized as a symplectic quotient.
In a similar vein, we now describe the moduli space of central Yang--Mills connections on \( P \) as a symplectically reduced space.
The space \( \ConnSpace(P) \) of connections on \( P \) is an affine space modeled on the tame Fréchet space \( \DiffFormSpace^1(M, \AdBundle P ) \) of \( 1 \)-forms on \( M \) with values in the adjoint bundle \( \AdBundle P \).
As in \cref{ex::diffOneFormsOnSurface:generalizationToConnections}, the \( 2 \)-form \( \omega \) on \( \ConnSpace(P) \) defined by the integration pairing
\begin{equation}
	\label{eq:yangMillsSurface:symplecticForm}
	\omega_A (\alpha, \beta) = \int_M \wedgeDual{\alpha}{\beta}
\end{equation}
for \( \alpha, \beta \in \DiffFormSpace^1(M, \AdBundle P) \) is a symplectic form, where \( \wedgeDualDot \) denotes the wedge product relative to the \( \AdAction_G \)-invariant pairing on \( \LieA{g} \).
The natural action on \( \ConnSpace(P) \) of the group \( \GauGroup(P) \) of gauge transformations of \( P \) is smooth and preserves the symplectic structure \( \omega \).
A straightforward calculation similar to the one in \cref{ex::diffOneFormsOnSurface:gaugeAction} verifies that the map
\begin{equation}
	\SectionMapAbb{J}: \ConnSpace(P) \to \sSectionSpace(\AdBundle P), \qquad A \mapsto \hodgeStar F_A
\end{equation}
is a momentum map for the \( \GauGroup(P) \)-action on \( \ConnSpace(P) \), relative to the natural pairing
\begin{equation}
	\kappa: \sSectionSpace(\AdBundle P) \times 
	\sSectionSpace(\AdBundle P) \to \R, \qquad (\phi, \varrho) \mapsto 
	- \int_M \dualPair{\phi}{\varrho} \, \vol_g.
\end{equation}
Thus, the symplectic quotient
\begin{equation}
	\SectionMapAbb{J}^{-1}(\xi) \slash \GauGroup(P)
\end{equation}
at \( \xi \in \aCenter \) (viewed as a constant section of \( \AdBundle P \)) coincides with the moduli space \( \check{\ConnSpace}_\xi (P) \) of central Yang--Mills connections with charge \( \xi \).
Recall that the \( \GauGroup(P) \)-action on \( \ConnSpace(P) \) is in general not free and thus the symplectic quotient is not a smooth symplectic manifold.
The following is the next best thing one could hope for. 
\begin{thm}
	\label{prop:yangMillsSurface:moduliSpaceCentralYMSymplecticStrata}
	For every \( \xi \in \aCenter \), the orbit type subsets of the moduli space \( \check{\ConnSpace}_\xi (P) \) are finite-dimensional symplectic manifolds.
\end{thm}
\begin{proof}
	For \( A \in \SectionMapAbb{J}^{-1}(\xi) \), let us verify the assumptions of \cref{prop:momentumMap:MGSnormalForm:elliptic} for the momentum map \( \SectionMapAbb{J} \):
	\begin{enumerate}
		\item
			Since \( \xi \) is a central element, its stabilizer coincides with the whole group \( \GauGroup(P) \), which is a geometric tame Fréchet Lie group.
		\item 
			The \( \GauGroup(P) \)-action on \( \ConnSpace(P) \) is proper and admits a slice \( \SectionSpaceAbb{S} \) at \( A \) as discussed in \parencite{DiezSlice,AbbatiCirelliEtAl1989,RudolphSchmidtEtAl2002}.
		\item
			Since the pairing \( \kappa \) is \( \AdAction_{\GauGroup(P)} \)-invariant, the coadjoint action coincides with the adjoint action and thus is clearly linear.
		\item
			The chain~\eqref{eq:momentumMap:MGSnormalForm:elliptic:chain} takes the following form here:
			\begin{equationcd}[label=eq:yangMillsSurface:moduliSpaceCentralYMSymplecticStrata:chain]
				0 \to[r] 
					& \DiffFormSpace^0(M, \AdBundle P)
						\to[r, "- \dif_B"]
					& \DiffFormSpace^1(M, \AdBundle P)
						\to[r, "\hodgeStar \dif_B"]
					& \DiffFormSpace^0(M, \AdBundle P)
						\to[r]
					& 0,
			\end{equationcd}
			where the connection \( B \) on \( P \) is an element of the slice \( \SectionSpaceAbb{S} \).
			This is clearly a chain of differential operators tamely parametrized by \( B \in \SectionSpaceAbb{S} \).
			Moreover, for \( B = A \), this chain is an elliptic complex, because we have
			\begin{equation}
				\dif_A \dif_A \eta 
					= \LieBracket{F_A}{\eta}
					= \LieBracket{\xi}{\eta} \vol_g
					= 0
			\end{equation}
			for every \( \eta \in \sSectionSpace(\AdBundle P) \).
		\item
			Arguments similar to those used in \cref{ex::diffOneFormsOnSurface:gaugeActionMomentumMap} show that \( \GauAlgebra(P) \ldot A \) is symplectically closed in \( \TBundle_A \ConnSpace(P) \) and that the image of \( \difLog_A \SectionMapAbb{J} \) is \( \LTwoFunctionSpace \)-closed in \( \GauAlgebra(P) \).
	\end{enumerate}
	Thus, by \cref{prop:momentumMap:MGSnormalForm:elliptic}, the momentum map \( \SectionMapAbb{J} \) can be brought into a strong MGS normal form at \( A \).
	Now the claim follows from the Singular Reduction \cref{prop:symplecticReduction:mainTheorem}.
\end{proof}
The statement that the top stratum of \( \check{\ConnSpace}_\xi (P) \) is endowed with a natural symplectic reduction was already obtained by \textcite[p.~587]{AtiyahBott1983}.
The symplectic nature of the singular strata has been established in \parencite[Theorem~1.2]{Huebschmann1996} in the Sobolev framework by reducing the problem to a \emph{finite-dimensional} symplectic quotient.
\begin{remark}
	Let us spell out the local structure of \( \check{\ConnSpace}_\xi (P) \) at a point \( \equivClass{A} \) as given by \cref{prop:momentumMap:MGSnormalForm:elliptic,prop:symplecticReduction:mainTheorem}.
	According to \cref{prop:bifurcationLemma:kernelMomentumMapIsSymplectic}, the space \( \ker \) in the strong MGS normal form is identified with the middle homology of the chain~\eqref{eq:yangMillsSurface:moduliSpaceCentralYMSymplecticStrata:chain} at \( B = A \), that is,
	\begin{equation}
		\ker = \ker\bigl(\hodgeStar \dif_A: \DiffFormSpace^1(M, \AdBundle P) \to \DiffFormSpace^0(M, \AdBundle P)\bigr) \slash \img \dif_A \equiv \deRCohomology^1_A(M, \AdBundle P).
	\end{equation}
	Note that this (co)homology group is only well-defined because \( A \) is a central Yang--Mills connection.
	Similarly, the Lie algebra of the stabilizer subgroup \( \GauGroup_A(P) \) of \( A \) is given by
	\begin{equation}
		\GauAlgebra_A(P) = \ker\bigl(\dif_A: \DiffFormSpace^0(M, \AdBundle P) \to \DiffFormSpace^1(M, \AdBundle P)\bigr) \equiv \deRCohomology^0_A(M, \AdBundle P).
	\end{equation}
	The symplectic structure \( \omega_A \) on \( \TBundle_A \ConnSpace(P) \) restricts to a symplectic structure \( \bar{\omega}_A \) on \( \deRCohomology^1_A(M, \AdBundle P) \).
	By~\eqref{eq:yangMillsSurface:symplecticForm}, we have
	\begin{equation}
		\bar{\omega}_A\bigl(\equivClass{\alpha}, \equivClass{\beta}\bigr) = \int_M \wedgeDual{\alpha}{\beta}.
	\end{equation}
	In other words, \( \bar{\omega}_A \) is the non-abelian generalization of the intersection form.
	It is clear that the action of \( \GauGroup(P) \) on \( \ConnSpace(P) \) induces a symplectic action of \( \GauGroup_A(P) \) on \( \deRCohomology^1_A(M, \AdBundle P) \):
	\begin{equation}
		\lambda \cdot \alpha = \difFracAt{}{\varepsilon}{0} \lambda \cdot (A + \varepsilon \alpha) = \AdAction_\lambda \alpha
	\end{equation}
	for \( \lambda \in \GauGroup_A(P) \) and \( \equivClass{\alpha} \in \deRCohomology^1_A(M, \AdBundle P) \).
	The infinitesimal action of \( \GauAlgebra_A(P) \) is given by the Lie bracket so that the associated momentum map is defined by
	\begin{equation}
		\SectionMapAbb{J}_\singularPart: \deRCohomology^1_A(M, \AdBundle P) \to \deRCohomology^0_A(M, \AdBundle P),
		\qquad \equivClass{\alpha} \mapsto \frac{1}{2} \hodgeStar \wedgeLie{\alpha}{\alpha},
	\end{equation}
	where \( \wedgeLieDot \) denotes the wedge product of \( \AdBundle P \)-valued differential forms relative to the Lie bracket.
	Indeed, by \cref{ex:momentumMap:linearCase}, we have
	\begin{equation}\begin{split}
		\kappa\bigl(\SectionMapAbb{J}_\singularPart(\equivClass{\alpha}), \xi \bigr)
			&= \frac{1}{2} \bar{\omega}_A \bigl( \equivClass{\alpha}, \equivClass{\adAction_\xi \alpha} \bigr)
			\\
			&= \frac{1}{2} \int_M \wedgeDual{\alpha}{\commutator{\xi}{\alpha}}
			\\
			&= - \frac{1}{2} \int_M \dualPair{\wedgeLie{\alpha}{\alpha}}{\xi}
			\\
			&= \frac{1}{2} \kappa\bigl(\hodgeStar \wedgeLie{\alpha}{\alpha}, \xi \bigr).
	\end{split}\end{equation}
	Now,~\eqref{eq:symplecticReduction:orbitTypeManifold:locallyMmu} entails that, locally, the moduli space of central Yang--Mills connections is modeled on
	\begin{equation}
		\SectionMapAbb{J}_\singularPart^{-1}(0) \slash \GauGroup_A(P).
	\end{equation}
	Such a local description of \( \check{\ConnSpace}_\xi (P) \) has already been obtained by \textcite[Theorem~2.32]{Huebschmann1995} in the framework of Sobolev spaces.
\end{remark}

Recall from \cref{ex::diffOneFormsOnSurface:symplecticReduction} that, for \( G = \UGroup(1) \), the moduli space of flat connections is identified with the space of group homomorphisms from \( \pi_1(M) \) to \( \UGroup(1) \).
In the non-abelian case, \textcite{AtiyahBott1983} give a bijection using Wilson loops between the moduli space of central Yang--Mills connections and similar spaces of group homomorphisms.
This bijection is in fact an isomorphism of stratified spaces which, on each stratum, restricts to a diffeomorphism onto the corresponding stratum of the target space, see \parencite[Theorem~7.1]{DiezHuebschmann2017}.
Moreover, using \parencite[Theorem~5.2]{DiezHuebschmann2017}, one can show that these strata-wise diffeomorphisms are compatible with the symplectic structures of both spaces, where the symplectic structure on the target space is induced by the extended moduli space construction of \parencite{HuebschmannJeffrey1994,Huebschmann1993,Jeffrey1994}.
Thus, in summary, the Wilson loop map is an isomorphism of stratified symplectic spaces. 

\subsection{Teichmüller space}
\label{sec:applications:teichmuellerSpace}

Let \( M \) be a closed surface with genus \( \geq 2 \) endowed with a volume form \( \mu \).
A Riemannian metric $g$ on $M$ is said to be compatible with $\mu$ if the volume form \( \mu_g \) induced by \( g \) coincides with \( \mu \).
The space \( \MetricSpace_\mu \) of Riemannian metrics compatible with \( \mu \) 
is identified with the space of sections of the associated bundle 
\(\FrameBundle M \times_{\SLGroup}\bigl(\SLGroup(2,\R)\slash \SOGroup(2)\bigr)\), 
where \( \FrameBundle M \) denotes the \( \SLGroup(2, \R) \)-frame bundle 
induced by the volume form \( \mu \). 
As such, \( \MetricSpace_\mu \) naturally 
has the structure of an infinite-dimensional Fr\'echet manifold.
Using the Iwasawa decomposition of \( \SLGroup(2, \R) \), the space \( \SymTensorFieldSpace^2_g(M) \) of \( g \)-trace-free symmetric covariant \( 2 \)-tensors may be identified with 
the tangent space to $ \MetricSpace_\mu $ at \( g \).
On the space \( \MetricSpace_\mu \), we define the symplectic form (see \eg \parencite[Equation~4.79]{DiezRatiuAutomorphisms})
\begin{equation}\label{eq:applications:coadSL:symplecticFormOnMetrics}
\Omega_g (h, k) = - \frac{1}{2}\int_M
\tensor{h}{^i_j} \, 
\tensor{\mu}{^j_l} \, \tensor{k}{^l_i} \, \mu \,,
\end{equation}
where \( g \in \MetricSpace_\mu \) and 
\( h, k \in \TBundle_g \MetricSpace_\mu \).
Since the upper half plane \( \SLGroup(2, \R) \slash \SOGroup(2) \) is contractible, \( \MetricSpace_\mu \) is contractible as well.
Hence, the symplectic form \( \Omega \) is exact.
Moreover, \( \MetricSpace_\mu \) carries the natural (weak) Riemannian metric
\begin{equation}
	\scalarProd{h}{k}_g = \frac{1}{2} \int_M \tensor{h}{^i_j} \, \tensor{k}{^j_i} \, \mu \,,
\end{equation}
which is compatible with \( \Omega \).

The left action of the identity component \( \DiffGroup_{\mu,0} (M) \) of the group of volume-preserving diffeomorphisms on \( \MetricSpace_\mu \) by push-forward preserves \( \Omega \).
Based on \parencite{DiezRatiuAutomorphisms,DiezRatiuSymplecticConnections}, a group-valued momentum map for this action is given as follows.
Let \( j \) be the almost complex structure on \( M \) defined by the relation \( \mu(\cdot, j \cdot) = g \). The complex 
line bundle \( \ExtBundle^{1,0} M \) of holomorphic forms is the canonical line bundle. 
The associated 
Hermitian frame bundle \( \KBundle_g M \) is a principal circle bundle, called the canonical circle bundle.
The bundle \( \KBundle_g M \) is naturally equipped with the Chern connection, that is, the unique connection compatible with the Hermitian metric and the Dolbeault operator (see \eg \parencite[Remark~3.4]{ChenZhang2023}).
Since every almost complex structure on a surface is integrable, the Chern connection coincides with (the connection induced by) the Levi-Civita connection of \( g \).

Let \( \csCohomology^2(M, \UGroup(1)) \) be the Abelian group of gauge equivalence 
classes of circle bundles with connection over \( M \).
Equivalently, \( \csCohomology^2(M, \UGroup(1)) \) is the group of Cheeger-Simons differential character classes of degree \( 2 \), see \parencite{CheegerSimons1985,BaerBecker2013}.
On the level of line bundles, the group operation is given by the tensor product.
To every (equivalence class of a) circle bundle with connection, we associate the curvature form and the first Chern class.
These constructions yield the following two exact sequences
\begin{equationcd}
	0 \to[r] &
	\sCohomology^1(M, \UGroup(1)) \to[r, thick] &
	\csCohomology^2(M, \UGroup(1))
	\to[r, thick, "\curv"] &
	\clZDiffFormSpace^2(M, \R) \to[r] &
	0
\end{equationcd}
and
\begin{equationcd}
	0 \to[r]
		& \DiffFormSpace^1(M) \slash \clZDiffFormSpace^1(M) \to[r]
		& \csCohomology^2(M, \UGroup(1)) \to[r, "c"] 
		& \sCohomology^2(M, \Z) \to[r]
		& 0,
\end{equationcd}
where \( \clZDiffFormSpace^2(M, \R) \) denotes the space of closed \( 2 \)-forms with integral periods.
The second sequence shows that \( \csCohomology^2(M, \UGroup(1)) \) is an Abelian Fréchet Lie 
group with Lie algebra \( \DiffFormSpace^1(M) \slash \dif \DiffFormSpace^0(M) \), see \parencite[Appendix~A]{BeckerSchenkelEtAl2014}.
Via the integration pairing, \( \DiffFormSpace^1(M) \slash \dif \DiffFormSpace^0(M) \) is the dual of \( \clDiffFormSpace^1(M, \R) \), which can be identified with the space of volume-preserving vector fields on \( M \).
Accordingly, the group  \( \csCohomology^2(M, \UGroup(1)) \) is dual to \( \DiffGroup_{\mu,0}(M) \).
The coconjugation action is given by the pull-back action of diffeomorphisms\footnotemark{}, see \parencite[Example~2.30]{DiezRatiuAutomorphisms}.
\footnotetext{The adjoint action of a diffeomorphism \( \phi \) on a vector field \( X \) is given by the push-forward \( \phi_* X \).
Hence, the coadjoint action of \( \phi \) on \( \equivClass{\alpha} \in \DiffFormSpace^1(M, \R) \slash \dif \DiffFormSpace^0(M) \) is given by the pull-back \( \equivClass{\phi^* \alpha} \).
This action is integrated by the pull-back of a circle bundle with connection.}

Following \parencite{DiezRatiuAutomorphisms}, the map defined by
\begin{equation}
\label{eq:applications:coadSL:momentumMapForMetrics}
\SectionSpaceAbb{J}: \MetricSpace_\mu \to \csCohomology^2(M, \UGroup(1)), 
\quad g \mapsto \KBundle_g M,
\end{equation}
is a group-valued momentum map equivariant with respect to the action of \( \DiffGroup_{\mu,0}(M) \).

\begin{defn}
	The Teichmüller space \( \SectionSpaceAbb{T} \)  is defined as the space of conformal structures on \( M \) modulo the action of the identity component of the group of all diffeomorphisms.
\end{defn}
Let \( S_g \) denote the scalar curvature of \( g \) and \( \bar{S}_g \) be the average of \( S_g \) with respect to \( \mu_g \).
By the Gauss--Bonnet Theorem, the average of the scalar curvature is a topological invariant:
\begin{equation}
	\bar{S}_g = 4 \pi \, \frac{2 - 2 \, \textnormal{genus}(M)}{\textnormal{vol}_{\mu_g}(M)}.
\end{equation}
In particular, it is independent of the metric \( g \in \MetricSpace_\mu \), and we will simply write \( \bar{S} \) in the following.
The curvature of the Chern connection on the canonical bundle \( \KBundle_g M \) is given by 
\( - S_g \mu \). 
Let \( \csCohomology^2(M, \UGroup(1))_{\bar{S}} = \curv^{-1}(-\bar{S} \mu) \) denote the subset of all classes of circle bundles with connection whose curvature is equal to \( - \bar{S} \mu \).
\begin{prop}
	The Teichmüller space \( \SectionSpaceAbb{T} \) is identified with the quotient
	\begin{equation}
		\label{eq:applications:sympQuotient}
		\SectionSpaceAbb{J}^{-1} \bigl(\csCohomology^2(M, \UGroup(1))_{\bar{S}}\bigr) \slash \DiffGroup_{\mu,0}(M).
		\qedhere
	\end{equation}
\end{prop}
\begin{proof}
The definition of \( \SectionSpaceAbb{J} \) implies that 
\begin{equation}
	\SectionSpaceAbb{J}^{-1} \bigl(\csCohomology^2(M, \UGroup(1))_{\bar{S}}\bigr) = \set[\big]{g \in \MetricSpace_\mu \given S_g = \bar{S}}.
\end{equation}
The proof of the uniformization theorem using the normalized Ricci flow shows that every metric \( g \) on \( M \) can be conformally deformed to a metric \( \tilde{g} \) with a given constant scalar curvature \( \rho < 0 \).
By the Gauss--Bonnet Theorem, we have
\begin{equation}
	\rho \int_M \mu_{\tilde{g}} = 8 \pi \, (1 -\textnormal{genus}(M)) = \bar{S} \int_M \mu,
\end{equation}
so that the \( \mu_{\tilde{g}} \)-area of \( M \) is directly related to the \( \mu \)-area.
In particular, if \( \rho = \bar{S} \), then \( \tilde{g} \) lies in \( \MetricSpace_\mu \).
Accordingly, the space of conformal structures on \( M \) is naturally identified with the subset of \( \MetricSpace_\mu \) of metrics with constant scalar curvature \( \bar{S} \), \cf \parencite[Theorem~7.7]{FischerTromba1984} for a similar result without the constraint of a fixed volume form.
The Moser trick yields the identification \parencite[Theorem~8.19]{Ebin1970}
\begin{equation}
	\MetricSpace_\mu \slash \DiffGroup_\mu(M) \isomorph \set*{g \in \MetricSpace \given \int_M \mu_g = \int_M \mu} \slash \DiffGroup(M).
\end{equation}
Restricting to the identity component of the diffeomorphism group, we obtain the identification of the quotient~\eqref{eq:applications:sympQuotient} with the Teichmüller space.
\end{proof}

Hence, the Teichmüller space may be regarded as a reduction at the subset \( \csCohomology^2(M, \UGroup(1))_{\bar{S}} \).
Below we will show that this reduction is an example of an infinite-dimensional (regular) \emph{symplectic orbit reduction}.
A general theory of symplectic orbit reductions in infinite dimensions is still missing, but our results below provide a first step in this direction for the example of the Teichmüller space.

We proceed in a series of lemmas, that roughly mirror the strategy in finite dimensions as presented in \parencite[Section~6.3]{OrtegaRatiu2003}.
\begin{lemma}
	The action of the group \( \DiffGroup_{\mu,0}(M) \) on \( \csCohomology^2(M, \UGroup(1))_{\bar{S}} \) is infinitesimally transitive.
	Moreover, the prescription
	\begin{equation}
		\Omega_{\bar{S}} \bigl(\equivClass{\alpha}, \equivClass{\beta}\bigr) = \frac{1}{2} \int_M \alpha \wedge \beta, \qquad \equivClass{\alpha}, \equivClass{\beta} \in \sCohomology^1(M, \R),
	\end{equation}
	defines a symplectic form \( \Omega_{\bar{S}} \) on \( \csCohomology^2(M, \UGroup(1))_{\bar{S}} \).
\end{lemma}
\begin{proof}
	Note that two bundles \( (L, \nabla) \) and \( (L', \nabla') \) with the same curvature form \( - \bar{S} \mu \) differ by a flat bundle, \ie, an element of \( \sCohomology^1(M, \UGroup(1)) \).
	Hence, \( \csCohomology^2(M, \UGroup(1))_{\bar{S}} \) is an abelian torsor under \( \sCohomology^1(M, \UGroup(1)) \), and we can identify its tangent space with \( \sCohomology^1(M, \R) \).
	Let \( (L, \nabla) \) be a circle bundle with connection whose curvature is equal to \( - \bar{S} \mu \).
	Then the infinitesimal action of \( \VectorFieldSpace_\mu(M) \) at \( (L, \nabla) \) is the Lie derivative of the connection, or equivalently, the contraction with the curvature form (up to an exact term):
	\begin{equation}
		\VectorFieldSpace_\mu(M) \to \sCohomology^1(M, \R), \quad X \mapsto \equivClass*{- \bar{S} \, X \contr \mu}.
	\end{equation}
	This map is clearly surjective.

	Nondegeneracy of \( \Omega_{\bar{S}} \) follows from pairing \( \equivClass{\alpha} \) with \( \equivClass{\hodgeStar \alpha} \), where \( \hodgeStar \) denotes the Hodge star operator of an auxiliary metric on \( M \).
\end{proof}
The next step, in the general procedure, is to show that the inverse image of the orbit under the momentum map is a submanifold.
This usually follows from the bifurcation lemma in combination with a transversality argument.
In the present case, it is more convenient to give a direct proof.
\begin{lemma}
	The subset \( \SectionSpaceAbb{J}^{-1} \bigl(\csCohomology^2(M, \UGroup(1))_{\bar{S}}\bigr) \) is a submanifold of \( \MetricSpace_\mu \).
\end{lemma}
\begin{proof}
	Clearly, \( \SectionSpaceAbb{J}^{-1} \bigl(\csCohomology^2(M, \UGroup(1))_{\bar{S}}\bigr) = \set[\big]{g \in \MetricSpace_\mu \given S_g = \bar{S}} \) is the zero set of the smooth function
	\begin{equation}
		\SectionSpaceAbb{S}: \MetricSpace_\mu \to \sFunctionSpace_0(M), \quad g \mapsto S_g - \bar{S},
	\end{equation}
	where \( \sFunctionSpace_0(M) \) is the space of smooth functions with zero average.
	It is a well-known fact, going back to the work of \textcite{FischerMarsden1975}, that the scalar curvature furnishes a submersion on the space of all Riemannian metrics.
	For the function \( \SectionSpaceAbb{S} \), we fix the total volume form \( \mu \) of the metric but at the same time restrict the target space to functions with zero average.
	Thus, the same argument shows that \( \SectionSpaceAbb{S} \) is a submersion, and the claim follows from the level set theorem for Fréchet manifolds \parencite{Hamilton1982,DiezRudolphKuranishi}, combined with the elliptic regularity theory for the scalar curvature operator.
\end{proof}
Alternatively, the issues coming from the infinite-dimensional setting can be resolved by working in the Sobolev chain (as in \parencite{FischerMarsden1975}).

The action of \( \DiffGroup(M) \) on the space \( \MetricSpace(M) \) of all Riemannian metrics is proper and admits a slice at every point as shown in the classical work of \textcite{Ebin1970}.
A similar reasoning using the explicit formula \parencite{FreedGroisser1989} for the geodesics in the space of Riemannian metrics shows that the action of \( \DiffGroup_{\mu,0}(M) \) on \( \MetricSpace_\mu(M) \) admits a slice at every point as well.
Alternatively, one could construct the slice using the general slice theorem for Fréchet Lie group actions \parencite{DiezSlice}.
Moreover, the action on metrics with constant scalar curvature \( \bar{S} \) is free, see \parencite[Theorem~7.10]{FischerTromba1984}.
As a consequence, the quotient~\eqref{eq:applications:sympQuotient} of \( \SectionSpaceAbb{J}^{-1} \bigl(\csCohomology^2(M, \UGroup(1))_{\bar{S}}\bigr) \) by the action of \( \DiffGroup_{\mu,0}(M) \) is a smooth Fréchet manifold.
To summarize, we have the following.
\begin{thm}
	The manifold \( \SectionSpaceAbb{T} \isomorph \SectionSpaceAbb{J}^{-1} \bigl(\csCohomology^2(M, \UGroup(1))_{\bar{S}}\bigr) \slash \DiffGroup_{\mu,0}(M) \) carries a natural symplectic structure \( \check{\Omega} \) uniquely determined by
	\begin{equation}
		\pi^* \check{\Omega} = \Omega - \SectionSpaceAbb{J}^* \Omega_{\bar{S}},
	\end{equation}
	where \( \pi: \SectionSpaceAbb{J}^{-1} \bigl(\csCohomology^2(M, \UGroup(1))_{\bar{S}}\bigr) \to \SectionSpaceAbb{J}^{-1} \bigl(\csCohomology^2(M, \UGroup(1))_{\bar{S}}\bigr) \slash \DiffGroup_{\mu,0}(M) \) is the canonical projection and restriction of the right-hand side to \( \SectionSpaceAbb{J}^{-1} \bigl(\csCohomology^2(M, \UGroup(1))_{\bar{S}}\bigr) \) is understood.
\end{thm}
	%
The expression~\eqref{eq:applications:coadSL:symplecticFormOnMetrics} 
for the symplectic form on the space of Riemannian metrics implies that the 
reduced symplectic form is (proportional to) the Weil--Petersson symplectic form.
We conclude that the Teichmüller space \( \SectionSpaceAbb{T} \) is a symplectic orbit reduction starting from the space of Riemannian metrics compatible with a fixed volume form.

In the sequel, let us also consider the action of the subgroup \( \DiffGroup_{\mu, \ex}(M) \subseteq \DiffGroup_{\mu,0}(M) \) of exact volume-preserving diffeomorphisms, which has an ordinary equivariant momentum map.
By definition, the Lie algebra \( \VectorFieldSpace_{\mu, \ex}(M) \) of \( \DiffGroup_{\mu, \ex}(M) \) is the space of vector fields \( X \) for which \( X \contr \mu \) is exact.
If we identify \( \VectorFieldSpace_{\mu, \ex}(M) \) with smooth functions on \( M \) up to a constant, then the space \( \sFunctionSpace_0(M) \) of smooth functions with zero average is a natural dual of \( \VectorFieldSpace_{\mu, \ex}(M) \).
The scalar curvature furnishes a momentum map \parencite{Donaldson1999,Fujiki1992}
\begin{equation}
	\SectionMapAbb{J}_{\ex}: \MetricSpace_\mu \to \sFunctionSpace_0(M), \quad g \mapsto S_g - \bar{S}_g,
\end{equation}
for the action of the group \( \DiffGroup_{\mu, \ex}(M) \).
The ordinary symplectic reduction with respect to the action of the group of exact diffeomorphisms thus yields the space
\begin{equation}
\SectionSpaceAbb{T}_\ex \equiv \SectionSpaceAbb{J}_{\ex}^{-1}(0) \slash \DiffGroup_{\mu, \ex}(M) = \set[\big]{g \in \MetricSpace_\mu(M) \given S_g = \bar{S}} \slash \DiffGroup_{\mu, \ex}(M).
\end{equation}
Let us verify that the momentum map \( \SectionMapAbb{J}_{\ex} \) can be brought into a MGS normal form using \cref{prop:momentumMap:MGSnormalForm:elliptic} by checking the conditions of this theorem:
\begin{enumerate}
	\item
		The group \( \DiffGroup_{\mu, \ex}(M) \) is a geometric tame Fréchet Lie subgroup of \( \DiffGroup_\mu(M) \) by \parencite[Proposition~3.8]{DiezJanssensNeebVizmannHolPres}.
	\item
		The properness of the action and the existence of slices has been confirmed above.
	\item
		As mentioned above, the action is free on the subset \( \SectionSpaceAbb{J}_{\ex}^{-1}(0) \).
		Hence, the third point is trivially satisfied.
	\item
		Let \( g \in \MetricSpace_\mu \) be a metric.
		Then the infinitesimal action of \( \VectorFieldSpace_{\mu, \ex}(M) \isomorph \sFunctionSpace(M) \slash \R \) on \( \MetricSpace_\mu \) at \( g \) is given by the Lie derivative of the metric
		\begin{equation}
			X_f \ldot g = - \difLie_{X_f} g = 2 \diF_g^* X_f, 
		\end{equation}
		where \( \diF_g^*: \VectorFieldSpace_\mu(M) \to \SymTensorFieldSpace^2_g(M) \) is the adjoint of the covariant divergence operator \( (\diF_g h)_j = \nabla_k \tensor{h}{_j^k} \), see \eg \parencite[Lemma~1.60]{Besse1987} for the second equality.
		Moreover, the linearization of the scalar curvature map at \( g \) is given by (see, \eg, \parencite[Theorem~1.174]{Besse1987})
		\begin{equation}
			\tangent_g S (h) = - \Delta_g (\tr_g h) + \diF_g \diF_g h - \textrm{Ric}_g \cdot h.
		\end{equation}
		If \( h \) is trace-free, then the first and third term vanish. 
		Thus the chain relevant for the MGS normal form is (suppressing some constants for simplicity)
		\begin{equationcd}[label=eq:teichmueller:chain, tikz={column sep=large}]
			0 \to[r] 
				& \sFunctionSpace(M)
					\to[r, "\diF_g^* \circ \mu^\sharp \circ \dif"]
				& \SymTensorFieldSpace^2_g(M)
					\to[r, "\diF_g \circ \diF_g"]
				& \sFunctionSpace(M)
					\to[r]
				& 0.
		\end{equationcd}
		By direct inspection, one verifies that this chain is elliptic.
	\item
		Finally, due to the ellipticity of the chain~\eqref{eq:teichmueller:chain}, the images of each map in the chain are closed orthogonally complemented subspaces of the respective target spaces.
		This shows directly that the image of \( \tangent_g S \) is weakly closed.
		In fact, since the action is free on the subset of metrics with constant scalar curvature the bifurcation lemma implies that \( \tangent_g S \) is surjective.
		This recovers the well-known fact about the linearization of the scalar curvature map \parencite[Theorem~A]{FischerMarsden1975}.
		Moreover, \cref{prop:symplecticFunctionalAnalysis:orthogonalToSymplecticOrthogonal} now implies that the infinitesimal orbit is symplectically closed.
\end{enumerate}
As a direct consequence of the Symplectic Reduction \cref{prop:symplecticReduction:mainTheorem}, the symplectic quotient \( \SectionSpaceAbb{T}_\ex \) is then a smooth symplectic manifold.

The group \( \DiffGroup_{\mu, \ex}(M) \) is a normal subgroup of \( \DiffGroup_{\mu}(M) \), and the finite-dimensional quotient group \( \DiffGroup_{\mu, 0}(M) \slash \DiffGroup_{\mu, \ex}(M) \) acts
on the space \( \SectionSpaceAbb{T}_\ex \).
It turns out that the orbits of this action are symplectic submanifolds of \( \SectionSpaceAbb{T}_\ex \) relative to the reduced symplectic form \parencite[p.~181]{Donaldson2003}.
Hence, the quotient 
\begin{equation}
	\SectionSpaceAbb{T}_\ex \slash \bigl(\DiffGroup_{\mu, 0}(M) \slash \DiffGroup_{\mu, \ex}(M)\bigr) \isomorph \SectionSpaceAbb{T}
\end{equation}
inherits a natural symplectic structure, which by construction coincides with the symplectic structure obtained by orbit reduction.

In summary, we have shown that the Teichmüller space \( \SectionSpaceAbb{T} \) is a symplectic orbit reduction of the space of Riemannian metrics compatible with a fixed volume form, and that the ordinary symplectic reduction with respect to the action of the group of exact diffeomorphisms yields a symplectic manifold \( \SectionSpaceAbb{T}_\ex \) that fibers over \( \SectionSpaceAbb{T} \) with symplectic fibers.

\begin{remark}
	Fix a reference metric \( g_0 \in \MetricSpace_\mu \).
For every \( g \in \MetricSpace_\mu \), the difference of the Chern connections on the canonical bundles \( \KBundle_{g} M \) and \( \KBundle_{g_0} M \) (under an arbitrary isomorphism \( \KBundle_{g} M \isomorph \KBundle_{g_0} M \)) is a \( 1 \)-form \( \widebar{J}(g_0, g) \in \DiffFormSpace^{1}(M) \), well-defined up to an exact form.
Then the map
\begin{equation}
	\widebar{\SectionMapAbb{J}}: \MetricSpace_\mu 
	\to \DiffFormSpace^{1}(M) \slash \dif \DiffFormSpace^{0}(M), \qquad g \mapsto \widebar{J}(g_0, g)
\end{equation}
is a momentum map for the action of \( \DiffGroup_\mu(M) \) on \( \MetricSpace_\mu \).
Here, \( \DiffFormSpace^{1}(M) \slash \dif \DiffFormSpace^{0}(M) \) is considered as the dual of the Lie algebra \( \VectorFieldSpace_\mu(M) \) of volume-preserving vector fields with respect to the integration pairing under the isomorphism \( \VectorFieldSpace_\mu(M) \isomorph \clDiffFormSpace^1(M) \) induced by the volume form \( \mu \).
Due to the need to fix a reference metric, the map \( \widebar{\SectionMapAbb{J}} \) is not equivariant.
It would be interesting to investigate the symplectic reduction relative to the non-equivariant momentum map \( \widebar{\SectionMapAbb{J}} \) and its relation to the Teichmüller space.
If \( \SectionMapAbb{J} \) were an ordinary momentum map, then the symplectic point reduction would be the same as the symplectic orbit reduction \parencite[Theorem~4.6.1]{OrtegaRatiu2003}, but a similar result for group-valued momentum maps is not yet known.
\end{remark}

\appendix
\section{Slices and Stratification}
\label{sec:slices}

Since slices play a fundamental role in the construction of the MGS normal form, for the convenience of the reader, we now recall the definition of a slice in infinite dimensions, \cf \parencite{DiezSlice}.   
\begin{defn}
	\label{defn:slice:slice}
	Let \( M \) be a \( G \)-manifold.
	A \emphDef{slice} at \( m \in M \) is a submanifold \( S \subseteq M \) containing \( m \) with the following properties:
	\begin{thmenumerate}[label=(SL\arabic*), ref=(SL\arabic*), leftmargin=*] 
		\item \label{i::slice:SliceDefSliceInvariantUnderStab}
			The submanifold \( S \) is invariant under the induced action of the stabilizer subgroup \( G_m \), that is, \( G_m \cdot S \subseteq S \).

		\item \label{i::slice:SliceDefOnlyStabNotMoveSlice}
			Any \( g \in G \) with \( (g \cdot S) \cap S \neq \emptyset \) is necessarily an element of \( G_m \). 

		\item \label{i::slice:SliceDefLocallyProduct}
			The stabilizer \( G_m \) is a principal Lie subgroup of \( G \) and the principal bundle \( G \to G \slash G_m \) admits a local section \( \chi: G \slash G_m \supseteq W \to G \) defined on an open neighborhood \( W \) of the identity coset \( \equivClass{e} \) in such a way that the map
			\begin{equation}
				\chi^S: W \times S \to M, \qquad (\equivClass{g}, s) \mapsto \chi(\equivClass{g}) \cdot s
			\end{equation}
			is a diffeomorphism onto an open neighborhood of \( m \), which is called a \emphDef{slice neighborhood} of \( m \).

		\item \label{i::slice:SliceDefPartialSliceSubmanifold}
			The partial slice \( S_{(G_m)} = \set{s \in S \given G_s \text{ is conjugate to } G_m} \) is a closed submanifold of \( S \).

		\item
			\label{i:slice:linearSlice}
			There exist a smooth representation of \( G_m \) on a locally convex space \( X \) and a \( G_m \)-equivariant diffeomorphism \( \iota_S \) from a \( G_m \)-invariant open neighborhood of \( 0 \) in \( X \) onto \( S \) such that \( \iota_S(0) = m \).
			\qedhere
	\end{thmenumerate}
\end{defn}

The notion of a slice is closely related to the concept of a tubular neighborhood.
In fact by \parencite[Proposition~2.6.2]{DiezSlice}, for every slice \( S \) at \( m \in M \), the tube map
\begin{equation}
	\chi^\tube: G \times_{G_m} S \to M, \qquad \equivClass{g, s} \mapsto g \cdot s
\end{equation}
is a \( G \)-equivariant diffeomorphism onto an open, \( G \)-invariant neighborhood of \( G \cdot m \) in \( M \).

As in the finite-dimensional case, the existence of slices implies many nice properties of the orbit space.
For example, the subset \( M_{(H)} \) of points having orbit type \( (H) \) is a submanifold of \( M \) and \( \check{M}_{(H)} = M_{(H)} \slash G \) carries a smooth manifold structure such that the natural projection \( \pi_{(H)}: M_{(H)} \to \check{M}_{(H)} \) is a smooth submersion.
If, in addition, a certain approximation property is satisfied, then the orbit type manifolds fit together nicely so that the orbit space is a stratified space, see \parencite[Theorem~4.2]{DiezSlice}.
For completeness, we include here our definition of stratification and refer the reader to \parencite{DiezSlice} for further details.
\begin{defn}
	\label{def:stratification:stratification}
	Let \( X \) be Hausdorff topological space. 
	A partition \( \stratification{Z} \) of \( X \) into subsets \( X_\sigma \) indexed by \( \sigma \in \Sigma \) is called a \emphDef{stratification} of \( X \) if the following conditions are satisfied:
	\begin{thmenumerate}[label=(DS\arabic*), ref=(DS\arabic*), leftmargin=*]
		\item \label{i::stratification:stratumIsManifold} 
			Every piece \( X_\sigma \) is a locally closed, smooth manifold (whose manifold topology coincides with the relative topology).
			We will call \( X_\sigma \) a \emphDef{stratum} of \( X \).
 
		\item \label{i::stratification:frontierCondition} (frontier condition)
			Every pair of disjoint strata \( X_\varsigma \) and \( X_\sigma \) with \( X_\varsigma \cap \closureSet{X_\sigma} \neq \emptyset \) satisfies:
			\begin{thmenumerate}[label=\alph*), ref=(DS2\alph*)]
				\item \label{i:stratification:frontierConditionBoundary}
					\( X_\varsigma \) is contained in the frontier \( \closureSet{X_\sigma} \setminus X_\sigma\) of \( X_\sigma \),
				\item \label{i:stratification:frontierConditionIntersection}
					\( X_\sigma \) does not intersect \( \closureSet{X_\varsigma} \).
			\end{thmenumerate}
			In this case, we write \( X_\varsigma < X_\sigma \) or \( \varsigma < \sigma \). 
			\qedhere
	\end{thmenumerate}
\end{defn}

\begin{refcontext}[sorting=nyt]{}
	\printbibliography
\end{refcontext}

\end{document}